\documentclass[11pt]{article}

\usepackage{enumerate,xspace,proof}
\usepackage{amsmath,xspace,amssymb,stmaryrd}
\usepackage[dvips]{graphics}
\usepackage{graphicx}
\usepackage[dvipsnames]{xcolor}
\usepackage{hyperref}
\hypersetup{
    colorlinks,
    citecolor=black,
    filecolor=black,
    linkcolor=MidnightBlue,
    urlcolor=black
}

\usepackage{verbatim}

\usepackage{xy}
\xyoption{all}

\usepackage{tikz-cd}

\title{Latent Fibrations: \\ Fibrations for Categories of Partial Maps}
\author{Robin Cockett, Geoff Cruttwell, Jonathan Gallagher, and Dorette Pronk 
\thanks{Partially supported by NSERC, Canada.} 
}

\newcommand{\x}{\times}
  
\newcommand{\<}{\langle}
\renewcommand{\>}{\rangle}

\newcommand{\E}{\ensuremath{\mathbb E}\xspace}
\newcommand{\A}{\ensuremath{\mathbb A}\xspace}
\newcommand{\B}{\ensuremath{\mathbb B}\xspace}
\newcommand{\F}{\ensuremath{\mathbb F}\xspace}
\newcommand{\X}{\ensuremath{\mathbb X}\xspace}
\newcommand{\Y}{\ensuremath{\mathbb Y}\xspace}

\newcommand{\W}{\ensuremath{\mathbb W}\xspace}
\newcommand{\p}{\ensuremath{\mathsf{p}}}
\newcommand{\q}{\ensuremath{\mathsf{q}}}
\newcommand{\M}{\ensuremath{\mathsf{M}}}

\newcommand{\R}{{\ensuremath{\mathbb R}}\xspace}
\renewcommand{\S}{{\ensuremath{\mathbb S}}\xspace}
\newcommand{\T}{{\ensuremath{\mathbb T}}\xspace}

\renewcommand{\O}{{\ensuremath{\mathcal O}}\xspace}

\newcommand{\rst}[1]{\overline{#1}}
\newcommand{\rs}[1]{\overline{#1}}

\renewcommand{\split}{{\sf Split}_r}

\newcommand{\idb}{1_{\bullet}}
\newcommand{\bu}{\bullet}

\newcommand{\tidle}[1]{\tilde{#1}}

\newtheorem{theorem}{Theorem}[section]    

\newtheorem{corollary}[theorem]{Corollary}   
\newtheorem{lemma}[theorem]{Lemma}   
\newtheorem{preremark}[theorem]{Remark}   
\newtheorem{prexample}[theorem]{Example}   
\newtheorem{proposition}[theorem]{Proposition}

\newtheorem{definition}[theorem]{Definition}

\newenvironment{example}{\begin{prexample}\rm}{\end{prexample}}

\makeatletter

\newdimen\w@dth

\def\setw@dth#1#2{\setbox\z@\hbox{\scriptsize $#1$}\w@dth=\wd\z@
\setbox\@ne\hbox{\scriptsize $#2$}\ifnum\w@dth<\wd\@ne \w@dth=\wd\@ne \fi
\advance\w@dth by 1.2em}

\def\t@^#1_#2{\allowbreak\def\n@one{#1}\def\n@two{#2}\mathrel
{\setw@dth{#1}{#2}
\mathop{\hbox to \w@dth{\rightarrowfill}}\limits
\ifx\n@one\empty\else ^{\box\z@}\fi
\ifx\n@two\empty\else _{\box\@ne}\fi}}
\def\t@@^#1{\@ifnextchar_ {\t@^{#1}}{\t@^{#1}_{}}}

\def\t@left^#1_#2{\def\n@one{#1}\def\n@two{#2}\mathrel{\setw@dth{#1}{#2}
\mathop{\hbox to \w@dth{\leftarrowfill}}\limits
\ifx\n@one\empty\else ^{\box\z@}\fi
\ifx\n@two\empty\else _{\box\@ne}\fi}}
\def\t@@left^#1{\@ifnextchar_ {\t@left^{#1}}{\t@left^{#1}_{}}}

\def\two@^#1_#2{\def\n@one{#1}\def\n@two{#2}\mathrel{\setw@dth{#1}{#2}
\mathop{\vcenter{\hbox to \w@dth{\rightarrowfill}\kern-1.7ex
                 \hbox to \w@dth{\rightarrowfill}}%
       }\limits
\ifx\n@one\empty\else ^{\box\z@}\fi
\ifx\n@two\empty\else _{\box\@ne}\fi}}
\def\tw@@^#1{\@ifnextchar_ {\two@^{#1}}{\two@^{#1}_{}}}

\def\tofr@^#1_#2{\def\n@one{#1}\def\n@two{#2}\mathrel{\setw@dth{#1}{#2}
\mathop{\vcenter{\hbox to \w@dth{\rightarrowfill}\kern-1.7ex
                 \hbox to \w@dth{\leftarrowfill}}%
       }\limits
\ifx\n@one\empty\else ^{\box\z@}\fi
\ifx\n@two\empty\else _{\box\@ne}\fi}}
\def\t@fr@^#1{\@ifnextchar_ {\tofr@^{#1}}{\tofr@^{#1}_{}}}

\newdimen\W@dth
\def\setW@dth#1#2{\setbox\z@\hbox{$#1$}\W@dth=\wd\z@
\setbox\@ne\hbox{$#2$}\ifnum\W@dth<\wd\@ne \W@dth=\wd\@ne \fi
\advance\W@dth by 1.2em}

\def\T@^#1_#2{\allowbreak\def\N@one{#1}\def\N@two{#2}\mathrel
{\setW@dth{#1}{#2}
\mathop{\hbox to \W@dth{\rightarrowfill}}\limits
\ifx\N@one\empty\else ^{\box\z@}\fi
\ifx\N@two\empty\else _{\box\@ne}\fi}}
\def\T@@^#1{\@ifnextchar_ {\T@^{#1}}{\T@^{#1}_{}}}

\def\T@left^#1_#2{\def\N@one{#1}\def\N@two{#2}\mathrel{\setW@dth{#1}{#2}
\mathop{\hbox to \W@dth{\leftarrowfill}}\limits
\ifx\N@one\empty\else ^{\box\z@}\fi
\ifx\N@two\empty\else _{\box\@ne}\fi}}
\def\T@@left^#1{\@ifnextchar_ {\T@left^{#1}}{\T@left^{#1}_{}}}

\def\Tofr@^#1_#2{\def\N@one{#1}\def\N@two{#2}\mathrel{\setW@dth{#1}{#2}
\mathop{\vcenter{\hbox to \W@dth{\rightarrowfill}\kern-1.7ex
                 \hbox to \W@dth{\leftarrowfill}}%
       }\limits
\ifx\N@one\empty\else ^{\box\z@}\fi
\ifx\N@two\empty\else _{\box\@ne}\fi}}
\def\T@fr@^#1{\@ifnextchar_ {\Tofr@^{#1}}{\Tofr@^{#1}_{}}}

\def\Two@^#1_#2{\def\N@one{#1}\def\N@two{#2}\mathrel{\setW@dth{#1}{#2}
\mathop{\vcenter{\hbox to \W@dth{\rightarrowfill}\kern-1.7ex
                 \hbox to \W@dth{\rightarrowfill}}%
       }\limits
\ifx\N@one\empty\else ^{\box\z@}\fi
\ifx\N@two\empty\else _{\box\@ne}\fi}}
\def\Tw@@^#1{\@ifnextchar_ {\Two@^{#1}}{\Two@^{#1}_{}}}

\def\to{\@ifnextchar^ {\t@@}{\t@@^{}}}
\def\from{\@ifnextchar^ {\t@@left}{\t@@left^{}}}
\def\tofro{\@ifnextchar^ {\t@fr@}{\t@fr@^{}}}
\def\To{\@ifnextchar^ {\T@@}{\T@@^{}}}
\def\From{\@ifnextchar^ {\T@@left}{\T@@left^{}}}
\def\Two{\@ifnextchar^ {\Tw@@}{\Tw@@^{}}}
\def\Tofro{\@ifnextchar^ {\T@fr@}{\T@fr@^{}}}

\makeatother

\begin{document}
\maketitle

\begin{abstract}
Latent fibrations are an adaptation, appropriate for categories of partial maps (as presented by restriction categories), of the usual notion of fibration.  The paper initiates the development of  the basic theory of latent fibrations and explores some key examples.   Latent fibrations cover a wide variety of examples, some of which are partial versions of standard fibrations, and some of which are particular to partial map categories (particularly those that arise in computational settings).  Latent fibrations with various special properties are identified: hyperconnected latent fibrations, in particular, are shown to support the construction of a fibrational dual -- important to reverse differential programming and, more generally, in the theory of lenses.   
\end{abstract}

\tableofcontents


\section{Introduction}

Fibrations play a fundamental role in categorical logic: in particular, they provide models of type theories \cite{book:Jacobs-Cat-Log} used in the semantics of logical and computational systems.  The  purpose of this paper is to develop the analog of fibrations for settings -- such as those describing computation -- which are based on partial maps.   Because categories of partial maps admit an abstract and completely algebraic description as {\em restriction categories\/},  it is more general and, indeed, more convenient to develop a theory of fibrations for restriction categories.   As the structure of a restriction category is more nuanced than that of a mere category -- due to the necessity to support the partiality of maps -- the restriction analogue of a fibration, which is called a \emph{latent fibration}, is a necessarily more subtle notion.  The purpose of this paper is to initiate a careful development of the theory of latent fibrations.

There are many reasons why an abstract theory of fibrations, 
into which partiality of maps has been baked, might be useful.  Recalling that computation is fundamentally partial, 
there is an obvious motive to consider such settings in the semantics of computation and, in particular, to have partiality built into type theories describing computation.  Notably the first 
use of latent fibrations was in Chad Nester's MSc.~thesis \cite{msc:nester-calgary}, where they were used in the study of realizability.  More recently the desire to understand the semantics of differential programming, which has received increased attention due to its connection to machine learning, has further stimulated the development of a general theory of partiality in fibrations.

In differential programming one needs to calculate the derivative of partially-defined smooth functions (that is, smooth functions defined on some open subset of their domain).  The  categorical structure of the so-called (total) ``forward'' derivative was developed in \cite{journal:BCS:CDC}: this characterized the operation which takes a smooth map $f\colon \R^n \to R^m$ and 
produces the map
	\[ D[f]\colon R^n \times R^n \to R^m \]
whose value at a pair $(x,v)$ is $J(f)(x)\cdot v$: the Jacobian of $f$, at $x$, in the direction $v$.  The structure of differentials for partially defined smooth functions is described in  \cite{journal:diff-rest} which introduced  differential {\em restriction\/} categories. In this case, the structure is axiomatized by an operation which still takes a map $f\colon A \to B$ and produces a map $D[f]\colon A \times A \to B$, but now with additional axioms relating the partiality of $D[f]$ to that of $f$: in particular, one asks that $\rst{D[f]} = \rst{f} \times 1$; 
that is, the partiality of $D[f]$ is entirely determined by the partiality of $f$.  This structure, while being interesting in its own right, is also of key importance in understanding how one can build differential manifolds at this level of generality \cite[Section 6]{journal:TangentCats}. 

Differentials are tightly linked to fibrations because the derivative, both in the total and the restriction case, can be regarded as a section of the ``simple fibration'' over $\X$.  That is, the category whose objects are objects in a context $\Sigma$, expressed as pairs $(\Sigma,A)$, with a map from $(\Sigma,A)$ to $(\Sigma',B)$ consisting of a pair of maps
	\[ f\colon \Sigma \to \Sigma' \mbox{ and } g\colon \Sigma \times A \to B \]
where the second component $g$ uses the context.  Then given a map $f\colon \Sigma \to \Sigma'$ in the base category, the pair $(f,D[f])$ gives a map in the simple slice category, and the functoriality of this operation is precisely the chain rule.

In the restriction case, the construction of the ordinary simple fibration does not inherit a restriction structure; however, a restriction structure 
is inherited when we require pairs $(f,g)$ such that $\rst{g} = \rst{f} \times 1$.  This new construction does not provide an ordinary fibration over 
$\X$, and there is no way to obtain an ordinary fibration.  However, what it does provide is a \emph{latent fibration}, which is, of course, the 
main subject of this paper.  Intriguingly, sections of these latent simple fibrations recapture the additional axioms of differential 
restriction categories because the equation $\rs{D[f]} = \rs{f} \times 1$ is then forced.

There is an important addendum to this story: the backpropogation algorithm, which is widely used in machine learning,  is based on computing the \emph{reverse} differential \cite{arxiv:RDC} as computationally this can be much more efficient.  The reverse derivative takes a smooth map $f\colon \R^n \to R^m$ and produces the map
	\[ R[f]\colon R^n \times R^m \to R^n \]
whose value at a pair $(x,v)$ is $J^T(f)(x)\cdot v$, the \emph{transpose} of the Jacobian of $f$, at $x$, in the direction $v$ (note its difference in type from the forward derivative $D[f]\colon R^n \times R^n \to R^m$).  Understanding the abstract properties of the reverse differential allows one to apply these machine learning techniques to different settings:  for example, they have already been used to develop machine-learning algorithms for Boolean circuits \cite{learn-bool}.

When we look at the fibrational structure of the reverse derivative operation, it can also be seen as giving a section to a fibration, but instead of being a section of the simple slice over $\X$, now it is a section of the \emph{fibrational dual\/} of the simple slice over $\X$.  The fibrational dual can be defined for any ordinary fibration by taking the opposite category in each fibre.  For the simple fibration, its dual is the category whose objects are pairs $(A,A')$ as before, but now a map from $(A,A')$ to $(B,B')$ consists of a pair of maps
	\[ f\colon A \to B, f^*\colon A \times B' \to A'. \]
Fibrational duals have also received renewed interest due to their role to the theory of lenses.  Lenses were originally developed for database theory \cite{journal:lenses-rosebrugh-johnson} but are now used more widely in learning systems, and elsewhere \cite{hedges2017coherence,spivak2019generalized}. The research in this paper opens the door to the formal study of partial lenses.

Thus, not only is it useful for the theory of reverse {\em restriction} differential categories to be able to form the fibrational dual of a latent fibration, but also, more generally,  in the theory of partial lenses.   However, in general, it is simply not the case that a latent fibration will have a fibrational dual! This is basically because the opposite of a restriction category is not generally a restriction category.  Thus, taking the ``opposite in each fibre'' of a latent fibration will not necessarily produce a latent fibration.  However, in the particular cases of interest, which use the simple latent fibration \ref{simple-latent-fibration} and the codomain latent fibration \ref{codomain-latent-fibration}, it is clear that it {\em is\/} possible to define a suitable fibrational dual.  Thus, one of the goals of this paper is to develop the theory of latent fibrations sufficiently so that the circumstances under which it is possible to define the fibrational dual of a latent fibration is fully understood.  

Of course, aside from the aforementioned motivations, latent fibrations are of intrinsic mathematical interest in their own right.  
The definition of a latent fibration (Definition \ref{newdefn}) involves a subtle change of the normal notion of Cartesian map to what, in order to clearly distinguish the notion, we call a {\em prone map\/}.   Latent fibrations first appeared in \cite{msc:nester-calgary} and were defined in a more complicated -- albeit equivalent -- manner:  see Definition \ref{olddefn} in Appendix \ref{Appendix-A}.   It is reasonable to wonder whether these definitions are ad hoc.  A theoretically convincing argument that the notion has a solid basis can be found in Appendix \ref{Appendix-B} where it is  shown that the notion of latent  fibration corresponds precisely to the 2-categorical notion introduced by Street \cite{street_fibration} for the (carefully chosen) 2-category of restriction semi-functors and transformations. 

While many results about ordinary fibrations are true of latent fibrations, there are some subtle aspects of the theory.  For example, while prone arrows always compose and isomorphisms are always prone, one might expect that partial isomorphisms should be prone as these generalize isomorphisms in a partial setting.  However, they are not in general.   
This failure, in fact, leads to the investigation of important additional properties that a latent fibration can satisfy (see Section \ref{sec:types}).   The first additional property we study is that of being an {\em admissible\/}  latent fibration (see Section \ref{sec:admissible}): in admissible latent fibrations restriction idempotents have prone liftings which are restriction idempotents.  An important consequence of being admissible is that one can split 
the idempotents to obtain an $r$-split latent fibration: these then can be linked to fibrations of partial map categories and ${\sf M}$-categories (see Section \ref{sec:fibrations-of-partial}).  The next additional property we consider is that of being {\em separated\/} (see Section \ref{sec:separated}), which is the requirement that the projection functor separates restriction idempotents.  This turns out to be equivalent to asking that all restriction idempotents in the total category be prone.  When both these conditions hold the latent fibration is a {\em hyperfibration} (see Section \ref{sec:connected}), and this is in turn equivalent to asking that the projection functor of the latent fibration be a hyperconnection (\cite[pg. 39]{journal:rcats-enriched}).  We provide separating examples for each of these additional requirements, and develop their properties: in particular, we show that latent hyperfibrations have fibrational duals.  

After developing the basic theory, we turn to obtaining an explicit description of latent fibrations as categories of partial maps (as ${\sf M}$-categories), Section \ref{sec:fibrations-of-partial}.  Indeed, we completely characterize the $r$-split latent fibrations of partial map categories, showing that they are equivalent to giving a fibration between the total categories which is ``$M$-plentiful'' -- this is a requirement, that $M$-maps lift in an appropriate manner.  Finally, in the last section, Section \ref{sec:dual}, we describe how to define the fibrational dual of a latent hyperfibration.

The authors are grateful for Bob Rosebrugh's many contributions to the category theory research community, through his interesting and enlightening talks and papers,  his many years of service with TAC, and, on a personal level, his discussions with each of us.  We hope that the present paper may serve as a continuation of his work on lenses and fibrations (e.g., \cite{journal:lenses-rosebrugh-johnson}) into settings which involve partial maps.  


\section{Restriction Categories}

In this section we give a brief introduction to restriction categories concentrating on some of the less well-known aspects that are relevant to this paper.  
For further details see \cite{journal:rcats1, cockett_lack_2007}.


\subsection{Restriction categories}

A {\bf restriction category} (see \cite{journal:rcats1} for details) is a category equipped with a {\bf restriction} combinator which given a map $f\colon A \to B$, returns 
an endomorphism on the domain $\rst{f}\colon A \to A$ which satisfies just four identities\footnote{Composition is written in diagrammatic order in this paper.}:
\[ \mbox{[R.1]}~\rst{f}f = f ~~~~ \mbox{[R.2]}~\rst{f}~\rst{g} = \rst{g}~\rst{f} ~~~~\mbox{[R.3]}~\rst{f}~\rst{g} = \rst{\rst{f} g} ~~~~ \mbox{[R.4]}~f \rst{g} = \rst{fg}f \] 
The prototypical restriction category is the category of sets and partial maps, ${\sf Par}$.  The restriction of a partial map in ${\sf Par}$ is the partial identity on the domain which is 
defined precisely when the partial map is defined. 

 In any restriction category $\rst{f}$ is always an idempotent and any idempotent $e = ee$ with $e=\rst{e}$ is called a {\bf restriction idempotent}. It is not the case that every idempotent need be a  restriction idempotent.  The restriction idempotents on an object $A$ form a meet semi-lattice, with the meet given by composition, which we denote by ${\cal O}(A)$:  the elements of ${\cal O}(A)$ may be regarded as distinguished predicates on the object $A$.

Restriction categories are always full subcategories of partial map categories (see \cite[Proposition 3.3]{journal:rcats1}) and this means that parallel maps in a restriction 
category can be partially ordered: the partial order is defined using the restriction by $f \leq g$ if and only if $\rst{f}g = f$ and is called the {\bf restriction order}, and in fact gives an enrichment.  

A map $f\colon A \to B$ in a restriction category $\X$ is said to be {\bf total} in case $\rst{f} = 1_A$.  Total maps compose and include identities and, thus, form a (non-full) subcategory denoted 
${\sf Total}(\X)$.  Any category can be endowed with a trivial restriction which takes each map to the identity on its domain: thus, every category occurs as the total maps of some restriction category.

A map $s\colon A \to B$ in a restriction category is a {\bf partial isomorphism} if there is a map $s^{(-1)}\colon B \to A$ -- the {\bf partial inverse} of $s$ -- with 
$s s^{(-1)} = \rst{s}$ and $s^{(-1)}s = \rst{s^{(-1)}}$.  The partial inverse of a map is unique.  Partial isomorphism include all the restriction idempotents 
and are closed to composition.  A restriction category in which all the maps are partial isomorphisms is called an {\bf inverse category\/}.  Inverse categories 
are to restriction categories what groupoids are to ordinary categories.


\subsection{\texorpdfstring{${\sf M}$}{M}-categories and \texorpdfstring{$r$}{r}-split restriction categories}

A restriction category is {\bf $r$-split} if all its restriction idempotents split.   Given an arbitrary restriction category, $\X$, one may always split its restriction idempotents to obtain 
${\sf Split}_r(\X)$, an $r$-split restriction category.  The 2-category of $r$-split restriction categories, restriction functors, and total natural transformations is 2-equivalent to the 2-category of ${\sf M}$-categories \cite{journal:rcats1}.  ${\sf M}$-categories are categories with a system of monics which is closed to composition and pullbacks along any map.  Functors between ${\sf M}$-categories must not only preserve the ${\sf M}$-maps but also the pullbacks along ${\sf M}$-maps.  Natural transformations between ${\sf M}$-functors are natural transformations which, in addition, are Cartesian (or tight) for transformations between ${\sf M}$-maps -- that is the naturality squares for ${\sf M}$-maps are pullbacks.  The 2-equivalence is given on the one hand, by moving from an ${\sf M}$-category $(\X,{\sf M}_\X)$ to its partial map category ${\sf Par}(\X,{\sf M}_\X)$ and on the other hand by moving to the ${\sf M}$-category consisting of the total map category of the $r$-split restriction category, ${\sf Total}(\E)$, with the restriction monics, ${\sf Monic}(\E)$, $({\sf Total}(\E),{\sf Monic}(\E))$.  The restriction monics can be variously described as partial isomorphisms which are total, restriction sections, or, more interestingly, as left adjoints with respect to the partial order enrichment.


\subsection{Restriction (semi)functors and transformations}

There are various sorts of morphisms between restriction categories which can be considered: the most basic is that of a {\bf restriction functor}, $F\colon \X \to \Y$, which is a functor between the categories which in addition preserves the restriction structure, that is $F(\rst{f}) = \rst{F(f)}$.  Restriction functors so defined preserve total maps, restriction idempotents, and partial isomorphisms.  As we shall see, the re-indexing or substitution functors for a latent fibration generally only satisfy the conditions for  the slightly weaker notion of a restriction \emph{semi}functor; thus, it will be important to briefly review these and the related notion of a transformation between restriction semifunctors.  

\begin{definition} Let $\X$ and $\Y$ be restriction categories.
\begin{itemize}
	\item A \textbf{restriction semifunctor} $F$ from $\X$ to $\Y$ is a semifunctor from $\X \to \Y$ (that is, a map on objects and arrows which preserves domains, codomains, and composition, but not necessarily identities) which preserves restrictions.  Note that for a restriction semifunctor, while $F(1_X)$ is not the identity, it is still a restriction idempotent since
		\[ \rs{F(1_X)} = F(\rs{1_X}) = F(1_X). \]
	\item If $F,G\colon \X \to \Y$ are restriction semifunctors, a \textbf{restriction transformation} $\alpha\colon F \Rightarrow G$ is a natural transformation from $F$ to $G$ such that for each $X \in \X$, $\rs{\alpha_X} = F(1_X)$.  
\end{itemize}
\end{definition}

Restriction functors and transformations organize themselves into a 2-category.  Similarly, restriction semifunctors and their transformations organize themselves into a 2-category, which we will denote ${\sf SRest}$.


\subsection{Precise diagrams}

In a restriction category, we shall call a commuting diagram {\em precise} in case the restriction of the overall map that the diagram describes is equal to the restriction of all the maps leaving the start node.   Precise commuting triangles play a key role in the definition of latent fibrations. 

\begin{definition}\label{defn:precise}
In a restriction category, a commuting triangle
\[ \xymatrix{ A \ar[dr]^g \ar[d]_k \\ B \ar[r]_f & C} \] 
is {\bf precise} in case $\rst{g}= \rst{k}$.  We refer to $k$  as the {\bf left factor} (and $f$ the right factor) of the triangle.
\end{definition}
  The following are some useful observations on precise triangles:

\begin{lemma}  
\label{Jaws}
In any restriction category:
\begin{enumerate}[(i)]
\item A commuting triangle 
\[ \xymatrix{ A \ar[dr]^g \ar[d]_k \\ B \ar[r]_f & C} \]
is precise if and only if $k\rst{f} = k$.
\item A commuting triangle with right factor a restriction idempotent
\[ \xymatrix{ A \ar[dr]^g \ar[d]_k \\ B \ar[r]_{e = \rst{e}} & C} \]
is precise if and only if $g=k$.
\item A commuting triangle with right factor a partial isomorphism
\[ \xymatrix{ A \ar[dr]^g \ar[d]_k \\ B \ar[r]_{\alpha} & C} \]
is precise if and only if $k = g \alpha^{(-1)}$.
\item The commuting triangle 
\[ \xymatrix{A \ar[d]_{\rst{f}} \ar[dr]^f \\ A \ar[r]_f& B} \]
is precise.
\end{enumerate}
\end{lemma}

\begin{proof}~
\begin{enumerate}[{\em (i)}]
\item We have the calculations:
\begin{description}
\item[($\Rightarrow$)]
If $\rst{g} = \rst{k}$ then as $g=kf$ we have $k\rst{f} = \rst{kf} k = \rst{g}k = \rst{k}k = k$.
\item[($\Leftarrow$)]
If $k\rst{f} = k$ then $\rst{g} = \rst{k f} = \rst{k\rst{f}} = \rst{k}$.
\end{description}
\item We shall use {\em (i)} above:  when the triangle is precise we have $g = ke = k \rst{e} = k$. Conversely, if $k=g$ then $k\rst{e} = ke = k$ so the triangle is precise.
\item Again we use {\em (i)} above: when the triangle is precise we have $g \alpha^{(-1)} = k \alpha \alpha^{(-1)} = k \rst{\alpha} = k$.  Conversely, if 
$k = g \alpha^{(-1)}$ we have
\[ k \rst{\alpha} = \rst{k \alpha}k =  \rst{g \alpha^{(-1)} \alpha}k =  \rst{g \rst{\alpha^{(-1)}}}k =  \rst{g \alpha^{(-1)}}k = \rst{k}k = k \]
showing the triangle is precise.
\item $\rst{f}f= f$ is precise with left factor $\rst{f}$ because $\rst{\rst{f}} = \rst{f}$.
\end{enumerate}
\end{proof}


\subsection{Cartesian restriction categories}

Cartesian restriction categories are described in \cite{cockett_lack_2007}.  They are restriction categories with a restriction terminal object and restriction products.  A {\bf restriction terminal object} in a restriction category, $\X$, is an object $1  \in \X$ such that for every object $X \in \X$ there is a unique total map $!_X\colon X \to 1$ such that every map $f\colon X \to 1$ has $f = \rst{f}!_X$.  
Given two objects $X,Y \in \X$ a {\bf restriction product} of $X$ and $Y$ is an object $X \x Y$ together with two projections $\pi_0\colon X \x Y \to X$ and $\pi_1\colon X \x Y \to Y$ which are total and are  such that given any other object $Z$ with two maps $a\colon Z \to X$ and $b\colon Z \to Y$ there is a unique map $\< a,b\>\colon Z \to X \x Y$ such that $\< a,b \>\pi_0 = \rst{b}a$ and 
$\< a,b\>\pi_1 = \rst{a}b$.
\subsection{Latent pullbacks}  

The notion of a latent pullback in a restriction category was introduced in \cite[pg. 459]{journal:guo-range-join}; here we give a slight modification of the original definition:  

\begin{definition}
A commuting square in a restriction category 
\[ \xymatrix{A' \ar[d]_a \ar[r]^{f'} & B' \ar[d]^{b} \\ A \ar[r]_f & B} \]
is a {\bf latent pullback} in case $f'\rst{b} = f'$ and $a \rst{f} = a$  (that is, it is precise) and given any $e$-commuting square (where $e=\rst{e}$ is a restriction idempotent)
\[ \xymatrix{X \ar[d]_{x_1} \ar[r]^{x_0} \ar@{}[dr]|{=_e} & B' \ar[d]^{b} \\ A \ar[r]_f & B} \]
that is $e x_1 f = e x_0 b$, where $e \leq \rst{x_1f}$ and $e \leq \rst{x_0b}$ there is a unique map $k\colon X \to A'$ 
\[ \xymatrix{X \ar[drr]^{x_0}_{\leq}  \ar[ddr]_{x_1}^{\geq} \ar[dr]|k \\
                    & A' \ar[d]_a \ar[r]_{f'} & B' \ar[d]^{b} \\  & A \ar[r]_f & B} \]
such that $\rst{k} = e$, $kf' \leq x_0$, $ka \leq x_1$, and $k \rst{f'} = k \rst{a} = k$.  
\end{definition}
The original definition asked that for any (ordinary) commuting square $x_0b = x_1f$ there was a unique $k$ with the same requirements above, except that $\rs{k} = \rs{x_1f} = \rs{x_0b}$ instead of $\rs{k} = \rs{e}$.  It is readily seen that the two definitions are equivalent, as if $e$ is a restriction idempotent on $X$ such that $e \leq \rs{x_1f}$ and $e \leq \rs{x_0b}$, then $\rs{e} = \rs{ex_1f} = \rs{ex_0b}$.  

We recall some basic properties of latent pullbacks:

\begin{lemma} \label{latent-pullbacks}~
\begin{enumerate}[(i)]
\item The two commuting squares
\[ \begin{matrix}\xymatrix{A' \ar[d]_a \ar[r]^{f'} & B' \ar[d]^{b} \\ A \ar[r]_f & B} \end{matrix} ~~~\mbox{and}~~~ 
   \begin{matrix}\xymatrix{Z \ar[d]_z \ar[r]^{y} & B' \ar[d]^{b} \\ A \ar[r]_f & B} \end{matrix} \]
are latent pullbacks if and only if there is a unique mediating partial isomorphism $\alpha\colon A' \to Z$ with 
$\alpha z = a$ , $\alpha y = f'$, $\rst{\alpha} =\rst{a}= \rst{f'}$, and $\rst{\alpha^{(-1)}} = \rst{z} = \rst{y}$ 
and one of the squares is a latent pullback.
\item If the two smaller squares below are latent pullbacks
\[ \xymatrix{A' \ar[d]_a  \ar[r]^{f'} & B' \ar[d]_b \ar[r]^{g'} & C' \ar[d]_c \\ A \ar[r]_f & B \ar[r]_g & C} \]
then the outer square is a latent pullback.  Furthermore, if the right square is a latent pullback and the perimeter is a latent pullback then the 
left square is a latent pullback.        
\item If $\alpha\colon X \to Y$ is a partial isomorphism and $f\colon Z \to Y$ is any map, then
	\[ \xymatrix{ Z \ar[r]^{\rst{f\alpha^{(-1)}}} \ar[d]_{f\alpha^{(-1)}} & Z \ar[d]^{f} \\ X \ar[r]_{\alpha} & Y} \]
is a latent pullback.  
\item If the latent pullback  square
\[  \xymatrix{A' \ar[d]_a \ar[r]^{f'} & B' \ar[d]^{b} \\ A \ar[r]_f & B} \]
has $f$ a partial isomorphism then $f'$ is a partial isomorphism.
\end{enumerate}
\end{lemma}

The last part of the lemma follows by combining parts (i) and (iii).  


\section{Latent Fibrations}\label{sec:latentFibrations}

We begin with the definition of a latent fibration before looking at examples and developing some of their basic properties.


\subsection{The definition}
 
\begin{definition} \label{newdefn}
Let $\E$ and $\B$ be restriction categories and ${\sf p} \colon \E \to \B$ a restriction semifunctor.  
\begin{enumerate}[(i)]
\item 
An arrow $f' \colon X' \to X$ in $\E$ is {\bf ${\sf p}$-prone} in case whenever we have $g \colon Y \to X$ in $\E$ and $h \colon {\sf p}(Y) \to {\sf p}(X')$ in $\X$ such that $h{\sf p}(f') = {\sf p}(g)$ is a precise triangle (Definition \ref{defn:precise}) then there is a unique {\bf lifting} of $h$ to $\widetilde{h}\colon Y \to X'$ so that $\p(\widetilde{h}) = h$ and $\widetilde{h}f' = g$ is a precise triangle:
\[ \xymatrix{Y \ar@{..>}[d]_{\widetilde{h}} \ar[dr]^{g} &~ \ar@{}[drr]|{\textstyle\mapsto}&& \p(Y) \ar[d]_h \ar[dr]^{\p(g)}  & ~
\\
 X'  \ar[r]_{f'} & X  && \p(X') \ar[r]_{\p(f')} & \p(X)} \]
\item
${\sf p} \colon \E \to \B$ is a {\bf latent fibration} if for each $X \in \E$ and each $f \colon A \to {\sf p}(X)$ in $\B$ such that $f = f\p(1_X)$, there is a ${\sf p}$-prone map $f'\colon X' \to  X$ sitting over $f$ (that is, ${\sf p}(f') = f$).
\end{enumerate}
\end{definition}
Note that if $\p$ is a restriction functor (so that $\p$ preserves identities) then the condition $f = f\p(1_X)$ is vacuous.  However, if $\p$ is a genuine semifunctor, then that condition is necessary, as if there is an $f'\colon X' \to X$ sitting over $f$, then since $f' = f'1_X$, $\p(f') = \p(f')\p(1_X)$, i.e., $f$ must satisfy $f = f\p(1_X)$.

Most of our examples of latent fibrations will be restriction functors; however, there is at least one important example of a latent fibration which is a genuine semifunctor (the forgetful functor from the restriction idempotent splitting of $\X$ to itself: see Proposition \ref{prop:splitting_example}).  

More importantly, however, the main reason we have chosen to work with restriction semifunctors is that latent fibrations are precisely fibrations in the 2-category ${\sf SRest}$ of restriction categories, semifunctors, and restriction transformations: see Appendix \ref{Appendix-B}. 


\subsection{Examples of latent fibrations}
It is clear that, for any restriction category $\X$, the identity functor is obviously a latent fibration.  Furthermore, for any restriction categories $\X$ and $\Y$, the projection $\pi_1\colon \X \times \Y \to \Y$ is a latent fibration.   Moreovr, there are a variety of other examples of latent fibrations: some are immediately seen to be restriction versions of a normal fibration couterpart; others are particular, however, to restriction categories.  We present an overview of some of the basic examples before going into a more detailed description of them.  
\begin{enumerate}[\bf ({3.2.}1)]
\item For any restriction category $\X$, the identity functor $1_{\X}\colon \X\to\X$ and projections $\pi_1: \Y \x \X \to \X$ are a latent fibrations.
	\item If $\X$ is a Cartesian restriction category, there are two latent versions of the simple slice fibration: the ``lax'' simple slice $\X(\X)$ and the ``strict'' simple slice $\X[\X]$ considered in the previous section.  See Definition \ref{def:simpleSlice} and Proposition \ref{prop:simpleSlice}.  
	\item If $\X$ is a restriction category with latent pullbacks, there are strict and lax versions of the codomain fibration: see Definition \ref{def:codomain} and Proposition  \ref{prop:codomain}. 
	\item For any functor $F\colon \X \to {\sf  Set}$ one may form the category of elements as a (normal) discrete fibration $\partial_F\colon {\sf Elt}(F) \to \X$. If $\X$ is a restriction category then 
	${\sf Elt}(F)$ is also a restriction category and $\partial_F$ is a latent fibration: see Section \ref{discrete}.
	\item For any restriction category $\X$, there is a latent fibration of ``propositions'' (restriction idempotents) over $\X$: see Definition \ref{def:propositions} and Proposition \ref{prop:propositions}. 
	\item For any restriction functor $F\colon \A \to \X$, there is a latent fibration of ``assemblies'', denoted ${\sf Asm}(F)$: see Definition \ref{def:assemblies} and \ref{prop:assemblies}.  
	\item For any restriction category $\X$, the forgetful functor from the restriction idempotent splitting of $\X$ to itself, $\split(\X) \to \X$ is a latent fibration which is a genuine semifunctor: see Proposition \ref{prop:splitting_example}.  
\end{enumerate}


\subsubsection{Identity and projection functors}

Clearly identity functors and projections are latent fibrations.  The prone map over a map for the identity function is just itself.  For a projection $\pi_1: \Y \x \X \to \X$ the prone map over 
a map $f: X \to X'$ at $Y$ is $1 \x f: Y \x X \to Y \x X'$.


\subsubsection{Simple latent fibrations}
\label{simple-latent-fibration} 

We consider simple fibrations for Cartesian restriction categories:  

\begin{definition}\label{def:simpleSlice}
If $\X$ is a Cartesian restriction category, the total category of the \textbf{lax simple fibration}, $\X(\X)$, is described as follows:
\begin{description}
\item[Objects:] Pairs of objects, $(\Sigma,X)$, of $\X$, where $\Sigma$ is called the ``context'';
\item[Maps:] $(f,f')\colon (\Sigma,X) \to (\Sigma',X')$ where $f\colon \Sigma \to \Sigma'$ and $f'\colon \Sigma \x X \to X'$ are maps in $\X$ such that 
$\rst{\pi_0 f} \geq \rst{f'}$;
\item[Composition:] Identities are $(1_\Sigma, \pi_1)\colon (\Sigma,X) \to (\Sigma,X)$ and given $(f,f')\colon  (\Sigma,X) \to (\Sigma',X')$ and $(g,g')\colon  (\Sigma',X') \to (\Sigma'',X'')$ the composite is the map $(f,f')(g,g') := (fg,\<\pi_0f,f'\> g')$;
\item[Restriction:] $\rst{(f,f')} := (\rst{f},\rst{f'}\pi_1)$.
\end{description}
The \textbf{strict simple fibration} is the subcategory $\X[\X]$ determined by the maps $(f,f')$ for which $\rst{\pi_0f} = \rst{f'}$.  
\end{definition}

The strict simple fibration was discussed in the introduction with respect to the differential of smooth partial maps as this differential may be seen as  a  section of the strict simple fibration. 

\begin{proposition}\label{prop:simpleSlice}
If $\X$ is a Cartesian restriction category, then $\X(\X)$ and $\X[\X]$ are Cartesian restriction categories, and the obvious projections to $\X$ are latent fibrations.  
\end{proposition}

\begin{proof}
For the identity to be well-defined we need $\rst{\pi_0 1} \geq \rst{\pi_1}$, which is obviously true as both sides are restrictions of total maps.  
To show composition is  well defined consider $(f,f')(g,g') = (fg,\<\pi_0f,f'\>g')$: we must show $\rst{\pi_0fg} \geq \rst{\<\pi_0f,f'\>g'}$ which is so by:
\[ \rst{\<\pi_0f,f'\>g'} \leq \rst{\<\pi_0f,f'\>\pi_0 g} = \rst{\rst{f'}\pi_0fg} = \rst{\rst{f'}\rst{\pi_0f}\pi_0fg} =  \rst{\rst{\pi_0f}\pi_0fg} = \rst{\pi_0fg}. \]

For the restriction product structure, set $(\Sigma,X) \x (\Sigma',X') := (\Sigma \x \Sigma',X \x X')$,
	\[ \pi_0:= (\pi_0, \pi_1\pi_0)\colon (\Sigma,X) \x (\Sigma',X') \to (\Sigma,X), \pi_1:= (\pi_1, \pi_1\pi_1)\colon (\Sigma,X) \x (\Sigma',X') \to (\Sigma',X') \] 
and 
	\[ \< (f,f'),(g,g') \> := (\<f,g\>, \<\<\pi_0\pi_0f,\pi_1\pi_0\>f',\<\pi_0\pi_1g,\pi_1\pi_1\>g'\>). \]   
We leave the remaining details of checking this is a Cartesian restriction category to the reader.  

There is an obvious restriction functor 
\[ \pi\colon \X(\X) \to \X; \quad \begin{matrix} \xymatrix{ (\Sigma, X) \ar[d]_{(f,f')}\ar@{}[dr]|{\mapsto}  & \Sigma \ar[d]^{f} \\ (\Sigma',X') & \Sigma' } \end{matrix} \]
We set the prone arrow of $f\colon\Sigma \to \Sigma'$ at $(\Sigma',X)$ to be $(f,\rst{\pi_0 f}\pi_1)\colon (\Sigma,X) \to (\Sigma',X)$; the lifting property is given by
\[ \xymatrix@C=3em{(\Gamma,Y)  \ar@{..>}[d]_{\widetilde{h} := (h,g')} \ar[dr]^{(g,g')} & ~ \ar@{}[drr]|{\textstyle\mapsto}&& \Gamma \ar[dr]^g \ar[d]_{h} 
\\ (\Sigma,X) \ar[r]_{(f,\rst{\pi_0f}\pi_1)} & (\Sigma',X) && \Sigma \ar[r]_f & \Sigma'} \]

The lifting $\widetilde{h}$ is well-defined as $\rst{g'} \leq \rst{\pi_0 g} = \rst{\pi_0 h}$ since the base triangle is precise.  The top triangle commutes since
\begin{eqnarray*}
&    & (h,g') (f,\rst{\pi_0f} \pi_1) \\
& = & (hf, \<\pi_0h, g'\>\rs{\pi_0 f} \pi_1)  \\
& = & (g,\rs{\<\pi_0h, g'\>\pi_0f} \<\pi_0h,g'\>\pi_1) \\
& = & (g,\rs{\rs{g'}\pi_0 hf} \rs{\pi_0 h} g') \\
& = & (g,\rs{g'}\rs{\pi_0g}\rs{\pi_0h} g') \\
& = & (g,g') \mbox{ (since $\rs{g'} \leq \rs{\pi_0g} = \rs{\pi_0h}$)}
\end{eqnarray*}
The top triangle is precise since the bottom triangle is.  The uniquess of $\widetilde{h}$ follows from a similar calculation to showing the top triangle commutes.  Thus, the lax slice $\X(\X)$ is a latent fibration, and the result for the strict slice $\X[\X]$ follows similarly.  

\end{proof}


\subsubsection{Codomain latent fibrations}
\label{codomain-latent-fibration}

When $\X$ has latent pullbacks, restriction variants of the codomain fibration form another source of examples.  In fact, just as in the ordinary case (for example, see \cite{book:Jacobs-Cat-Log}), a restriction category has latent pullbacks if and only if the underlying functor from the (strict) arrow category is a latent fibration.   

\begin{definition}\label{def:codomain}
For any restriction category $\X$, the restriction category $\X^{\leadsto}$ is defined as follows:
\begin{description}
\item[Objects:] are maps $a\colon A' \to A$ of $\X$;
\item[Maps:] are pairs of maps $(f,f')\colon a \to b$ such that $f'b \leq af$ and $f' \rst{b} = f'$:
\[ \xymatrix{A' \ar@{}[dr]|{\geq} \ar[d]_{a} \ar[r]^{f'} & B' \ar[d]^b \\ A \ar[r]_f & B} \]
we will refer to such lax squares as being {\bf semi-precise};
\item[Composition:] if $(f,f')\colon a \to b$ and $(g,g')\colon b \to c$ then the composite is $(ff',gg')\colon a \to c$;
\item[Restriction:] the restriction of $(f,f')\colon a \to b$ is $(\rst{f},\rst{f'})\colon a \to a$. 

\end{description}
The restriction category $\X^{\rightarrow}$ has the same definition except that the maps require the squares commute (and are still semi-precise).  Thus there is an embedding $\X^{\rightarrow} \subseteq \X^{\leadsto}$ which is the identity on objects but, in general, a strict inclusion on maps.
\end{definition}

We define the functor $\partial^{\leadsto}$ by
\[ \partial^{\leadsto} \colon\X^{\leadsto} \to \X; \quad\begin{matrix} \xymatrix{a \ar[d]_{(f,f')}\ar@{}[dr]|{\mapsto} & A \ar[d]^f \\ b & B} \end{matrix} \]
it is clearly a restriction functor.  The restriction functor $\partial^{\rightarrow}$ is defined similarly.

\begin{proposition}\label{prop:codomain}
For any restriction category, $\X$, $\X^{\leadsto}$ and $\X^{\rightarrow}$ are restriction categories, and the restriction functors, $\partial\colon \X^{\leadsto} \to \X$ and $\partial\colon \X^{\rightarrow} \to \X$ are latent fibrations precisely when $\X$ has latent pullbacks.
\end{proposition}

\begin{proof}
In $\X^{\leadsto}$, composites are well-defined as
               \[ afg \geq f'bg \geq f'g'c~~~\mbox{and}~~~ f'g' \rst{c} = f'(g' \rst{c}) = f'g' \]
and restrictions are well-defined as
\[  a \rst{f} = \rst{af} a \geq \rst{f'b} a = \rst{f'\rst{b}} a = \rst{f'} a ~~~\mbox{and}~~~ \rst{f'}~\rst{a} = \rst{f'\rst{b}} ~\rst{a} = \rst{f'b} ~\rst{a} = \rst{f'b}~\rst{af}~\rst{a} = \rst{f'b}~\rst{af} = \rst{f'b} = \rst{f'} \]
Similarly composites and restrictions are well-defined in $\X^{\rightarrow}$.

First we show that, when $\X$ has latent pullbacks, we have $\partial$-prone arrows, making it a latent fibration.   

Latent pullback gives a prone arrow in $\X^{\leadsto}$ by considering
\[ \xymatrix{C' \ar[dd]_c \ar[drr]^{g'} \ar@{..>}[dr]_{\widetilde{w}} \\
                                  & P\ar[dd]^{a'} \ar[r]_{f'} & A ' \ar[dd]^{a} \\
                                  C \ar@{}[ur]|\geq \ar[rrd]^<<<<g|\hole \ar[dr]_w \\
                                  & B \ar[r]_f  & A} \]
where the back face is a lax semi-precise square (so $cg \geq g'a$), the bottom triangle is precise (so $\rst{w} = \rst{g}$ and it commutes) and the square with apex 
$P$ is a latent pullback.  The back face $\rst{g'a} = \rst{g'}$-commutes so there is a unique $\widetilde{w}$ with $\rst{\widetilde{w}} = \rst{g'}$, $\widetilde{w}f' \leq g'$,
$\widetilde{w}a' \leq cw$.  The top triangle is now clearly precise as $\widetilde{w}f' = g'$ since $\rst{\widetilde{w}f' }= \rst{\widetilde{w}}= \rst{g'}$.  Finally, we must show 
the left square is semi-precise: 
\[ \widetilde{w} \rst{a'} = \rst{\widetilde{w} a'}  \widetilde{w}  =  \rst{\widetilde{w} a' \rst{f}}  \widetilde{w} = \rst{\widetilde{w} a' f}  \widetilde{w} = \rst{\widetilde{w} f' a} \widetilde{w} 
= \rst{\widetilde{w} f' \rst{a}} \widetilde{w} = \rst{\widetilde{w} f'} \widetilde{w} = \rst{\widetilde{w} \rst{f'}} \widetilde{w} = \rst{\widetilde{w}} \widetilde{w} = \widetilde{w}. \]
This shows that latent pullbacks are prone arrows; thus, if $\X$ has latent pullbacks $\partial\colon \X^{\leadsto} \to \X$ is a latent fibration.

We need to show that $\X^{\rightarrow}$ also has prone arrows given by the latent pullback.  For this we need to show that if the back square is an equality the 
front must also be an equality.  This is given by the following calculation:
\[ \rst{\widehat{w} a'} = \rst{\widehat{w}} = \rst{ g'} = \rst{g'a} = \rst{cg} = \rst{c\rst{g}} = \rst{c \rst{w}} = \rst{cw} \]

Conversely if $\partial$ is a latent fibration then we have a prone arrow above and $B \to^f A$ and this at least is a lax semi-precise square.  However, by restricting $b$ by 
$\rst{f'}= \rst{f'a}$ we can turn it into a precise square.  Next suppose we have an $e$-commuting square with apex $Z$ then modifying the top arrow to $ez_0$ gives a lax semi-precise square:
\[ \xymatrix{ Z \ar[ddr]_{z_1}^{\geq} \ar[drr]^{ez_0} \ar@{..>}[dr]|{k} \\
                                  & X\ar[d]^{f} \ar[r]_{\rst{fe}} & X \ar[d]^{f} \\
                                  & Y \ar[r]_e  & Y} \]
This implies there is a $k\colon Z \to P$ with the top triangle precise and the left face a lax semi-precise square.  However, $k$ clearly now provides the mediating map 
to make the modified prone square a latent pullback.
\end{proof}

If $\X$ has latent pullbacks, it is not hard to see that the (strict/lax) simple latent fibration can be embedded into the (strict/lax) codomain latent fibration in a manner analogous to that for ordinary fibrations: 

\[ \xymatrix{\X(\X) \ar[d]_\pi\ar@{}[dr]|{\rightarrow} & \X^{\leadsto} \ar[d]^\partial \\ \X  & \X \ar@{}[u]_{~~~~;}}
\quad\xymatrix{(\Sigma,X) \ar[d]^{(f,f')}\ar@{}[drr]|{\mapsto}  && \Sigma \x X \ar[d]_{\< \pi_0f,f'\>}\ar@{}[dr]|{\leq} \ar[r]^{\pi_0} & \Sigma \ar[d]^f \\ 
(\Sigma',X') && \Sigma' \x X' \ar[r]_{\pi_0} & \Sigma'} \]


\subsubsection{Discrete latent fibrations}
\label{discrete}

Let $\X$ be a restriction category and $F\colon \X^{\rm op} \to {\sf Set}$ any functor (which ignores the restriction structure), then we may form the category of elements of $F$, ${\sf Elt}(F)$:
\begin{description}
\item[Objects:]   $(X,x)$ where $X$ is an object of $\X$ and $x$ is an element of $F(X)$;
\item[Maps:] $f\colon (X,x) \to (Y,y)$ where $f\colon X \to Y$ in $\X$ and $x = F(f)(y)$.
\end{description}
The functor $\partial_F\colon {\sf Elt}(F) \to \X$ has $\partial_F(X,x) = X$ and $\partial_F(f) = f$.  

The restriction of ${\sf Elt}(\X)$ is just the restriction in $\X$.  We must check that this is well-defined that is that if $f\colon (X,x) \to (Y,y)$ that $\rst{f}\colon (X,x) \to (X,x)$ is a map of 
${\sf Elt}(F)$.  For this we need $F(\rst{f})(x) = x$ which follows as $F(\rst{f})(x) = F(\rst{f})(F(f)(y)) = F(\rst{f}f)(y) = F(f)(y) = x$.  The restriction identities are then immediate.

This is a latent fibration as one can easily check that the prone map above $f\colon X \to Y$ at $(Y,y)$ is the usual Cartesian map, namely, $f\colon (X,F(f)(y)) \to (Y,y)$.  


\subsubsection{Latent fibrations of propositions}
\label{propositions}

One way in which restriction idempotents can be used to construct a latent fibration is as follows:

\begin{definition}\label{def:propositions}
Let $\X$ be a restriction category, and define the restriction category $\mathcal{O}(\X)$ by:
\begin{description}
\item[Objects:] pairs $(X,e)$ where $X$ is an object of $\X$, and $e \in {\cal O}(X)$ is a restriction idempotent on $X$.
\item[Maps:] $f\colon (X,e) \to (X',e')$ are maps $f \colon X \to X'$ of $\X$ such that $e \leq \rst{fe'}$, or equivalently $e = \rst{efe'}$.
\item[Composition:] is composition in $\X$. This is well-defined since if 
\[(X,e) \to^f (X',e') \to^{f'} (X'',e'') \]
are maps of $\mathcal{O}(\X)$ given by $f \colon X \to X'$ and $f' \colon X' \to X''$ in $\X$, then we have
\[ \rst{eff'e''} = \rst{\rst{efe'}ff'e''} = \rst{e\rst{fe'}ff'e''} = \rst{efe'f'e''}  = \rst{ef\rst{e'f'e''}} = \rst{efe'} = e \]
so that $ff'$ gives a map $(X,e) \to (X'',e'')$ in $\mathcal{O}(\X)$.
\item[Identities:] as in $\X$. That is, $1_{(X,e)} = 1_X$: this is well-defined as $e = \rst{e1_Xe}$.
\item[Restriction:] also as in $\X$, with $\rst{f} \colon (X,e) \to (X,e)$ for $f \colon (X,e) \to (X',e')$ given by $\rst{f} \colon X \to X$. This is well-defined as 
\[ e = \rst{efe'} = \rst{e\rst{f}fe'} = \rst{ee\rst{f}fe'} = \rst{e\rst{f}efe'} = \rst{e\rst{f}\,\rst{efe'}} = \rst{e\rst{f}e} \]
\end{description}
\end{definition}

\begin{proposition}\label{prop:propositions}
The canonical map $\mathcal{O} \colon \mathcal{O}(\X) \to \X$ is a latent fibration. 
\end{proposition}

We refer to $\mathcal{O}$ as the {\bf latent fibration of propositions}.  

\begin{proof}
Associativity of composition, the restriction combinator axioms, and the requirements on identity maps all hold in $\mathcal{O}(\X)$ immediately since they hold in $\X$. 

There is an obvious restriction functor 
\[ \mathcal{O} \colon \mathcal{O}(\X) \to \X;\quad \begin{matrix} \xymatrix{ (X,e) \ar[d]_f\ar@{}[dr]|{\mapsto} & X \ar[d]_f \\ (Y,d) & Y} \end{matrix} \]

Suppose $f \colon X' \to X$ in $\X$ and let $(X,e)$ be an object of ${\cal O}(\X)$. Then $(X',\rst{fe})$ is also an object of ${\cal O}(\X)$, and $f\colon(X',\rst{fe}) \to (X,e)$ is a map in ${\cal O}(\X)$ since $\rst{fe} = \rst{\rst{fe}fe}$. We shall show that this map is prone over $f$. To that end, suppose that $g \colon (Y,e') \to (X,e)$ and $h \colon Y \to X'$ are maps in ${\cal O}(\X)$ and $\X$ respectively such that
\[\xymatrix{Y \ar[rd]^g \ar[d]_h & \\X' \ar[r]_f  & X }\]
is precise in $\X$, then $h\colon (Y,e') \to (X',\rst{fe})$ is a map in ${\cal O}(\X)$ since $\rst{h \rst{fe}} = \rst{hfe} = \rst{ge} \geq e'$.  Furthermore this gives a precise triangle in ${\cal O}(\X)$ 
showing $f\colon (X',\rst{fe}) \to (X,e)$ is prone.
\end{proof}


\subsubsection{Assembly categories}
\label{assemblies}

Another example comes from \cite{msc:nester-calgary} and is constructed from a category of assemblies, ${\sf Asm}(F)$, associated to a restriction functor, $F\colon \A \to \X$, from any restriction category, $\A$, called the category of {\bf realizers}, into a Cartesian restriction category, $\X$, the base.  

\begin{definition}\label{def:assemblies}
The category of assemblies, ${\sf Asm}(F)$, is defined as follows:
\begin{description}
\item[Objects:]  $\varphi \in {\cal O}(F(A) \x X)$ for all objects $A \in \A$ and $X \in \X$.
\item[Maps:]  $f\colon \varphi \to \varphi'$, where $\varphi \in {\cal O}(F(A) \x X)$and $\varphi' \in {\cal O}(F(A') \x X')$, are maps $f\colon X \to X'$  for which there is a ``tracking'' map 
$\gamma\colon A \to A'$ in $\A$ satisfying:
\begin{enumerate}[{\bf [Tk.1]}] 
\item $\varphi(F(\gamma) \x f) =  \varphi(F(\gamma) \x f)\varphi'$
\item $\rst{\varphi(1 \x f)} = \rst{\varphi(F(\gamma) \x f)}$
\end{enumerate}
\item[Restriction:]  As in $\X$.
\end{description}
\end{definition}

The proof that ${\sf Asm}(F)$ is a restriction category -- and, indeed, when $\A$ is a Cartesian restriction category and $F$ preserves this Cartesian structure, then ${\sf Asm}(F)$ is a Cartesian restriction category -- may be  found in \cite[Prop. 5.2 and 5.3]{msc:nester-calgary}.

If $\X$ is a Cartesian restriction category and $F \colon \A \to \X$ is a restriction functor, the forgetful functor ${\sf p} \colon {\sf Asm}(F) \to \X$ is a latent fibration. Specifically, if $\varphi \in {\cal O}(F(A) \times X)$, $\psi \in {\cal O}(F(B) \times Y)$, and $f \colon \varphi \to \psi$  a map in ${\sf Asm}(F)$, then ${\sf p}(f)$ is $f \colon X \to Y$, the underlying map in $\X$. Clearly ${\sf p}$ is a Cartesian restriction functor.

\begin{proposition}\label{prop:assemblies}
${\sf p} \colon {\sf Asm}{F} \to \X$ as defined above is a latent fibration.
\end{proposition}
\begin{proof}
Suppose $f \colon X \to Y$ is a map in $\X$, and let $\psi \in {\cal O}(F(B) \times Y)$ be an assembly. Then $\rst{(1 \times f)\psi} \in {\cal O}{F(B) \times X}$ is also an assembly, and 
$f$ can be viewed as a map $\rst{(1 \times f)\psi} \to \psi$ in ${\sf Asm}(F)$ which is readily seen to be prone. 
\end{proof}


\subsubsection{Idempotent splitting}
\label{splitting}

We end these examples with a latent fibration which is a genuine semifunctor. 
\begin{proposition}\label{prop:splitting_example}
For any restriction category $\X$, the forgetful functor from the restriction idempotent splitting of $\X$ back to $\X$, $U: \split(\X) \to \X$, is a latent fibration.
\end{proposition}
\begin{proof}
Recall that in $\split(\X)$:
\begin{itemize}
	\item an object is a pair $(X,e)$ with $e$ a restriction idempotent on $\X$;
	\item a map $f\colon (X,e) \to (X',e')$ is a map $f\colon X \to X'$ such that $efe' = f$;
	\item composition and restriction are as in $\X$;
	\item the identity of $(X,e)$ is $e$ itself.
\end{itemize}
Thus in general the forgetful functor $U\colon \split(\X) \to \X$ is a genuine semifunctor (that is, it preserves composition and restriction but not necessarily identities).  

To see that it is a latent fibration, suppose $(X,e)$ is an object of $\split(\X)$, and $f\colon X' \to \p(X)$ in $\X$ is such that $f = fU(1_{(X,e)}) = fe$.  We want to show that $f\colon (X',\rs{fe}) \to (X,e)$ is a prone lift of $f$.  Indeed, $f$ is well-defined since
	\[ \rs{fe}f = f\rs{e}e = fe = f \]
If we have a precise factorization $hf=g$, then $h$ is the unique lift:
\[ \xymatrix{(X'',e') \ar@{..>}[d]_{h} \ar[dr]^{g} & ~ \ar@{}[drr]|{\textstyle\mapsto}&&X'' \ar[d]_h \ar[dr]^{g}  & ~ 
\\ 
(X',\rs{fe})  \ar[r]_{f} & (X,e) &&  X' \ar[r]_{f} & X} \]
It is well-defined since
	\[ e''h\rs{fe} = \rs{e''hfe}h = \rs{e''ge}h = \rs{g}h = \rs{h}h = h, \]
and is clearly unique.
\end{proof}


\subsection{Basic theory of latent fibrations and prone arrows}

A simple observation is that a latent fibration reduces to an ordinary fibration when it is between total map categories.  To see this note that (i) a restriction semifunctor between such categories must be an ordinary functor, and (ii) the requirement of preciseness is automatically satisfied by total maps and so the condition of having enough prone maps reduces to the definition of a Cartesian arrow in an ordinary fibration.

\begin{lemma} \label{latent-fibration-total}
A latent fibration between restriction categories in which all maps are total is a fibration.
\end{lemma}

Just as for ordinary fibrations we can define morphisms of latent fibrations:

\begin{definition}
A morphism of  latent fibrations $F = (F_1,F_0)\colon (\p\colon \E \to \B) \to (\p'\colon \E' \to \B')$ is a pair of restriction semifunctors, respectively, between the ``total'' categories, $F_1\colon \E \to \E'$, and the bases $F_0\colon \B \to \B'$, making
\[ \xymatrix{ \E \ar[r]^{F_1} \ar[d]_{\p} & \E' \ar[d]^{\q} \\ \B \ar[r]_{F_0} & \B'} \]
commute, and such that $F_1$ preserves prone maps.
\end{definition}

Transformations of fibrations are ``pillows'' of  restriction transformations  $(\alpha,\beta)\colon (F_1,F_0) \to (G_1,G_0)$ with $\alpha \q = \p \beta$.

Many results for Cartesian arrows for an ordinary functor have a restriction analogue for prone arrows for a restriction functor. For example:

\begin{lemma} \label{composites-prones}
For any restriction semifunctor, ${\sf p}\colon \E \to \B$, composites of ${\sf p}$-prone arrows in $\E$ are ${\sf p}$-prone and identity maps are $\p$-prone.
\end{lemma}

\begin{proof}
Suppose $f_1$ and $f_2$ are prone in $\E$ and $k {\sf p}(f_1f_2) = {\sf p}(g)$ is a precise triangle.  Then $(k {\sf p}(f_1)) {\sf p}(f_2)$ is also precise as:
\[ k {\sf p}(f_1) \rst{{\sf p}(f_2)} = \rst{k {\sf p}(f_1f_2)} k {\sf p}(f_1)  = \rst{k \rst{{\sf p}(f_1f_2)}} k {\sf p}(f_1) = \rst{k} k {\sf p}(f_1) = k {\sf p}(f_1) \]
Now as $f_2$ is prone we have a unique $\widetilde{k {\sf p}(f_1)}$ sitting above making $\widetilde{k {\sf p}(f_1)} f_2$ precise.  But then as 
$k {\sf p}(f_1) = {\sf p}(\widetilde{k {\sf p}(f_1)})$ is precise this gives a unique $\widetilde{k}$ with 
$\widetilde{k} f_1 = \widetilde{k {\sf p}(f_1)}$ precise.  This certainly  means $\widetilde{k} (f_1 f_2)= \widetilde{k {\sf p}(f_1)} f_2 = g$ but also this is 
precise as  $\rst{\widetilde{k}} = \rst{\widetilde{k {\sf p}(f_1)}}$ as $\widetilde{k} f_1 = \widetilde{k {\sf p}(f_1)}$ precisely commutes, and 
$\rst{\widetilde{k {\sf p}(f_1)}} = \rst{g}$ as $\widetilde{k {\sf p}(f_1)} f_2 =g$ precisely commutes, so $\rst{\widetilde{k}} = \rst{g}$ showing 
$\widetilde{k} (f_1f_2)= g$ precisely commutes.   

Finally $\widetilde{k}$ is unique as supposing $h$ was an alternate then $(hf_1) f_2 = g$ precisely commutes with ${\sf p}(hf_1) = k{\sf p}(f_1)$ making 
$hf_1 = \widetilde{k{\sf p}(f_1)}$ but then by similar reasoning $h= \widetilde{k}$.

Identity maps are clearly always prone when $\p$ is a restriction functor, but thi is not completely immediate if $\p$ is only a semifunctor.  However, it still works in this case as if we have
\[ \xymatrix{Y \ar@{..>}[d]^{} \ar[dr]^{g} & ~\ar@{}[drr]|{\textstyle\mapsto}&& \p(Y) \ar[d]_h \ar[dr]^{\p(g)}  & ~ 
\\ 
X  \ar[r]_{1_X} & X && \p(X) \ar[r]_{\p(1_X)} & \p(X)} \]
with the base triangle precise, then since $\p(1_X)$ is a restriction idempotent, by Lemma \ref{Jaws}.ii, $h = p(g)$.  Thus $g$ is the unique fill-in for the triangle in $\E$.  
\end{proof}

As we shall see, however, restriction idempotents are \emph{not} always prone: see Section \ref{sec:types} for details, and Proposition \ref{prop:separated_equivalences} for a characterization of which semifunctors have all restriction idempotents prone.  

In the total case, it is well-known that any two Cartesian arrows over the same arrow (with a common codomain) have a unique mediating isomorphism.  For prone arrows we now show that the analogous situation induces a mediating \emph{partial} isomorphism.   

\begin{definition}
In a restriction category, a {\bf mediating} map between two maps $f$ and $f'$ with a common codomain, is a partial isomorphism, $\alpha$,  such that $\alpha f' = f$ and $\alpha^{(-1)}f = f'$ are both precise.  
\end{definition}

\begin{lemma} \label{mediating_maps}
If ${\sf p}\colon \E \to \B$ is a restriction semifunctor, then:
\begin{enumerate}[(i)]
\item If $f\colon X \to Y$ and $f'\colon X' \to Y$ are ${\sf p}$-prone maps of $\E$ with ${\sf p}(f) = {\sf p}(f')$  
then there is a unique mediating partial isomorphism $\alpha \colon X \to X'$ (so that    ${\sf p}(\alpha) = \rst{{\sf p}(f)}$, $\alpha f' = f$, and $\rst{\alpha} = \rst{f}$, $\alpha^{(-1)} f = f'$, and $\rst{\alpha^{(-1)}} = \rst{f'}$).
\item If $\alpha$ is a mediating partial isomorphism in $\E$ between $f$ and $f'$ (so that $\alpha f' = f$, and $\rst{\alpha} = \rst{f}$, $\alpha^{(-1)} f = f'$, and $\rst{\alpha^{(-1)}} = \rst{f'}$ as above) then, if either $f$ or $f'$ is ${\sf p}$-prone, then both $f$ and $f'$  are ${\sf p}$-prone.
\item If $e,e'\colon X \to X$ are both prone restriction idempotents in $\E$ such that $\p(e) = \p(e')$, then $e = e'$.
\end{enumerate}
\end{lemma}

\begin{proof}~
\begin{enumerate}[{\em (i)}]
\item Let $\alpha,\alpha'$ be the liftings for the precise triangle $\rst{{\sf p}(f)}{\sf p}(f) = {\sf p}(f)$ (recalling ${\sf p}(f)={\sf p}(f')$):
\[
\xymatrix{
X \ar@{.>}[d]_\alpha \ar[rd]^f \\
X' \ar[r]_{f'} & Y
}
\hspace{30pt}
\xymatrix{
X' \ar@{.>}[d]_{\alpha'} \ar[rd]^{f'} \\
X \ar[r]_f & Y
}
\]
The fact that they are precise triangles implies immediately that $\rst{\alpha} = \rst{f}$ and $\rst{\alpha'} = \rst{f'}$.
The composite of these precise triangles sits over the same precise triangle, $\rst{\p(f)}\p(f) = {\sf p}(f)$, and so, using Lemma \ref{Jaws}.iv, $\alpha\alpha' = \rst{f}$ and $\alpha'\alpha = \rst{f'}$ making them partial inverses.
 \item  Suppose $f$ and $f'$ are mediated by $\alpha$, then the precise triangles on $f$ are in bijective correspondence to those on $f'$, as a precise $kf=g$ is carried to 
 a precise $(k\alpha) f' = g$ from which the result is immediate.
 \item By (i), there is a mediating partial isomorphism $\alpha\colon X \to X$ between $e$ and $e'$.  In particular, $\alpha e = e'$ is precise.  But then by Lemma \ref{Jaws}.ii, 
    this means $\alpha = e'$, and similarly $\alpha^{(-1)} = e$.  But then $\alpha$ is a restriction idempotent, so its partial inverse is itself, so we have $e = \alpha^{(-1)} = \alpha = e'$.  
\end{enumerate}
\end{proof}

As noted above, restriction idempotents are not always prone.  However, isomorphisms, and, more generally, restriction retractions are:

\begin{lemma} \label{restriction-retraction}
If ${\sf p}\colon \E \to \B$ is a restriction semifunctor, then: 
\begin{enumerate}[(i)]
\item Isomorphisms are always $\p$-prone and, if $f$ is total and $\p$-prone with $\p(f)$ an isomorphism and $\p$ a restriction functor, then $f$ is itself an isomorphism.
\item All restriction retractions in $\E$ are ${\sf p}$-prone.
\item If $f = r f'$ is ${\sf p}$-prone in $\E$ and $r$ is a restriction retraction, then $f'$ is $\p$-prone.
\item If $\p$ is a restriction functor, and $f$ is a ${\sf p}$-prone map such that $\rst{f}$ splits, then $\p(f)$ is a restriction retraction if and only if $f$ is a restriction retraction.
\end{enumerate}
\end{lemma}

\begin{proof}~
\begin{enumerate}[{\em (i)}]
\item That isomorphisms are prone is immediate (and follows for example from part {\em (ii)}, see below).  Suppose $f\colon E' \to E$ is total, prone, and 
$\p(f)$ is an isomorphism with $\p$ a restriction functor, then $\p(f)^{-1}$ lifts:
\[ \xymatrix{E \ar@{=}[dr] \ar@{..>}[d]_{\widetilde{\p(f)^{-1}}} & ~\ar@{}[drr]|{\textstyle\mapsto}&&& \p(E) \ar@{=}[dr] \ar[d]_{\p(f)^{-1}}
\\ 
E' \ar[r]_{f'} & E &&&\p(E') \ar[r]_{\p(f)} & \p(E)} \]
 and then we have 
 \[ \xymatrix{E' \ar[dr]^{f} \ar@{..>}[d]_{f~\widetilde{p(f)^{-1}}} & ~ \ar@{}[drr]|{\textstyle\mapsto}&& \p(E) \ar[dr]^f \ar@{=}[d] 
\\ 
E' \ar[r]_{f} & E && \p(E') \ar[r]_{\p(f)} & \p(E)} \]
 showing $f~\widetilde{\p(f)^{-1}}$ is a lifting of the identity map.  However, an alternate lifting is the identity map so $f~\widetilde{\p(f)^{-1}} = 1_{E'}$, showing $f$ is an isomorphism.  
 This is the same argument for why Cartesian arrows above isomorphisms are isomorphims in an ordinary fibration.
\item A restriction retraction, $r$, is a partial isomorphism whose inverse is a section $m$, thus $rm = \rst{r}$ and $mr = 1$.  
Suppose ${\sf p}(g) = h {\sf p}(r)$ then, by Lemma \ref{Jaws}.iii, $h = {\sf p}(g)  {\sf p}(m) = {\sf p}(gm)$ so that $gm$ is a possible lifting and will necessarily be so provided $gmr = g$ but this is the case as $gmr = g1 = g$.

\item Suppose the restriction retraction $r\colon Y \to X'$ has section $m\colon X' \to Y$, so $rm = \rs{r}$ and $mr = 1_{X'}$.  Suppose that there is a $g\colon Z \to X$ so that $h\p(f') = \p(g)$ is precise.  Then we claim that there is a lifting $\tilde{h}$ over $h\p(m)$ via the prone-ness of $f$:
	\[ \xymatrix{Z \ar[dr]^g \ar@{..>}[d]_{\widetilde{h} } & ~ \ar@{}[drr]|{\textstyle~~\mapsto}&&&\p(Z) \ar[dr]^{\p(g)} \ar[d]_{h \p(m)}
\\ 
Y \ar[r]_{f} & X &&& \p(Y) \ar[r]_{\p(f)} & \p(X)} \]
For this, we need to check that the base triangle is precise.  (It obviously commutes by definition of $f$).  For preciseness, note that since $h\p(f') = \p(g)$ is precise,
	\[ h = h \rs{\p(f')} = h \rs{\p(1_{X'}f')} = h \rs{\p(1_{X'})\p(f')} \leq h \rs{\p(1_{X'})}, \]
but we always have the opposite inequality, so $h  = h \rs{\p(1_{X'})}$.  Thus the base triangle above is precise since
	\[ \rs{h\p(m)} = \rs{h\p(\rs{m})} = \rs{h\p(1_{X'})} = \rs{h} = \rs{\p(g)}. \]
	
We now claim that $\tilde{h}r\colon Z \to X'$ is the unique lift for $h$, demonstrating the prone-ness of $f'$.  The triangle commutes since $ \tilde{h}rf' = \tilde{h}f = g$, is precise since
	\[ \tilde{h} r\rs{f'} = \rs{\tilde{h}rf'} \tilde{h}r = \rs{\tilde{h}f}\tilde{h}r = \rs{\tilde{h}}\tilde{h}r = \tilde{h}r, \]
and is over $h$ since 	
	\[ \p(\tilde{h}r) = h\p(m)\p(r) = h \p(mr) = h\p(1_{X'}) = h \]
with the final equality proven above.  

Finally, suppose we have some other $k\colon Z \to X'$ so that $kf' = g, \rs{k} = \rs{g}$, and $\p(k) = h$.  Then $km$ is a precise lifting of $h\p(m)$ as
	\[ \p(km) = \p(k)\p(m) = h\p(m), \ \ kmf = kmrf' = kf' = g, \]
and
	\[ km\rs{f} = \rs{kmf}km = \rs{g}km = \rs{k}km = km. \]
Thus $km = \tilde{h}$, so $kmr = \tilde{h}r$, so $k = \tilde{h}r$.  Thus $f'$ is indeed $\p$-prone.  
\item If $f$ is a restriction retraction then $\p(f)$ is certainly a restriction retraction since $\p$ is a restriction functor.   For the converse, suppose $f\colon Y \to X$ is $\p$-prone with $\p(f)$ a restriction retraction, so that there is an $m\colon \p(X) \to \p(Y)$ with $m\p(f) = 1_{\p(X)}$ and $\rst{\p(f)} = \p(f)m$.  Then $f$ is a retraction with $\widetilde{m}$ its section:
\[ \xymatrix{X \ar@{=}[dr] \ar@{..>}[d]_{\widetilde{m}} & ~ \ar@{}[drr]|{\textstyle\mapsto}&&\p(X) \ar@{=}[dr] \ar[d]_{m} & ~
\\ 
Y \ar[r]_f & X && \p(Y) \ar[r]_{\p(f)} & \p(X)} \] 
so this means, by splitting $\rst{f}$, that $f$ can be factorized into $rf'$ where $r\colon Y \to Y'$ is the coequalizer of $\rst{f}$ and $1_Y$ and lies over the equalizer of $\p(f)$ and $1_{\p(X)}$ which means $f'$ is prone over an isomorphism.  However, $f'$ is also total as using the fact that $r$ is epic we have $r\rst{f'} = \rst{rf'}r =\rst{f}r =r = r1_{Y'}$.  It follows that $\widetilde{\p(f)^{-1}}f' = 1_{X}$ and so $f'$ is an isomorphism from part (i).    This means, in turn, that $f=rf'$ is a restriction retraction.
\end{enumerate}
\end{proof}

In the case when $\p$ is a restriction functor and all restriction idempotents in $\E$ split, Cartesian arrows give prone arrows:
\begin{lemma}\label{lemma:cart_are_prone}
Suppose that $\p\colon \E \to \B$ is a restriction functor, and all restriction idempotents in $\E$ split.  If $f\colon X \to Y$ a total map in $\E$ which is Cartesian for the functor ${\sf Total}(\p)\colon {\sf Total}(\E) \to {\sf Total}(\B)$, then $f$ is $\p$-prone.
\end{lemma}
\begin{proof}
Suppose we have
\[ \xymatrix{Z  \ar[dr]^{g} &~\ar@{}[drr]|{\textstyle\mapsto}&&  \p(Z) \ar[d]_h \ar[dr]^{\p(g)}  & ~
\\ 
X  \ar[r]_{f} & Y && \p(X) \ar[r]_{\p(f)} & \p(Y)} \]
with the triangle in $\B$ precise, so $\rs{h} = \rs{\p(g)}$.  Let $m\colon Z' \to Z$, $r\colon Z \to Z'$ be a splitting of $\rs{g}$, so $rm = \rs{g}$ and $mr = 1$.  Then in particular $mg\colon Z' \to Y$ is total.  Moreover, the pair $(\p(m),\p(r))$ gives a splitting of $\rs{\p(g)}$ and hence of $h$, so $\p(m)h$ is also total.  Thus, we have diagrams in ${\sf Total}(\E)$ and ${\sf Total}(\B)$, and hence get a unique total $h'\colon Z' \to X$:
\[ \xymatrix{Z'  \ar@{..>}[d]_{h'} \ar[dr]^{mg} & ~ \ar@{}[drr]|{\textstyle~~\mapsto}&&&\p(Z) \ar[d]_{\p(m)h} \ar[dr]^{\p(m)\p(g)}  & ~ 
\\ 
X  \ar[r]_{f} & Y &&& \p(X) \ar[r]_{\p(f)} & \p(Y)} \]
We now claim that the composite 
	\[ Z \to^{r} Z' \to^{h'} X \]
gives the required unique arrow in the first triangle.  Indeed, we have
	\[ rh'f = rmg = \rs{g}g = g, \]
and since $h'$ is total,
	\[ \rs{rh'} = \rs{r} = \rs{g}, \]
and
	\[ \p(rh') = \p(r)\p(m)h = \p(rm)h = \p(\rs{g})h = \rs{\p(g)}h = \rs{h}h = h, \]
so $rh'$ has all the required properties.  Moreover, if there is some other $k\colon Z \to X$ with these properties, then $mk$ is total (as $\rs{mk} = \rs{m\rs{g}} = \rs{mg} = 1$) and has the same other properties as $h'$, so $h' = mk$.  Thus, $rh' = rmk = \rs{g}k = \rs{k}k  = k$, so $rh'$ is unique.  
\end{proof}

It follows from the general theory of fibrations of 2-categories that latent fibrations behave well with respect to composition and pullback.  However, it is still helpful to see how this behaviour works out concretely. 

\begin{lemma}\label{lemma:prone_with_comp_functors}
  If $\q\colon \F \to \E$ and $\p\colon \E \to \B$ are restriction semifunctors, $f'$ is $\p$-prone over $f$, and $f''$ is  $\q$-prone over $f'$, then 
  $f''$ is $\q\p$-prone over $f$.  
\end{lemma}

\begin{proof}
This is given by a straightforward two step lifting:
\[ \xymatrix{ F \ar[dr]^{g} \ar@{..>}[d]_{\widetilde{\widetilde{h}}} & ~\ar@{}[drr]|{\textstyle\mapsto}&& \q(F) \ar[dr]^{\q(g)} \ar[d]_{\widetilde{h}} & ~ \ar@{}[drr]|{\textstyle\mapsto} && \p(\q(F)) \ar[dr]^{\p(\q(g))} \ar[d]_{h}
\\ 
B \ar[r]_{f''} & F && \q(B) \ar[r]_{f'} & \q(F') && \p(\q(B)) \ar[r]_f & \p(\q(F'))} \]
\end{proof}

\begin{corollary}
Latent fibrations are closed under composition.
\end{corollary}

Unfortunately, it is not the case that splitting the idempotents of a latent fibration $\p\colon \E \to \B$, to obtain $\split(\p)\colon \split(\E) \to \split(\B)$, will yield in general a latent fibration.  However, it is the case that splitting the idempotents of the total category will yield a latent fibration over the same base, $U \p\colon\split(\E) \to \B$, as we may precompose with the latent fibration $U\colon \split(\E)  \to \E$ of subsection \ref{splitting}.  Shortly we shall see that $\split(\p)$ is a latent fibration when $\p$  is an \emph{admissible} latent fibration (see Proposition \ref{splitting-latent-fibration}).  

Pullbacks in the category of restriction categories and restriction functors are described in \cite{journal:rcats1}, and work similarly for restriction semifunctors:
\begin{definition}
If $p\colon \E \to \B$ and $F\colon \X \to \B$ are restriction semifunctors, then their pullback is
\[\xymatrix{
\W \ar[r]^{p_1} \ar[d]_{p_0} & \E \ar[d]^{\p} \\
\X \ar[r]_F &  \B }\]
in which the category $\W$ is defined by
\begin{description}
\item {\bf objects} are pairs $(X,E)$ where $X$ and $E$ are objects of $\X$ and $\E$ respectively which satisfy $F(X) = {\sf p}(E)$.
\item {\bf maps} of type $(X,E) \to (X',E')$ are pairs $(f,g)$ where $f$ and $g$ are maps of $\X$ and $\E$ respectively which satisfy $F(f) = \p(g)$.
\item {\bf composition} and {\bf identities} are defined pointwise.
\item {\bf restriction} is given by $\rst{(f,g)} = (\rst{f},\rst{g})$, which is well defined since, if $F(f) = \p(g)$, then $F(\rst{f}) = \rst{F(f)} = \rst{\p(g)} = \p(\rst{g})$.
\end{description}
and the pullback maps $p_0,p_1$ are the first and second projections. 
\end{definition}

\begin{lemma}\label{lemma:pullback_prone}
If $p\colon \E \to \B$ and $F\colon \X \to \B$ are restriction semifunctors and $W$ is their pullback as defined above, then $(f,g)\colon (X,E) \to (X',E')$ in $\W$ is $\p_0$-prone if and only if $g$ is $\p$-prone.
\end{lemma}
\begin{proof}
Suppose we have 
\[  \xymatrix@C=3em{   (X'',E'') \ar[rd]^-{(f',g')}  & ~\ar@{}[drr]|{\textstyle\mapsto}&& X'' \ar[d]_h \ar[rd]^{f'} & 
\\  
(X,E) \ar[r]_-{(f,g)}& (X',E') &&   X  \ar[r]_f & X'}
\]
then since $g$ is $p$-prone and $F(f) = \p(g), F(f') = \p(g')$, etc., we also have
\[  \xymatrix@C=3em{   E'' \ar[rd]^-{g'}  \ar[d]_{\widetilde{h}} & ~\ar@{}[drr]|{\textstyle\mapsto} && F(X'') \ar[d]_{F(h)} \ar[rd]^{F(f')} & 
\\ E \ar[r]_-{g}& X' && F(X)  \ar[r]_{F(f)} & F(X')}
\]
so $(h, \tilde{h})$ uniquely fills in the first triangle.  Thus, $(f,g)$ is $\p_0$-prone; a similar proof shows the converse.  
\end{proof}

\begin{corollary}\label{prop:pullback_latent}
The pullback of a latent fibration along any restriction semifunctor is a latent fibration.
\end{corollary}

Recall that in an ordinary fibration, any map in the total category can be factored as a vertical map followed by a Cartesian map, and this factorization is unique up to a unique vertical isomorphism.  A similar result holds for latent fibrations, with vertical replaced by \emph{sub}vertical:

\begin{definition}
If $\p\colon \E \to \B$ is a restriction semifunctor, then a map $v\colon X \to Y$ in $\E$ is said to be \textbf{subvertical} if $\p(v)$ is a restriction idempotent.
\end{definition}

\begin{proposition}\label{prop:factorization}
If $\p\colon \E \to \B$ is a latent fibration, then for any map $f\colon X \to Y$ in $\E$, there is a subvertical map $v\colon X \to X'$ and a prone map $c\colon X' \to Y$ such that
\[ \xymatrix{ X \ar[d]_{v} \ar[dr]^{f} & \\ X' \ar[r]_{c} & Y } \]
is a precise triangle, and, moreover, $\rs{\p(c)} = \p(v)$.   Such a factorization is unique up to unique subvertical partial isomorphism.
\end{proposition}
\begin{proof}
For existence, let $c\colon X' \to Y$ be a prone arrow over $\p(f)\colon \p(X) \to \p(Y)$, and let $v$ be the induced unique map
\[ \xymatrix{X  \ar@{..>}[d]_{v} \ar[dr]^{f} & ~\ar@{}[drr]|{\textstyle\mapsto}&& \p(X) \ar[d]_{\rs{\p(f)}} \ar[dr]^{\p(f)}  & ~
\\ 
X'  \ar[r]_{c} & Y &&  \p(X') \ar[r]_{\p(f)} & \p(Y)} \]
This satisfies all the required conditions; uniqueness follows from using Lemma \ref{mediating_maps}.i.  
\end{proof}


\subsection{Substitution semifunctors}

Analogous to the fiber of an object in a fibration, is the notion of a \emph{strand} of an object in a latent fibration.  

\begin{definition}
Given a restriction functor $\p\colon \E \to \B$, the {\bf strand} of $B \in \B$, written ${\sf p}^{(-1)}(B)$, is the category whose objects are objects $E$ of $\E$ such that $\p(E) = B$, and whose maps are the $f$ of $\E$ such that $\p(f) \leq 1_B$ (i.e.  $\p(f) \in {\cal O}(B)$); that is, each $f$ is subvertical.  
\end{definition}

A latent fibration may be \emph{cloven} in the same sense as for ordinary fibrations:

\begin{definition}
A latent fibration {\bf has a cleavage} (or is {\bf cloven}) in case there is a chosen prone arrow $f^*_E \colon f^*(E) \to E$ over each $f \colon A \to \p(E)$ for each $E \in \E$.
\end{definition}

We now construct, for a cloven latent fibration, the analogue  of a reindexing or substitution functor.  In general, the substitution functors we obtain 
between the strands of a latent fibration will only be restriction semifunctors.  

Let ${\sf p} \colon \E \to \B$ be a cloven latent fibration, and let $u \colon A \to B$ be a map in $\B$. We define the \emph{substitution semifunctor} $u^* \colon {\sf p}^{(-1)}(B) \to {\sf p}^{(-1)}(A)$, from the strand above $B$ to the strand above $A$ as follows: on arrows $f \colon X \to Y$ in ${\sf p}^{(-1)}(B)$, $u^*(f)$ is the arrow $u^*(X) \to u^*(Y)$ in ${\sf p}^{(-1)}(A)$ given by the lifting 
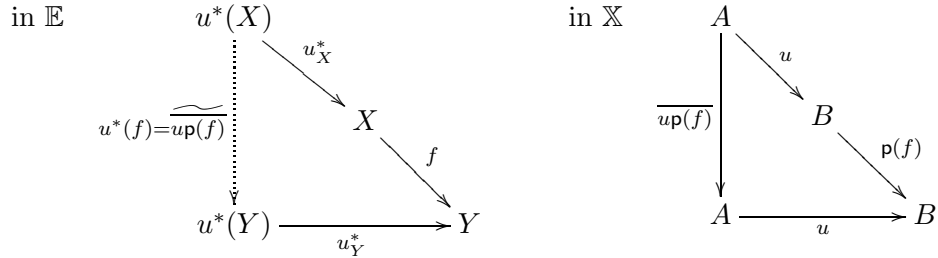
\[
\text{in $\E$}\hspace{10pt}
\xymatrix{
u^*(X) \ar@{.>}[dd]_{u^*(f)= \widetilde{\rst{u\p(f)}}} \ar[rd]^{u^*_X} \\
& X \ar[rd]^f \\
u^*(Y) \ar[rr]_{u^*_Y} && Y
}
\hspace{30pt}
\text{in $\X$}\hspace{10pt}
\xymatrix{
A \ar[dd]_{\rst{u\p(f)}} \ar[rd]^u && \\
& B \ar[rd]^{{\sf p}(f)} \\
A \ar[rr]_u  && B
}
\]
where $u^*(X)$ is the domain of the prone map above $u$ with codomain $X$.  Furthermore, if $u \leq v$ then there is for each $X \in \p^{(-1)}(B)$ a map 
$(u \leq v)^{*}_X\colon u^{*}(X) \to v^{*}(X)$ given by:
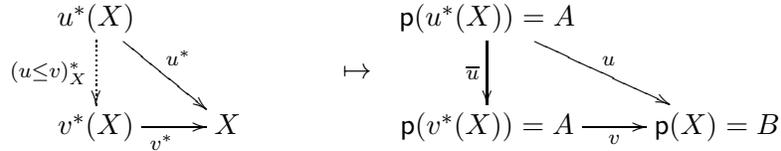
\[ \xymatrix{ u^{*}(X) \ar@{..>}[d]_{(u \leq v)^{*}_X} \ar[dr]^{u^{*}} & ~\ar@{}[drr]|{\textstyle\mapsto}&& \p(u^{*}(X))=A \ar[d]_{\rst{u}} \ar[dr]^-u
\\
 v^{*}(X) \ar[r]_{v^{*}} & X &&  \p(v^{*}(X))=A \ar[r]_-{v} & \p(X)=B} \] 
We then have:

\begin{proposition} \label{pseudo-functor}
If ${\sf p} \colon \E \to \B$ is a latent fibration with a cleavage and $u \colon A \to B$ is a map in $\X$, then $u^* \colon {\sf p}^{(-1)}(B) \to {\sf p}^{(-1)}(A)$ as defined above is a restriction semifunctor.  Furthermore, the assignment $(\_)^*\colon \B^{\rm op} \to {\sf SRest}$ is a pseudo 2-functor where we regard $\B$ as a 2-category whose 2-cells are given by the restriction ordering of maps.
\end{proposition}

Recall that a pseudo 2-functor $P\colon \B^{\rm op} \to {\sf SRest}$ associates to each map (1-cell)  $f\colon A \to B$ in $\B$ a functor (1-cell) $P(f)\colon P(B) \to P(A)$ in ${\sf SRest}$ such that there are (2-cell) natural isomorphisms $\alpha_{f,g}\colon P(g) P(f) \to P(fg)$ and $\alpha_X\colon 1_{P(X)} \to P(1_X)$ satisfying:
\[ \xymatrix@C=3em{P(f) \ar@{=}[dr] \ar[r]^-{P(f) \alpha_Y}  &P(f)P(1_Y) \ar[d]^{\alpha_{f,1_x}} \\ & P(f)} 
~~~~~ \xymatrix@C=3em{P(f) \ar@{=}[dr] \ar[r]^-{\alpha_X P(f)}  &P(1_X)P(f) \ar[d]^{\alpha_{f,1_x}} \\ & P(f)} 
~~~~~ \xymatrix@C=3em{P(h) P(g) P(f) \ar[d]_{P(h) \alpha_{f,g}} \ar[r]^-{\alpha_{g,h} P(f)} & P(gh) P(f) \ar[d]^{\alpha_{f,gh}} \\ P(h) P(fg) \ar[r]_-{\alpha_{h,fg}} & P(fgh)} \]
Similarly if $(f \leq g)\colon f \to g$ is a 2-cell in $\B$ then $P(f \leq g)\colon P(f) \to P(g)$ must be a 2-cell (or transformation) in ${\sf SRest}$.  This assignment must preserve the horizontal (1-cell)
and vertical (2-cell) composition.  Preserving the 2-cell composition is the  requirement that $P(f \leq g)P(g \leq h) = P(f \leq h)$ while preserving the 1-cell composition means
\[ \xymatrix{P(g);P(f) \ar[d]_{P(g\leq h);P(f\leq k)} \ar[r]^-{\alpha_{f,g}} & P(fg) \ar[d]^{P(fg \leq hk)}  \\ P(h);P(k) \ar[r]_-{\alpha_{h,k}} & P(hk)} \]
where the semicolon emphasizes that we are using the horizontal composition in ${\sf SRest}$.

Thus, for $P\colon \B^{\rm op} \to {\sf SRest}$ the inequalities (which are covariant) become transformations between semifunctors whose composites must be preserved.

\medskip
\begin{proof}
We must first show that $u^{*}$ is a semifunctor, that is $u^*(\rst{f}) = \rst{u^*(f)}$ and $u^*(fg) = u^*(f)u^*(g)$. 

First, suppose $f \colon X \to Y$ is a map in ${\sf p}^{(-1)}(B)$. Then $u^*(\rst{f})$ is defined as the lifting of $\rst{u \p(\rst{f})}$.  However, $\rst{u^*_X(\rst{f})}$ also makes 
the triangle precise and $\p( \rst{u^*_X(\rst{f})}) = \rst{\p(u^*_X(\rst{f}))} = \rst{u\p(\rst{f})}$ so $u^*(\rst{f})= \rst{u^*(\rst{f})}$ and, thus $u^*$ preserves the restriction.   

Next, suppose $f \colon X \to Y$ and $g \colon Y \to Z$ are maps of ${\sf p}^{-1}(B)$. Then $u^*(fg)$ is defined as the lifting of $\rst{u\p(fg)} =\rst{u\p(f)\p(g)}$  while $u^*(f)u^*(g)$ is the lifting of 
$\p(u^*(f)u^*(g)) = \p(u^*(f)) \p(u^*(g)) = \rst{u\p(f)}~\rst{u\p(g)} = \rst{\rst{u\p(f)}u \p(g)} = \rst{u\rst{p(f)}\p(g)} = \rst{u\p(f)\p(g)}$ so that they are equal.
Thus $u^*$ preserves composition, and is therefore a restriction semifunctor.

To show that $(\_)^*$ is a pseudo functor we use the fact that composites of prone maps are unique up to a unique mediating partial isomorphism  (combine Lemma \ref{composites-prones} with Lemma \ref{mediating_maps}) and the unit data is provided by the fact that the identity map $1_X$ is prone so there is a mediating map $X \to 1_{\p(X)}^*(X)$.

We must show $\rst{(u\leq v)^{*}_X} = u^*(1_X)$ and that the transformation is natural; that is, $u^*(f) (u\leq v)^{*}_Y = (u\leq v)^{*}_X v^*(f)$ (so $(u \leq v)^*_f$ is the identity).   The first requirement is immediate from the 
preciseness of the triangle defining $(u \leq v)^{*}_X$ as $u^{*}(1_X) = \rst{u^*_X} = \rst{(u \leq v)^{*}_X}$.  Suppose for the second that $f\colon X \to Y$ in the strand over $B$ then 
$\p(u^*(f) (u\leq v)^{*}_Y) = \rst{u\p(f)}~\rst{u} = \rst{u\p(f)}$ and $\p((u\leq v)_X v^*(f)) = \rst{u} ~\rst{vp(f)} = \rst{\rst{u}v\p(f)} = \rst{u\p(f)}$ so both arrows sit above $\rst{u\p(f)} \in \B$.  
As all the components are precise liftings their composites are as well; thus, they are equal.  This shows that the transformation is natural as required.

The assignment is immediately functorial with respect to the restriction preorder as the components of the transformations lift restriction idempotents which compose.

The coherence properties with the associator are similarly straightforward to prove.
\end{proof}

Recall that in an $r$-split latent fibration $\p \colon \E \to \B$, whenever $f$ is total we know that there is a total prone map above it (see Lemma \ref{totals-in-flush-fibration}).   If the cleavage is such that, for every $X$, $f^*_X$ is total whenever $f$ is total,  then substitution semifunctors along total maps become restriction functors.

\begin{proposition}
If ${\sf p} \colon \E \to \B$ is a latent fibration with a cleavage which preserves total maps then $f^* \colon {\sf p}^{(-1)}(B) \to {\sf p}^{(-1)}(A)$ is a restriction functor whenever $f$ is total.
\end{proposition}
\begin{proof}
We have already shown that $f^*$ is a restriction semifunctor. Since the cleavage of ${\sf p}$ preserves total maps, we have that $f^*_X \colon u^*(X) \to X$ is total, when $f$ is. This gives
\begin{align*}
& f^*(1_X) = f^*(\rst{1_X}) = \rst{f^*(1_X)} = \rst{f^*(1_X)\rst{f^*_X}} \\
& = \rst{f^*(1_X)f^*_X} = \rst{f^*_X} = 1_{f^*(X)}
\end{align*}
and so $f^*$ is a restriction functor.
\end{proof}


\section{Types of Latent Fibrations}\label{sec:types}

In the previous section we described some properties of latent fibrations.  However, it is important to note two things that do \emph{not} hold in an arbitrary latent fibration $\p\colon \E \to \B$.  Both involve the behaviour of restriction idempotents.  

The first is that a restriction idempotent in the base need not lift as one might expect: that is, given a restriction idempotent $e\colon A \to A$ in $\B$ and an object $X$ over $A$ in $\E$, there need not be a prone restriction idempotent $e'\colon X \to X$ over $e$.   For an example of this, consider the latent fibration of propositions of a restriction category $\X$ (Example \ref{def:propositions}).  Suppose $e\colon A \to A$ is a restriction idempotent in $\X$, and $(A,e')$ is an object in $\O(\X)$ over $A$.  Since the projection in this case simply sends a map to itself, a restriction idempotent over $e$ must be $e$ itself; thus, we would need $e$ to be a map from $(X,e')$ to $(X,e')$ in $\O(\X)$.  But this would mean that $e = e'e$, which requires $e \leq e'$.  Thus, for any restriction idempotent $e'$ which is not $\geq$ e, this is not possible.  Of course, being a latent fibration, there is \emph{a} lift of $e$: as per Proposition \ref{prop:propositions}, it is the map
	\[ (X,ee') \to^{e} (X,e'), \]
it is just that this map is not a restriction idempotent in $\O(\X)$ (in particular, it is not an endomorphism!)  

Thus, in this case, there isn't even a restriction idempotent over $e$ (let alone a prone restriction idempotent).  As we shall see, however, the ability to lift restriction idempotents to prone restriction idempotents is useful, and is true for most latent fibrations.  Thus,  in the next section we consider such latent fibrations; we term these \emph{admissible}. 

The second issue is the behaviour of restriction idempotents in the total category of a latent fibration.   In particular, while identities are always prone, restriction idempotents (i.e., ``partial identities'') need not be.  To see the problem in general, suppose $e\colon X \to X$ is a restriction idempotent in $\E$, and suppose we have $g\colon Y \to X$ in $\E$ and $h\colon B \to A$ in $\B$ such that $h\p(e) = \p(g)$: 
\[ \xymatrix{Y \ar@{..>}[d] \ar[dr]^{g} & ~\ar@{}[drr]|{\textstyle\mapsto}&&A \ar[dr]^{\p(g)} \ar[d]_h \\ X \ar[r]_{e} & X &&  A \ar[r]_{\p(e)} & A} \]
If $e$ were the identity, we could obviously choose $g$ as the unique lift.  However, if $e$ is partial, there is no obvious lift that will make the triangle in $\E$ commute.

However, note that we do have some control given that the commutativity of the bottom triangle tells us something about how defined $g$ is relative to $e$.  Thus, if the restriction idempotents in $\E$ were closely related to those in $\B$, then we could hope to have a lift.  In fact, this is the case in the ``strict'' versions of the simple and codomain fibrations.  In these examples, one can check that restriction idempotents are prone (while they are not generally in the ``lax'' versions).  

Thus, restriction idempotents being prone is clearly an important condition that is true in some latent fibrations but not all.  We shall see that this condition is equivalent to $\p$ being monic on restriction idempotents, or \emph{separated}.  Thus, in Section \ref{sec:separated} we consider the theory and examples of separated latent fibrations.

Many examples (including the motivating example of the strict simple fibration) are both admissible and separated.   The combination of these conditions has some very useful consequences (for example, see Proposition \ref{prop:hyper_consequences}).  Moreover, we will see that the combination of the admissible and separated conditions is equivalent to the restriction semifunctor $\p$ being \emph{hyperconnected}; that is, $\p$ is a bijection on restriction idempotents.   

Thus, in the next few sections, we consider the theory and examples of admissible, separated, and hyperconnected latent fibrations.  


\subsection{Admissible latent fibrations}\label{sec:admissible}

\begin{definition} A restriction semifunctor ${\sf p} \colon \E \to \B$ is {\bf admissible} if for every $X \in \E$ and every restriction idempotent, $e$ on ${\sf p}(X)$ in the base such that $e\p(1_X) = e$,  there is a prone restriction idempotent $e^{*}$ on $X$ over $e$, that is with ${\sf p}(e^{*}) = e$.
\end{definition}

\begin{example} Most latent fibrations are admissible:
\begin{enumerate}[(i)]
\item The identity $1_\X\colon \X\to\X$ is admissible.
	\item The lax and strict simple fibrations are admissible, as the prone lifting given in Proposition \ref{prop:simpleSlice} is a restriction idempotent.
	\item The lax and strict codomain fibrations are admissable; for proof, see the result below.  
	\item The product latent fibration is admissible; the prone lifting $(1,e)$ of a restriction idempotent $e$ is again a restriction idempotent.  
	\item The assemblies fibration is admissible; given a restriction idempotent $e\colon X \to X$ and an assembly $\phi$ over $X$, it is easy to check that $e$ (with identity tracking map) is prone in the assemblies category.   
	\item The splitting fibration $\split(\X) \to \X$ is admissible.  If we are  given an object $(X,e')$ in $\split(\X)$ and a restriction idempotent $e\colon X \to X$ such that $e = e\p(1_{(X,e')}) = ee'$, then
		\[ (X,e') \to^{e} (X,e') \]
	is a well-defined map in $\split(\X)$, and it is straightforward to check that it is prone (since this fibration is separated, this also follows from Proposition \ref{prop:separated_equivalences}).  
	\item As noted in the introduction to this section, the latent fibration of propositions is \textbf{not} generally admissible.  Nor is a discrete fibration in general admissible: the prone arrow above a restriction idempotent is not necessarily an endomorphism unless the point is fixed by the function induced by the idempotent.
\end{enumerate}
\end{example}

\begin{proposition}
For any restriction category $\X$ with latent pullbacks, the latent fibrations $\X^{\leadsto}$ and $\X^{\rightarrow}$ are admissible. 
\end{proposition}
\begin{proof}
Given a restriction idempotent $e\colon Y \to Y$ and a map $f\colon X \to Y$ in $X$, we need to define a prone lifting of $e$ to $\X^{\leadsto}$ which is also a restriction idempotent.  We claim that the pair $(\rs{fe}, e)$ does this.  Indeed, consider the diagram
 \[     \xymatrix{C' \ar[dd]_c \ar[drr]^{g'} \ar@{..>}[dr]_{\widetilde{w}} \\
                                  & X \ar[dd]^{f} \ar[r]_{\rst{fe}} & X \ar[dd]^{f} \\
                                  C \ar@{}[ur]|\geq \ar[rrd]^<<<<g|\hole \ar[dr]_w \\
                                  & Y\ar[r]_e  & Y} \]
Since $g = we$ is precise, by Lemma \ref{Jaws}.iii, $g = w$.  Similarly, for the top triangle to be precise, we must have $\widetilde{w} = g'$.  It remains to check that $\widetilde{w} = g'$ satisfies the required conditions.

Indeed, it gives a lax square since the outer square is lax by assumption:
	\[ g'f \leq cg = cwe = cw, \]
and the top triangle commutes since
	\[ g'\rs{fe} = \rs{g'fe}g' = \rs{chee} g' = \rs{che}g' = \rs{g'f} = g'\rs{f} = g' \]
with the last equality holding since the outer square is assumed precise.

Thus $(fe,e)$ is prone, and so $\X^{\leadsto}$ is admissible; a similar proof shows that $\X^{\rightarrow}$ is admissible. 
\end{proof}

The definition of admissible only asks for the existence of a prone restriction idempotent over a restriction idempotent in the base.  However, by Lemma \ref{mediating_maps}.iii, such a restriction idempotent is unique; thus, one in fact gets a section:

\begin{lemma} \label{reflection-fibration}
If ${\sf p}\colon \E \to \B$ is admissible, then the semilattice map 
	\[ {\sf p}|_{{\cal O}(X)}\colon {\cal O}(X) \to \{e \in \O(\p(X))\colon e \leq \p(1_X)\} \]
has a section $(\_)^*$, where $e^*$ is the (unique) prone restriction idempotent above $e$ at $X$.  Furthermore, for any $d \in \O(X)$, $d \leq {\sf p}(d)^{*}$, so that the prone arrows are a reflective subposet  of ${\cal O}(X)$.
\end{lemma}

\begin{proof}
Note that $(e_1 e_2)^{*} = e_1^{*} e_2^{*}$, as both are prone over $e_1e_2$, and $\p(1_X)^{*} = 1_X$, as $1_X$ is trivially prone.  Thus, the section is also a semilattice morphism.  Moreover, ${\sf p}(d) {\sf p}(d) = {\sf p}(d)$ is a trivial precise triangle with $d {\sf p}(d)^{*} = e$ precise above it: but this implies $d \leq {\sf p}(d)^{*}$ as required.
\end{proof}

Another useful property of admissible latent fibrations is:

\begin{lemma} \label{restriction-monics-in-fibrations}
If $\p\colon \E \to \B$ is admissible and $m\colon E' \to E$ is prone over a restriction monic, then $m$ is a partial isomorphism.    
\end{lemma}

\begin{proof}
If $m\colon E' \to E$ is over a restriction monic, then there is a map $r\colon \p(E) \to \p(E')$ such that $r\p(m) = \rs{r}$ and $\p(m)r = 1_{\p(E')}$.  By admissibility, there is a prone restriction idempotent $\rs{r}^*\colon E \to E$ over $\rs{r}$.  Let $r'$ be defined as the unique lift
\[ \xymatrix{ E \ar@{..>}[d]_{r'} \ar[dr]^{\rs{r}^*} & ~\ar@{}[drr]|{\textstyle\mapsto}&& \p(E)  \ar[d]_{r} \ar[dr]^{\rs{r}} & ~  
\\
E' \ar[r]_m & E &&\p(E') \ar[r]_{\p(m)} & \p(E) } \]
Thus, $r'm = \rs{r}^* = \rs{\rs{r}^*} = \rs{r'}$.  Thus, $r'$ satisfies one half of being a partial inverse to $m$; we also need $mr' = \rs{m}$. 

For this, first, since $\rs{r}^*$ is itself prone, we have some unique $k\colon E' \to E$ such that
\[ \xymatrix{ E' \ar@{..>}[d]_{k} \ar[dr]^{m} & ~\ar@{}[drr]|{\textstyle\mapsto}&& \p(E')  \ar[d]_{\p(m)} \ar[dr]^{\p(m)} & ~  
\\
E \ar[r]_{\rs{r}^*} & E && \p(E) \ar[r]_{\rs{r}} & \p(E) } \]
However, by Lemma \ref{Jaws}.ii, $k = m$.  Thus, in particular, since the top triangle is precise, we have $m\rs{r}^* = m$, so $mr'm = m\rs{r}^* = m$.  Thus we have the triangle
\[ \xymatrix{ E' \ar[d]_{mr'} \ar[dr]^{m} & ~ \ar@{}[drr]|{\textstyle\mapsto}&& \p(E')  \ar[d]_{1} \ar[dr]^{\p(m)} & ~
\\ 
E' \ar[r]_{m} & E && \p(E') \ar[r]_{\p(m)} & \p(E) } \]
But $\rs{m}$ also fits precisely into the top triangle in place of $mr'$ (and is over $1$ since $\p(m)$ is total), so since $m$ is prone, $mr' = \rs{m}$.  Thus $m$ is a partial isomorphism with partial inverse $r'$.
\end{proof}

\begin{lemma} \label{nearly_connected}
If ${\sf p}\colon \E \to \B$ is an admissible latent fibration then any mediating partial isomorphism $\alpha$ between two ${\sf p}$-prone maps $f'$ and $f$ over the same map such that 
$\rst{f} = \rst{{\sf p}(f)}^*$ is itself prone .
\end{lemma}

\begin{proof}
Consider maps $f$ and $f'$ which are both ${\sf p}$-prone above the same map, ${\sf p}(f) = {\sf p}(f')$.  Let $\alpha$ be the mediating map between $f$ and $f'$
so in particular ${\sf p}(\alpha) = \rst{{\sf p}(f)}$.  Now consider $g\colon C \to B$ with ${\sf p}(g) = h \rst{{\sf p}(f)}$.  By Lemma \ref{Jaws}.ii, it follows that $h = {\sf p}(g)$.   We wish to show that 
there is a lifting $\widetilde{h}$ so that $\widetilde{h}\alpha = g$ is precise, however, the best we can do is to lift $h$, using the fact that $f$ is prone, making $\widetilde{h}f' = gf$:
\[ \xymatrix{C \ar[dr]_g  \ar@{..>}[r]^{\widetilde{h}} & B' \ar[d]^{\alpha} \ar[dr]^{f'} & ~\ar@{}[drr]|{\textstyle\mapsto}&&{\sf p}(C) \ar[dr]_{{\sf p}(g)}  \ar[r]^{h} & {\sf p}(B) \ar[d]|{{\sf p}(\alpha) =\rst{{\sf p}(f)}} \ar[dr]^{{\sf p}(f)}
\\
 & B \ar[r]_f & A &&& {\sf p}(B) \ar[r]_{{\sf p}(f)} & {\sf p}(A) } \]
Now we observe that $\widetilde{h} \alpha$ sits over ${\sf p}(g)$ as ${\sf p}(\widetilde{h} \alpha) = {\sf p}(\widetilde{h}){\sf p}(\alpha) = h \rst{{\sf p}(f)} = {\sf p}(g)$; furthermore, 
$\widetilde{h} \alpha f = \widetilde{h} f' = g f$ and $\widetilde{h} \alpha \rst{f} = \widetilde{h} \alpha$ so the triangle $(\widetilde{h}\alpha) f = (\widetilde{h} f')$ is precise.
However, as $g\rst{f} = g \rst{{\sf p}(f)}^*$ as ${\sf p}(g) ={\sf p}(g) \rst{{\sf p}(f)}$, the triangle $g f = (\widetilde{h} f')$ is precise.  So we may conclude 
that $g = \widetilde{h}\alpha$ and, as $\widetilde{h}\rst{\alpha} = \widetilde{h}\rst{f'}$, it is precise.  The triangle is, furthermore, unique as $\alpha$ is a partial isomorphism, and thus $\alpha$ is prone.
\end{proof}

\begin{proposition}
Admissible latent fibrations are closed under composition and pullback.
\end{proposition}
\begin{proof}
This follows from Lemma \ref{lemma:prone_with_comp_functors} and Lemma \ref{lemma:pullback_prone}.
\end{proof}

Perhaps the most useful property of admissible latent fibrations is that they can be $r$-split; we prove this in Proposition \ref{splitting-latent-fibration}.  


\subsection{Separated latent fibrations}\label{sec:separated}

We now consider a different property a restriction semifunctor can have.  

\begin{definition}
A restriction semifunctor $\p\colon \E \to \B$ is said to be \textbf{separated} if for any object $X \in \E$ and restriction idempotents $e,e'$ on $X$, if $\p(e) = \p(e')$ then $e = e'$.
\end{definition}

\begin{example}\label{separated_examples}
\begin{enumerate}[(i)]
\item The identity fibration $1_\X\colon \X\to\X$ is separated.
	\item The strict simple fibration $\X[\X]$ is separated: since the restriction of a map in the total category is entirely determined by its restriction in the base, the functor $\partial\colon \X[\X]\to\X$ is separated.  
	\item Conversely, in general the lax simple fibration $\X(\X)$ is not separated, as the inequality allowed in the definition of an arrow permits multiple possible restriction idempotents over a fixed one in the base.  Thus, this is an example of a latent fibration which is admissible but not separated.  
    \item For any restriction category $\X$ with latent pullbacks, $\X^{\rightarrow} \to \X$ is separated: if
	\[ \xymatrix{ X \ar[r]^{e'} \ar[d]_{f} & X \ar[d]^{f} \\ Y \ar[r]_{e} & Y}	\]
	is a restriction idempotent in $\X^{\rightarrow}$, then since the square must strictly commute and be precise, we have
		\[ e' = \rs{e'} = \rs{e'\rs{f}} = \rs{e'f} = \rs{fe}. \]
    Thus, the restriction idempotent $(e',e)$ is entirely determined by $e$, and so the functor $\partial\colon \X^{\rightarrow} \to \X$ is a separated.
    \item By lemma \ref{latent-pullbacks}.ii, one may be tempted to think that $\X^{\leadsto}$ is separated; however, for this to be the case partially inverting the top and bottom 
arrows of latent pullback would have to produce a semi-precise square: this is not the case in general.  
     \item The propositions fibration is clearly separated, as in this case $\p(e) = e$.  Thus, this is an example of a latent fibration which is separated but not admissible.  
     \item For assemblies, notice that all restriction idempotents $e \in \X$ are tracked by identity realizers and so every restriction idempotent is tracked.  This means  that the restriction 
	idempotents in ${\sf Asm}(F)$ are the same as those in $\X$ making $\p$ separated.
     \item The splitting fibration $\split(\X) \to \X$ is clearly separated, as in this case $\p(e) = e$.  Similarly, a discrete fibration is always separated.
\end{enumerate}
\end{example}

An immediate property of a separated semifunctor is the following:
\begin{lemma}\label{lemma:separated_reflects_total}
If $\p\colon \E \to \B$ is separated, then $\p(f)$ total implies $f$ is total.
\end{lemma}
\begin{proof}
If we have such an $f$, then we also have
	\[ 1_{\p(X)} = \rs{\p(f)} = \p(\rs{f}) = \p(1_X)\p(\rs{f}) \leq \p(1_X), \]
but we always have the opposite inequality, so $\p(1_X) = 1_{\p(X)}$.  Then $\p(\rs{f}) = 1_{\p(X)} = \p(1_X)$, so since $p$ is separated, $\rs{f} = 1_X$.  Thus $f$ is total.
\end{proof}

More importantly, however, being separated is precisely the right condition to assure that all restriction idempotents (and, more generally, partial ismorphisms) are prone.

\begin{proposition}\label{prop:separated_equivalences}
For any restriction semifunctor $\p\colon \E \to \B$, the following are equivalent:
\begin{enumerate}[(i)]
	\item $\p$ is separated;
	\item all partial isomorphisms in $\E$ are $\p$-prone;
	\item all restriction idempotents in $\E$ are $\p$-prone.
\end{enumerate}
\end{proposition}
\begin{proof}
For (i) $\Rightarrow$ (ii), suppose that $\p$ is separated.  Suppose we have a partial isomorphism $f\colon X \to Y$ (with partial inverse $f^{(-1)}$) and a map $g\colon Z \to Y$ such that $h\p(f) = \p(g)$ is precise:
\[ \xymatrix{Z \ar@{..>}[d]^{} \ar[dr]^{g} & ~\ar@{}[drr]|{\textstyle\mapsto}&& \p(Z) \ar[d]_h \ar[dr]^{\p(g)}  & ~ 
\\
 X  \ar[r]_{f} & Y && \p(X) \ar[r]_{\p(f)} & \p(Y)} \]
We want to show that $gf^{(-1)}$ precisely fits in the left triangle (it is the only possible choice by Lemma \ref{Jaws}.iii).  First, note that 
	\[ \p(\rs{gf^{(-1)}}) = \rs{h\p(f)\rs{\p(f^{(-1)})}} = \rs{h\p(f)} = \rs{\p(g)} = \p(\rs{g}) \]
so since $\p$ is separated, $\rs{gf^{(-1)}} = \rs{g}$.  Then the triangle commutes since
	\[ gf^{(-1)}f = g\rs{f^{(-1)}} = \rs{gf^{(-1)}}g = \rs{g}g, \]
and is precise.  Moreover, $gf^{(-1)}$ is over $h$ since
	\[ \p(gf^{(-1)}) = h\p(f)\p(f^{(-1)} = h\rs{\p(f)} = h \]
since $h\p(f) = \p(g)$ is precise. \\

(ii) $\Rightarrow$ (iii) is immediate, since restriction idempotents are partial isomorphisms.\\

For (iii) $\Rightarrow$ (i), if all restriction idempotents are prone, then $\p$ is separated by Lemma \ref{mediating_maps}.iii.  

\end{proof}

When a restriction semifunctor is separated, the prone condition is a bit simpler to check:

\begin{lemma}\label{lemma:prone_for_separated}
If $\p\colon \E \to \B$ is separated, then an arrow $f\colon X \to X'$ in $\E$ is prone if and only if for any $g\colon Y \to X'$ and any $h\colon p(Y) \to p(X)$ such that $hp(f) = p(g)$ is precise, there is a unique $h'\colon Y \to X$ such that $h'f = g$.  
\end{lemma}
\begin{proof}
The only change in the ordinary definition is the requirement that the triangle in $\E$ be precise.  However, if $h'f = g$, then $\rs{p(h')} = \rs{h} = \rs{p(g)}$, so since $\p$ is separated, $\rs{h'} = \rs{g}$ is guaranteed.
\end{proof}

These functors also have a useful factorization property of prone arrows:

\begin{lemma}\label{lemma:left_factor_separated}
Suppose $\p\colon \E \to \B$ is separated.  If 
\[	 \xymatrix{X \ar[dr]^{f} \ar[d]_{f_1} &  \\ X' \ar[r]_{f_2} & X''} \]
is a precise triangle in $\E$ and $f$ and $f_2$ are prone, then $f_1$ is prone. 
\end{lemma}
\begin{proof}
Throughout this proof we will use the fact that when looking at whether an arrow is prone for a separated semifunctor, we don't need to consider precise-ness of the top triangle (see Lemma \ref{lemma:prone_for_separated}).  

Thus, we need to show that for $g\colon Y \to X'$ with $h\colon p(Y) \to p(X)$ such that $hp(f_1) = p(g)$ is precise, we have a unique fill-in $h'\colon Y \to X$.  Consider this diagram, extended to the right by post-composing with $f_2$:
\[ \xymatrix{Y \ar@{..>}[d]_{h'}  \ar[dr]^{g} \ar@/^1pc/[drr]^{gf_2} & &\ar@{}[drr]|{\textstyle\mapsto}&& \p(Y) \ar[dr]^{\p(g)} \ar[d]_{h} \ar@/^1pc/[drr]^{\p(gf_2)} & &
\\
X \ar[r]_{f_1} & X' \ar[r]_{f_2} & X'' &&\ \p(X) \ar[r]_{\p(f_1)} & \p(X') \ar[r]_{\p(f_2)} & \p(X')}
\]
Now, the outer triangle in the base is also precise since
	\[ \rs{\p(gf_2)} = \rs{h\p(f_1)\rs{\p(f_2)}} = \rs{h\p(f_1)} = \rs{\p(g)} = \rs{h} \]
where we have used the fact that $f = f_1f_2$ is precise to give $f_1\rs{f_2} = f_1$.

Thus, since $f = f_1f_2$ is prone, there is a unique $h'\colon Y \to X$ such that $\rs{h'} = \rs{gf_2}$ and $h'f_1f_2 = gf_2$.  We need to show that $h'$ makes the inner triangle commute; that is, we need $h'f = g$.

However, $\p(h'f_1) = h\p(f_1) = \p(g)$, so since we have $h'f_1f_2 = gf_2$ and $f_2$ is prone, $h'f = g$.  Uniqueness of $h'$ then follows by uniqueness of factorizations for $f_1f_2$.  
\end{proof}

\begin{corollary}\label{cor:res_monic_prone}
If $\p\colon \E \to \B$ is separated, then any restriction monic in $\E$ is prone.
\end{corollary}
\begin{proof}
Suppose we have a restriction monic $m\colon X \to Y$ in $\E$.  Then by assumption there is a restriction retraction $r\colon Y \to X$ so that $mr = 1$; thus,
	\[ \xymatrix{X \ar[dr]^{1} \ar[d]_{m} & \\ Y \ar[r]_{r} & X} \]
is a precise triangle.  By Lemma \ref{restriction-retraction}, identities and restriction retractions are always prone, so by the previous result, $m$ is also prone.
\end{proof}

Separated latent fibrations have the useful property that one can (latently) pullback a prone arrow along a subvertical arrow:

\begin{lemma}\label{lemma:pullback_prone_subvertical}
  Suppose $\p \colon \E \to \B$ is a separated latent fibration.  Then every cospan 
  	\[ B \to^{c} C \from^{v} A \]
with $v$ subvertical and $c$ prone has a corresponding   latent pullback: 
  \[
    \xymatrix{
      U \ar[r]^{c'} \ar[r] \ar[d]_{w} & A \ar[d]^{v} \\
      B \ar[r]_{c} \ar[r] & C
    }
  \]
  where $w$ is subvertical and $c'$ is prone.
\end{lemma}
\begin{proof}
Since $v$ is subvertical, $p(A) = p(C)$, and so the composite
	\[ \p(B) \to^{\p(c)} \p(A) \to^{\p(v)} \p(C) \]
is well-defined; we take $c'\colon U \to A$ to be a prone lift of this map to $A$.

Then let $w\colon U \to B$ be the unique arrow from the universal property of $c$:
\[ \xymatrix{ U \ar@{..>}[d]_{w} \ar[dr]^{c'v} & ~\ar@{}[drr]|{\textstyle\mapsto}&&& \p(U) \ar[dr]^{\p(c'v)} \ar[d]_{\rs{\p(c)\p(v)}} 
\\ 
B \ar[r]_{c}  & C &&&\p(B) \ar[r]_{\p(c)} & \p(C)} \]
(Note that the triangle in $\B$ is precise since $\rs{\p(c'v)} = \rs{\p(c)\p(v)\p(v)} = \rs{\p(c)\p(v)}$).  

Thus, $w$ as above exists, and has the property that $\rs{w} = \rs{c'v}$.  Moreover, since $\rs{\p(w)} = \rs{\p(c)\p(v)} = \rs{\p(c')}$, $\rs{w} = \rs{c'}$ as $\p$ is separated. Thus, the square is a potential candidate to be a latent pullback.

We will check that the original definition (see \cite[pg. 460]{journal:guo-range-join}) of latent pullback holds, as it involves one fewer piece of data.   So, assume we have $a\colon X \to A$ and $b\colon X \to B$ so that $av = bc$.  We need to find a unique $b'$ as indicated:
\[ \xymatrix{X \ar@/^0.7pc/[drr]^{a}_{\leq}  \ar@/^-0.7pc/[ddr]_{b}^{\geq} \ar@{-->}[dr]|{b'} \\
                    & U \ar[d]_w \ar[r]_{c'} & A \ar[d]^{v} \\  & B \ar[r]_c & C} \]
with the additional properties that $\rs{b'} = \rs{bc} = \rs{av}$ and $b'\rs{w} = b'\rs{c'} = b'$.  

Let $b'$ be the unique arrow arising from the universal property of $c'$:
\[ \xymatrix{ X \ar@{..>}[d]_{b'} \ar[dr]^{a\rs{v}} & ~\ar@{}[drr]|{\textstyle\mapsto}&&& \p(X) \ar[dr]^{\p(a\rs{v})} \ar[d]_{\p(b)\p(w)} 
\\ 
U \ar[r]_{c'}  & A &&& \p(U) \ar[r]_{\p(c)\p(v)} & \p(C)} \]
Note that $\p(b)\p(w)$ is well-defined since $w$ is subvertical.  The triangle in $\B$ commutes since
	\[ \p(b)\p(w)\p(c)\p(v) = \p(b)\rs{\p(c)\p(v)}\p(c)\p(v) = \p(b)\p(c)\p(v) = \p(a)\p(v)\p(v) = \p(a)\p(v) = \p(a\rs{v}) \]
and is precise since
	\[ \rs{\p(b)\p(w)} = \rs{\p(b)\rs{\p(c)\p(v)}} = \rs{\p(b)\p(c)\p(v)} = \rs{\p(a)\p(v)\p(v)} = \rs{\p(a)\p(v)}. \]
Thus, such a $b'$ does exist, and has the property that $\rs{b'} = \rs{a\rs{v}} = \rs{av}$.  

It immediately satisfies many of the required properties: 
	\[ b'c' = a\rs{v} \leq a, \]
and
	\[ \rs{b'} = \rs{av} \]
and 
	\[ b'\rs{c'} = \rs{b'c'}b' = \rs{av}b' = \rs{b'}b' = b'. \]
	
However, checking that $b'w \leq b$, that is, $\rs{b'w}b = b'w$, takes a bit more work.  For this, we will use the cancellation property of the prone map $c$: if $\rs{b'w}bc = b'wc$, $\p(\rs{b'w}b) = \p(b'w)$, and the relevant triangle in the base is precise, then $\rs{b'w}b = b'w$ follows.

For the first requirement, $b'wc = b'c'v = a\rs{v}v = av$ while $\rs{b'w}bc = \rs{b'w}av$.  However, we also have
	\[ \p(\rs{b'w}) = \rs{\p(b')\p(w)} = \rs{\p(b)\p(w)\p(w)} = \rs{\p(b)\p(w)} = \rs{\p(a)\p(v)} \]
by above; thus since $\p$ is separated we get $\rs{b'w}av = \rs{av}av = av$.  For the second requirement,
	\[ \p(\rs{b'w}b) = \rs{\p(b)\p(w)} \p(b) = \p(b)\rs{p(w)} = \p(b)\p(w)\p(w) = \p(b'w). \]
For preciseness of the triangle in the base,
	\[ \rs{\p(b'w)} = \rs{\p(b)\p(w)} = \rs{\p(a)\p(v)} \]
by above.  Thus, by proneness of $c$, $\rs{b'w}b = b'w$, and so $b'w \leq b$.

Thus $b'$ satisfies all the required properties; it remains to show that it is unique.  Thus, suppose we have 
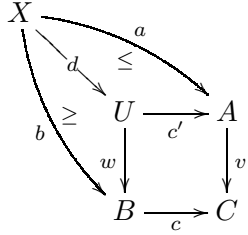
\[ \xymatrix{X \ar@/^0.7pc/[drr]^{a}_{\leq}  \ar@/^-0.7pc/[ddr]_{b}^{\geq} \ar[dr]|{d} \\
                    & U \ar[d]_w \ar[r]_{c'} & A \ar[d]^{v} \\  & B \ar[r]_c & C} \]
with $\rs{d} = \rs{bc} = \rs{av}$ and $d = d\rs{c'} = d\rs{w}$.  

We want to show that $d = b'$; we will use the universal property of $c'$.  Thus, we need to show that $dc' = b'c'$, $\p(d) = \p(b')$, and the relevant triangle in the base is precise.  For the first requirement, since $dc' \leq a$,
	\[ dc' = \rs{dc'}a = \rs{d}a = \rs{av}a = a\rs{v} = b'c'. \]
For the second requirement, since $d = d\rs{w}$, and $dw \leq b$,
	\[ \p(d) = \p(dw)\p(w) = \p(\rs{dw}b)\p(w) = \p(\rs{bc}b)\p(w) = \p(b\rs{c})\p(w) = \p(b)\p(w) = \p(b'), \]
with the second-last equality since $w\rs{c} = w$.  Finally, since $\rs{d} = \rs{av}$, the relevant triangle in $\E$ and hence in $\B$ is precise.

Thus, proneness of $c'$ gives $d = b'$, and we have completed the proof that the square is a latent pullback.

\end{proof}

\begin{proposition}
Separated latent fibrations are closed under composition and pullback.
\end{proposition}
\begin{proof}
This follows from Lemma \ref{lemma:prone_with_comp_functors} and Lemma \ref{lemma:pullback_prone}.
\end{proof}


\subsection{Hyperconnected latent fibrations}\label{sec:connected}

\begin{definition}
A semifunctor $\p\colon \E \to \B$ is said to be \textbf{hyperconnected} (and $\p$ is said to be a \textbf{hyperconnection}) if for any $X \in \E$, the map
	\[ \p_{|\O(X)}\colon \O(X) \to \{d \in \O(\p(X)): d\p(1_X) = d\} \]
sending $e \in \O(X)$ to $\p(e)$ is an isomorphism.  We say that a latent fibration $\p$ is a \textbf{hyperfibration} if $\p$ is a hyperconnection.  
\end{definition}
This is a generalization of the definition of hyperconnection (see \cite[pg. 39]{journal:rcats-enriched}) to restriction \emph{semi}functors.  

Before looking at examples, it will be useful to know the following:
\begin{proposition}\label{prop:hyperconnected_equivalence}
A semifunctor $\p\colon \E \to \B$ is hyperconnected if and only if it is separated and admissible.
\end{proposition}
\begin{proof}
If $\p$ is hyperconnected, then it is monic on restriction idempotents, and so is separated.  Moreover, if $e\colon \p(X) \to \p(X)$ is a restriction idempotent in $\B$, then there is a restriction idempotent $e'\colon X \to X$ over it (since $\p$ is hyperconnected), and this is prone by Proposition \ref{prop:separated_equivalences}.

Conversely, if $\p$ is separated and admissible, then for any $X$, the map
	\[ \p_{|\O(X)}\colon \O(X) \to \{d \in \O(\p(X)): d\p(1_X) = e\} \]
is monic and has a section (by Lemma \ref{reflection-fibration}) thus is an isomorphism.
\end{proof}

Thus, by the results of the previous two sections, we have the following examples (and non-examples) of latent hyperfibrations:

\begin{example}\label{connected_examples}
\begin{enumerate}[(i)]
\item  The identity fibration $1_\X\colon \X\to\X$ is a hyperfibration.
	\item The strict simple fibration $\X[\X]$ is a hyperfibration, while in general the lax simple fibration $\X(\X)$ is not.  
	\item For any restriction category $\X$ with latent pullbacks, $\X^{\rightarrow} \to \X$ is a hyperfibration, while in general $\X^{\leadsto}$ is not.  
	\item The assemblies fibration is a hyperfibration. 
	\item The splitting fibration $\split(\X) \to \X$ is a hyperfibration.
	\item The propositions fibration and discrete fibrations are not a hyperfibration (in general) as they are not admissible.
\end{enumerate}
\end{example}

Table \ref{examples-properties} presents an overview of the examples of latent fibrations we have considered with their properties.
\begin{table}
\begin{center}
\begin{tabular}{|r|c|c|c||c|} \hline
Restriction semifunctor & Restriction Functor & Admissible  & Separated & Hyperfibration \\ \hline
$\split(\X) \to \X$ & $\times$ & \checkmark & \checkmark & \checkmark \\ \hline
$\X \x \Y \to \X$ & \checkmark & \checkmark & $\times$ & $\times$ \\ \hline
(lax) $\X(\X) \to \X$ & \checkmark & \checkmark & $\times$ & $\times$ \\ \hline
(strict) $\X[\X] \to \X$ & \checkmark & \checkmark & \checkmark & \checkmark \\ \hline
(lax) $\X^\leadsto \to \X$ & \checkmark & \checkmark & $\times$ & $\times$ \\ \hline
(strict) $\X^\rightarrow \to \X$ & \checkmark & \checkmark & \checkmark & \checkmark \\ \hline
$\O(\X) \to \X$ & \checkmark & $\times$ & \checkmark & $\times$ \\ \hline
${\sf Elt}(F) \to \X$ & \checkmark & $\times$ & \checkmark & $\times$ \\ \hline
${\sf Asm}(F) \to \X$ & \checkmark & \checkmark &  \checkmark & \checkmark \\ \hline
\end{tabular}
\end{center}
\caption{Properties of latent fibrations\label{examples-properties}}
\end{table}

Hyperconnections are particularly well-behaved when looking at prone arrows.   For a general restriction semifunctor, there can be prone maps over total maps or restriction monics which are themselves neither total nor restriction monic.  However, for a hyperconnection this is no longer the case, as {\em all} prone maps above a total map are total, and, similarly, {\em all\/} maps prone above a restriction monic are restriction monics.  In addition, any prone maps above a restriction retraction is a restriction retraction:

\begin{proposition}\label{prop:hyper_consequences}
If $\p\colon \E \to \B$ is a hyperconnection, then:
\begin{enumerate}[(i)]
	\item $\p$ reflects total maps;
	\item if $\alpha$ in $\E$ is prone over a partial isomorphism, then $\alpha$ is itself a partial isomorphism;
	\item if $s$ in $\E$ is prone over a restriction monic then $s$ is itself a restriction monic;
	\item if $r$ in $\E$ is prone over a restriction retraction, then $r$ is itself a restriction retraction;
	\item if $\p$ is a latent hyperfibration, then the total maps and restriction monics of $\p$ each determine a subfibration of $\p$.
\end{enumerate}
\end{proposition}
\begin{proof}
\begin{enumerate}[(i)]
	\item We saw in Lemma \ref{lemma:separated_reflects_total} that separated semifunctors reflect total maps.
	\item By Proposition \ref{prop:hyperconnected_equivalence}, $\p$ is admissible, so there is a prone restriction idempotent $e$ over $\widehat{\p(\alpha)} = \rs{\p(\alpha)^{(-1)}}$.  Let $\alpha'\colon X \to Y$ be the lift of $\p(\alpha)^{(-1)}$:
\[ \xymatrix{X \ar[dr]^e \ar@{..>}[d]_{\alpha'} & ~\ar@{}[drr]|{\textstyle\mapsto}&&&\p(X) \ar[dr]^{\widehat{\p(\alpha)}} \ar[d]_{\p(\alpha)^{(-1)}} 
\\ 
Y \ar[r]_{\alpha} & X &&& \p(Y) \ar[r]_{\p(\alpha)} & \p(X)} \]
Thus $\alpha' \alpha = e = \rs{e} = \rs{\alpha'}$ since the triangle is precise.  We want to show that $\alpha'$ is the partial inverse of $\alpha$, and so we also need $\alpha \alpha' = \rs{\alpha}$.

By definition, $e$ and $\alpha$ are prone, so by Lemma \ref{lemma:left_factor_separated}, $\alpha'$ is also prone, so by Lemma \ref{composites-prones}, $\alpha\alpha'$ is prone.  Thus $\alpha \alpha'$ and $\rs{\alpha}$ are prone over the same map ($\p(\rs{\alpha})$), and so by Lemma \ref{mediating_maps}.i, there is a mediating partial isomorphism $\beta\colon Y \to Y$ 
	\[ \xymatrix{ Y \ar[r]^{\beta} \ar[dr]_{\rs{\alpha}} & Y \ar[d]^{\alpha \alpha'} \\ & Y} \]
with $\rs{\beta} = \rs{\alpha}$ and $\rs{\beta^{(-1)}} = \rs{\alpha \alpha'}$.  But then we have
	\[ \beta = \beta \beta^{(-1)}\beta = \beta \rs{\beta^{(-1)}} = \beta \alpha \alpha' = \rs{\alpha}. \]
So $\beta$ is a restriction idempotent, and so its partial inverse $\beta^{(-1)}$ is itself, namely $\rs{\alpha}$.  Thus $\alpha \alpha^{(-1)} = \beta^{(-1)} \rs{\alpha} = \rs{\alpha}$, as required.

	\item Suppose $s\colon Y \to X$ is prone over a restriction monic $\p(s)$.   Then in particular $s$ is prone over a partial isomorphism, so by (ii), it has a partial inverse $r$, so that $sr = \rs{s}$ and $rs = \rs{r}$.  But since $\p(s)$ is a restriction monic, it is total, so by (i), $s$ is total, and thus $sr = \rs{s} = 1$.  Thus $s$ is a restriction monic.
	\item Similar proof to (iii).
	\item Restricting the fibration to total maps gives a trivial latent fibration which is a fibration.  Prone maps become Cartesian: so, in particular, restriction 
monics, being partial isomorphisms and therefore prone, are Cartesian.
\end{enumerate}
\end{proof}

Given any admissible latent fibration ${\sf p}\colon \E \to \B$ we may always extract a latent hyperfibration $\widehat{\p}\colon \widehat{\E} \to \B$ by restricting the restriction idempotents to the prone restriction idempotents and the maps to those which pullback prone restrictions to prone restrictions: that is, $f \in \widehat{\E}$ in case for every prone restriction idempotent, $e$ on the codomain of $f$,  $\rst{fe}$ is a prone restriction idempotent (recall from Lemma \ref{reflection-fibration} that the prone restrictions for an admissible semifunctor form a (reflective) subsemilattice).  We then have:

\begin{proposition}
If ${\sf p}\colon \E \to \B$ is an admissible latent fibration then 
\[ \xymatrix{~\widehat{\E}~ \ar[dr]_{\widehat{\p}} \ar@{>->}[rr] && \E \ar[dl]^{{\sf p}} \\ & \B} \]
commutes, and $\widehat{\p}$ is a latent hyperfibration.
\end{proposition}

\begin{proof}
$\widehat{\p}$ is clearly still admissible, and by definition, all restriction idempotents in it are prone.  Thus, by Proposition \ref{prop:separated_equivalences} and Proposition \ref{prop:hyperconnected_equivalence}, $\widehat{\p}$ is a latent hyperfibration.
\end{proof}

\begin{proposition}
Hyperconnected latent fibrations are closed under composition and pullback.
\end{proposition}
\begin{proof}
Follows from Lemma \ref{lemma:prone_with_comp_functors} and Lemma \ref{lemma:pullback_prone}.
\end{proof}

Another key property of a latent hyperfibration (see Section \ref{sec:dual}) is that one can construct its fibrational dual.


\section{\texorpdfstring{${\sf M}$}{M}-category Fibrations} \label{sec:fibrations-of-partial}

The purpose of this section is to provide a series of results describing the relation between fibrations of ${\sf M}$-categories and latent fibrations.  In particular, we completely characterize, as an ${\sf M}$-category, a certain type of latent fibration which we call $r$-split (see Definition \ref{defn:split}).  This allows us to similarly characterize $r$-split admissible latent fibrations and $r$-split latent hyperfibrations. 

Another key result in this story is that \emph{admissible} latent fibrations are stable under the process of splitting restriction idempotents: this means that any admissible latent fibration can be embedded into an $r$-split latent fibration and hence into the partial map category of an ${\sf M}$-category.  It is worth emphasizing that for this we need the additional property of being admissible.  Thus,  in general, one cannot split the restriction idempotents of a latent fibration to get a latent fibration.

In this section we shall use restriction functors rather than semi-functors.  In particular, we  shall need functors which preserve the total maps in order to access the equivalence between $r$-split restriction categories and ${\sf M}$-categories.


\subsection{\texorpdfstring{$r$}{r}-split latent fibrations}

In this section we will focus on restriction functors between $r$-split restriction categories, and demand that the fibres of these functors behave well with respect to splitting.

\begin{definition}\label{defn:split} Let $\p\colon \E \to \B$ be a restriction functor.  
\begin{itemize}
	\item $\p$ is said to be \textbf{well-fibred} if for all $B$ in $\B$, whenever a restriction idempotent in the fibre\footnote{Note that we reserve the notation $\p^{-1}(B)$ for the classical fibre, containing only arrows that are mapped to $1_B$ by $\p$, whereas $\p^{(-1)}(B)$ is used for the strand containing all arrows that are mapped to a restriction idempotent on $B$.} $\p^{-1}(B)$ splits in $\E$, it also splits in $\p^{-1}(B)$.
	\item $\p$ is said to be {\bf $r$-split} in case the restriction categories $\E$ and $\B$ are $r$-split, and $\p$ is well-fibred.  
\end{itemize}
\end{definition}

One reason to ask that the projection functor be well-fibred is the following result:

\begin{lemma} \label{totals-in-flush-fibration}
If $\p\colon  \E \to \B$ is a restriction functor in which $\E$ is $r$-split, then every total map, $f\colon B \to \p(E)$, having a prone map $f'\colon E' \to E$  above it whose idempotent $\rst{f'}$ splits in $\p^{-1}(B)$, has a prone map, which is {\em total}, above $f$ at $E$.
\end{lemma}

\begin{proof}
Suppose $f\colon B \to \p(E)$ is total in $\B$ with an $f'\colon E' \to E$ which is prone above it.  Note that $\rst{f'} \in \p^{-1}(B)$ as $\p(\rst{f'}) =\rst{\p(f')} = \rst{f} = 1_B$.  Thus, since $\p$ is well-fibred, there is a splitting $m\colon E'' \to E$, $r\colon E' \to E''$ of $\rs{f'}$ with both $m$ and $r$ in $\p^{-1}(B)$ also.  Let $f'' = fm'$; we claim that $f''$ is the required total prone map over $f$.

First, since $m \in \p^{-1}(B)$, 
	\[ \p(f'') = \p(m)\p(f') = 1_B f = f, \]
so $f''$ is over $f$.  Next,
	\[ \rs{f''} = \rs{mf'} = \rs{m\rs{f'}} = \rs{mrm} = \rs{m} = 1_{E''}, \]
so $f''$ is total.  Finally, we have $f' = \rs{f'}f' = rmf = rf''$, so since $f'$ is prone and $r$ is a restriction retraction, $f''$ is itself prone by Lemma \ref{restriction-retraction}.iii.  
\end{proof}

When $\p\colon \E \to \B$ is an $r$-split latent fibration the total maps form a fibration because, as observed previously, prone maps in a category in which all maps are total, are just Cartesian maps.  Thus, we have the important observation:

\begin{corollary} \label{totals-in-latent-fibration}
The total maps of an $r$-split latent fibration form a fibration.
\end{corollary}


\subsection{Latent fibrations and \texorpdfstring{${\sf M}$}{M}-categories}

Next, it will be helpful to understand precise triangles and prone maps in a partial map category.  

\begin{lemma}\label{lemma:precise_in_partialMapCat}
Suppose that 
	\[ X \from^{m} X' \to^{f} Y, \ X \from^{m'} X'' \to^{f'} Z, \ Y \from^{n} Y' \to^g Z \]
are representatives of maps in a partial map category ${\sf Par}(\X,{\sf M})$.  Then the triangle 
	\[ \xymatrix{X \ar[dr]^{(m',f')} \ar[d]_{(m,f)} & \\ Y \ar[r]_{(n,g)} & Z } \]
precisely commutes in ${\sf Par}(\X,{\sf M})$ if and only if there is a map $k\colon X'' \to Y'$ such that $kg = f'$ and $(m,f) \cong (m',kn)$.  
\end{lemma}
\begin{proof}
First, suppose that the triangle precisely commutes.  Then $\rs{(m,f)} \cong \rs{(m',f')}$, so there is an isomorphism $\alpha\colon X'' \to X'$ such that $\alpha m = m'$.  Let $(am,bg)$ be a representative of $(m,f) \circ (n,g)$:
\[ \xymatrix{ & & P \ar[dl]_{a} \ar[dr]^b & & \\ & X' \ar[dl]_m \ar[dr]^f & & Y' \ar[dl]_n \ar[dr]^g & \\ X & & Y & & Z} \]
then since $(m,f) \circ (n,g) \cong (m',f')$, there is an isomorphism $\beta\colon X'' \to P$ such that $\beta am = m'$ and $\beta bg = f'$.  Let $k = \beta b$, so we have $kg = f'$.  Moreover, we claim that $\alpha$ witnesses $(m,f) \cong (m',kn)$.  Indeed, we have $\alpha m = m'$.  We also need $\alpha f = kn = \beta bn = \beta a f$, so it suffices to show $\alpha = \beta a$. But $\alpha m = m' = \beta a m$, so since $m$ is monic, $\alpha = \beta a$, as required.

Conversly, suppose that we have such a $k$, and an $\alpha\colon X'' \to X$ that witnesses $(m,f) \cong (m',kn)$.  Then by definition $\alpha$ witnesses $\rs{(m,f)} \cong \rs{(m',f')}$.  Moreover, 
	\[ \xymatrix{ X'' \ar[r]^{k}  \ar[d]_{\alpha} & Y' \ar[d]^{n} \\ X' \ar[r]_{f} & Y } \]
commutes by assumption, and it is straightforward to check that it is a pullback since $\alpha$ is an isomorphism and $n$ is monic.  Thus 
\[ \xymatrix{ & & X'' \ar[dl]_{\alpha} \ar[dr]^k & & \\ & X' \ar[dl]_m \ar[dr]^f & & Y' \ar[dl]_n \ar[dr]^g & \\ X & & Y & & Z} \]
is a representative of $(m,f) \circ (n,g)$, and equals $(m',f')$ by assumption on $\alpha$ and $k$.

\end{proof}

The following result allows us to build prone maps in a partial map category from Cartesian maps in the original category:

\begin{lemma}\label{lemma:prone_in_partialMapCat}
Suppose $\p\colon (\E,\M_{\E}) \to (\B,\M_{\B})$ is an $\M$-category functor.  If $f\colon X' \to Y$ is $\p$-Cartesian in $\E$, then for any $m\colon X \to X'$ in $\M_{\E}$, (the equivalence class of) $(m,f)$ is ${\sf Par}(\p)$-prone in the restriction category ${\sf Par}(\E,\M_{\E})$.  
\end{lemma}
\begin{proof}
Suppose $Z \from^{n} Z' \to^{g} Y$ is a map from $Z$ to $Y$:
\[ \xymatrix{Z \ar[dr]^{(n,g)} & \\ X \ar[r]_{(m,f)} & Y } \]

Then by the previous lemma, to have a precise triangle with $\p(m,f)\colon X \to Y$ in the base means that there is an $h\colon Z' \to \p(X')$ with $h\p(g) = \p(f)$, with the triangle of the form
\[ \xymatrix@C=4em{\p(Z) \ar[dr]^{(\p(n),\p(g))} \ar[d]_{(\p(n),h\p(m)} & \\ \p(X) \ar[r]_{(\p(m),\p(f))} & \p(Y) } \]
Then since $f$ is $\p$-Cartesian, there is a unique $\tilde{h}\colon Z' \to X'$ in $\E$ such that $\tilde{h}f = g$ and $\p(\tilde{h}) = h$.  Then by the previous lemma, $(n,\tilde{h}m)$ is a precise fill-in for the first triangle:
\[ \xymatrix{Z \ar[dr]^{(n,g)} \ar[d]_{(n,\tilde{h}m)} & \\ X \ar[r]_{(m,f)} & Y } \]
and is over $(\p(n),\tilde{h}\p(m))$ by definition.  

It remains to show it is unique.  If we have another precise fill-in, by the previous lemma we can take it to be of the form $(n,km)$ for some $k\colon Z' \to X'$ such that $kf = g$.  Since it must be over $(\p(n),h\p(m))$, we must have some isomorphism $\alpha\colon \p(Z') \to \p(Z')$ such that $\p(n) = \alpha \p(n)$ and $\p(k)\p(m) = \alpha h \p(m)$.  But since $\p(n)$ is monic, the first equality means $\alpha = 1$; then since $\p(m)$ is monic the second equality gives $\p(k) = h$.  Thus we have $kf = g$ and $\p(k) = h$, so since $\tidle{h}$ was unique with these properties, $k = \tidle{h}$.  So $(n,\tilde{h}m)$ is unique.  
\end{proof}

Thus, having an $\M$-category functor $\p\colon (\E,\M_{\E}) \to (\B,\M_{\B})$ being an ordinary fibration is close to ${\sf Par}({\sf p})\colon {\sf Par}(\E,{\sf M}_{\E}) \to {\sf Par}(\B,{\sf M}_{\B})$ being a latent fibration.  However, something is missing: if we are given a span
\[ \xymatrix{ & Y' \ar[dl]_m \ar[dr]^{f} & \\ Y & & \p(X)} \]
in ${\sf Par}(\B,{\sf M}_{\B})$, then by the above result, we can lift $f$ to a Cartesian arrow; however, there is no reason why there should be an $\M_{\E}$ monic over $m$ in $\E$.  Thus, we make the following definitions.  

\begin{definition} Let $\p\colon (\E,{\sf M}_{\E}) \to (\B,{\sf M}_{\B})$ be an  ${\sf M}$-category functor.  
\begin{itemize}
	\item $\p$ is said to be {\bf ${\sf M}$-plentiful} in case for every object $X \in \E$ and ${\sf M}_{\B}$-map $m\colon  \p(X) \to Y$ there is an ${\sf M}_\E$-map $n\colon X \to X'$ in $\E$ with ${\sf p}(n) = m$.
	\item $\p$ is said to be an \textbf{${\sf M}$-fibration} if it is a fibration (in the ordinary sense) and it is ${\sf M}$-plentiful.  
\end{itemize}
\end{definition}

This is enough to get an ($r$-split) latent fibration:  

\begin{lemma}
If $\p\colon (\E,{\sf M}_{\E}) \to (\B,{\sf M}_{\B})$ is an ${\sf M}$-fibration then ${\sf Par}({\sf p})\colon {\sf Par}(\E,{\sf M}_{\E}) \to {\sf Par}(\B,{\sf M}_{\B})$ is an $r$-split latent fibration.
\end{lemma}

\begin{proof}
Let $(m,f)\colon Y \to \p(X)$ be a map in ${\sf Par}(\B,{\sf M}_{\B})$, i.e., a span:
\[ Y \from^{m} Y' \to^f \p(X) \]
We lift this to 
\[ Y_0 \from^n Y_1' \to^{f^*} X \] 
where $f^*$  is a Cartesian arrow above $f$ and $n$ is any ${\sf M}_{\E}$-map above $m$, the existence of which is guaranteed by the plentiful requirement.  Then by Lemma \ref{lemma:prone_in_partialMapCat}, $(n,f^*)$ is prone in ${\sf Par}(\T,{\sf M}_{\T})$.

We also need to show that ${\sf Par}({\sf p})$ is well-fibred.  An idempotent  is a span $(m,m)$ where $m\colon X' \to X$ is in ${\sf M}_\E$: its splitting (as an equalizer with the identity) is 
$(m,1_{X'})$.  Now $(m,m)$ could have $(\p(m),\p(m)) = (1_X,1_X)$ and yet $\p(X') \neq \p(X)$ although, as both are limits of the trivial equalizer, they must be isomorphic, say by $\gamma\colon \p(X') \to \p(X)$. But that means there is a Cartesian arrow above $\gamma^{-1}$,  $(\gamma^{-1})^{*}\colon X_0 \to X'$ which is an isomorphism and has $\p(T_0) =\p(X)$.  This means that $(\gamma^{-1})^{*}m\colon X_0 \to X$ is also a splitting and it is now in the fiber over $\p(X)$. 
\end{proof}

Conversely we have:

\begin{lemma}\label{lemma:split_to_MFibration}
If $\p\colon \E \to \B$ is an $r$-split latent fibration, then the induced functor between the total map categories, with their restriction monics, ${\sf Total}(\p)\colon ({\sf Total}(\E),{\sf Monic}(\E)) \to ({\sf Total}(\B),{\sf Monic}(\B))$, is an ${\sf M}$-fibration.  
\end{lemma}

\begin{proof} 
The fact that $\p$ is a latent fibration ensures that ${\sf Total}(\p)$ is a fibration as discussed in Corollary \ref{totals-in-latent-fibration}.  It remains, therefore, to show that ${\sf Total}(\p)$ is ${\sf M}$-plentiful.  
Toward this end let $m\colon  \p(E) \to X$ be a  restriction monic in $\B$, so that by definition there is a restriction retraction $r\colon X \to \p(E)$ such that $mr = 1_{\p(E)}$ and $rm = \rst{r}$.  Let $r'\colon X' \to E$ be a prone arrow over it.  Then since $\E$ is $r$-split, we can use Lemma \ref{restriction-retraction}.iv to get that $r'$ is itself a restriction retraction.  Thus, there is a restriction monic $m'\colon E \to X'$ which is its partial inverse.  Since restriction functors preserve partial inverses, we have
	\[ \p(m') = \p(r'^{(-1)}) =\p(r')^{(-1)} = r^{(-1)} = m, \]
showing that $\p$ is ${\sf M}$-plentiful. 
\end{proof}

As a consequence we now have:

\begin{theorem}
$r$-split latent fibrations, in the equivalence between $r$-split restriction categories and ${\sf M}$-categories, correspond precisely to ${\sf M}$-fibrations.
\end{theorem}

This gives us a further source of examples of latent fibrations.  Any category with finite limits $\X$ is an ${\sf M}$-category with respect to all the monics.  The arrow 
category gives the standard fibration over $\X$, $\partial\colon \X^{\rightarrow} \to \X$.  Then moving to ${\sf Par}(\partial)$ always gives a latent fibration as above any monic $m\colon X \to Y$ in 
$\X$ and any object $x\colon X' \to X$ in $\X^\rightarrow$ lies the map (which is a square):
\[ \xymatrix{X' \ar[d]_x \ar@{=}[r] & X' \ar[d]^{xm} \\ X \ar[r]_m & Y} \]
which is clearly monic.  Note  when $m$ is an isomorphism this is an isomorphism in $\X^\rightarrow$ and thus this $\partial$ is ${\sf M}$-conservative.  This means the functor ${\sf Par}(\partial)$ is ${\sf M}$-plentiful and the ${\sf Par}(\partial)$ is a latent fibration  -- and in a few moments we will see, because pullbacks of monics along any map are monic, that this is an \emph{admissible} latent fibration.  Furthermore, it is not hard to see that this is actually the full subfibration of ${\sf Par}(\X,{\sf Monic})^{\leadsto} \to {\sf Par}(\X,{\sf Monic})$ determined by the objects $a\colon A' \to A \in \E$ such that $a$ is total.  This means, in particular, it is not hyperconnected.

It is worth noting that this is not the only choice of ${\sf M}$-maps possible for the total category $\X^\rightarrow$, as we could have chosen squares which are pullbacks.  These are clearly monics in $\X^\rightarrow$ and still give an ${\sf M}$-plentiful fibration using the codomain map as the above square for demonstrating ${\sf M}$-plenitude is a pullback.  It is not hard then to see that passing to the partial map categories for this  fibration gives a subfibration of the strict codomain fibration, 
${\sf Par}(\X,{\sf Monic})^{\rightarrow} \to {\sf Par}(\X,{\sf Monic})$ determined by objects $a$ which are total maps.


\subsection{Admissible, separated, and hyperconnected \texorpdfstring{${\sf M}$}{M}-fibrations}

In this section we characterize the admissible, separated, and hyperconnected latent fibrations between partial map categories.  

\begin{definition}
An ${\sf M}$-fibration $\q\colon (\T,{\sf M}_{\T}) \to (\S,{\sf M}_{\S})$ is an \textbf{admissible ${\sf M}$-fibration} if Cartesian maps over ${\sf M}$-maps are themselves ${\sf M}$-maps.
\end{definition}

We can immediately observe:

\begin{lemma}
If ${\sf p}\colon (\E,{\sf M}_{\E}) \to (\B,{\sf M}_{\B})$ is an admissible ${\sf M}$-fibration then ${\sf Par}({\sf p})\colon {\sf Par}(\E,{\sf M}_{\E}) \to {\sf Par}(\B,{\sf M}_{\B})$ is an admissible $r$-split latent fibration.
\end{lemma} 

\begin{proof} 
We must show that there is a prone restriction idempotent above every restriction idempotent.  A restriction idempotent as a partial map is a span $(m,m)$.  Thus, if we let $m^*$ be a Cartesian map above $m$, by assumption $m^*$ is in ${\sf M}_{\E}$, and by Lemma \ref{lemma:prone_in_partialMapCat},  $(m^*,m^*)$ is prone over $(m,m)$.  
\end{proof}

Conversely, we also have:

\begin{lemma}
If $\p\colon \E \to \B$ is an admissible $r$-split latent fibration, then ${\sf Total}(\p)\colon ({\sf Total}(\E),{\sf Monic}(\E)) \to ({\sf Total}(\B),{\sf Monic}(\B))$ is an admissible ${\sf M}$-fibration.  
\end{lemma}
\begin{proof}
We already know by Lemma \ref{lemma:split_to_MFibration} that ${\sf Total}(\p)$ is an ${\sf M}$-fibration; thus, it remains to show the ${\sf M}$-admissible property.  Suppose that $m\colon X' \to X$ is a Cartesian map in ${\sf Total}(\E)$ such that $\p(f)$ is a restriction monic in $\B$.  By Lemma \ref{lemma:cart_are_prone}, $m$ is prone in $\E$.  Then by Lemma \ref{restriction-monics-in-fibrations}, $m$ is a partial isomorphism, so there is an $r\colon X \to X'$ so that $rm = \rs{r}$ and $mr = \rs{m} = 1$.  That is, $m$ is a restriction monic.
\end{proof}

\begin{theorem}
Admissible $r$-split latent fibrations, in the equivalence between $r$-split restriction categories and ${\sf M}$-categories, correspond precisely to admissible ${\sf M}$-fibrations.
\end{theorem}

We next turn to characterizing separated latent fibrations between partial map categories.  

\begin{definition}
An ${\sf M}$-fibration $\q\colon (\T,{\sf M}_{\T}) \to (\S,{\sf M}_{\S})$ is a \textbf{separated ${\sf M}$-fibration} if all ${\sf M}_{\T}$-maps are Cartesian.
\end{definition}

\begin{lemma}
If ${\sf q}\colon (\T,{\sf M}_{\T}) \to (\S,{\sf M}_{\S})$ is a separated ${\sf M}$-fibration then ${\sf Par}({\sf q})\colon {\sf Par}(\T,{\sf M}_{\T}) \to {\sf Par}(\S,{\sf M}_{\S})$ is a separated $r$-split latent fibration.
\end{lemma} 
\begin{proof}
By Proposition \ref{prop:separated_equivalences}, it suffices to show that restriction idempotents are prone, so suppose that $(m,m)$ is a restriction idempotent in ${\sf Par}(\T,{\sf M}_{\T})$.  Then by assumption $m$ is Cartesian, so by Lemma \ref{lemma:prone_in_partialMapCat}, $(m,m)$ is prone.
\end{proof}

\begin{lemma}
If $\p\colon \E \to \B$ is a separated $r$-split latent fibration, then ${\sf Total}(\p)\colon ({\sf Total}(\E),{\sf Monic}(\E)) \to ({\sf Total}(\B),{\sf Monic}(\B))$ is a separated ${\sf M}$-fibration.  
\end{lemma}

\begin{proof}
Suppose $m\colon X \to Y$ is a restriction monic in $\E$.  Then by Corollary \ref{cor:res_monic_prone}, $m$ is prone, and by definition it is total, hence it is Cartesian in ${\sf Total}(\E)$.  
\end{proof}

\begin{theorem}
Separated $r$-split latent fibrations, in the equivalence between $r$-split restriction categories and ${\sf M}$-categories, correspond precisely to separated ${\sf M}$-fibrations.
\end{theorem}

By proposition \ref{prop:hyperconnected_equivalence}, we now also have a characterization of $r$-split latent hyperfibrations:

\begin{corollary}
$r$-split latent hyperfibrations, in the equivalence between $r$-split restriction categories and ${\sf M}$-categories, correspond precisely to ${\sf M}$-fibrations which are admissible and separated.  
\end{corollary}


\subsection{\texorpdfstring{$r$}{r}-splitting a latent fibration}

Next, we turn to the question of what happens when we split the restriction idempotents of a latent fibration.  Unfortunately, in general this will not again be a latent fibration.  To see the issue, suppose we have a latent fibration $\p\colon \E \to \B$ and a map $f\colon (Y,e_0) \to (\p(X),\p(e))$ in $\split(\B)$.  As a mere map in $\B$, we can lift it to a $\p$-prone map $f^*\colon X' \to X$ in $\E$.  However, we need it to be a map in $\split(\E)$, and in particular we need a restriction idempotent $e'$ over $e$ to serve as the domain for $f^*$ as a map in $\split(\E)$.  As we have seen earlier (specifically, in the introduction to section \ref{sec:types}), this need not always exist.  Moreover, we will need this idempotent $e'$ to itself be prone, so that when we compose it with $f^*$ the result will again be prone.

Of course, asking for a prone restriction idempotent in $\E$ over a restriction idempotent in $\B$ is precisely what we demand of an admissible latent fibration, and this property turns out to be sufficient.  

\begin{proposition}  \label{splitting-latent-fibration}
If ${\sf p}\colon \E \to \B$ is an admissible latent fibration then $\split(\p)\colon \split(\E) \to \split(\B)$ is an admissible latent fibration.
\end{proposition}

\begin{proof}
It is easy to see that splitting idempotents always results in a well-fibred restriction functor: suppose $e = \rst{e}$ is an object of $\split(\B)$ and $e_0\colon e_1 \to e_1$ is in $\p^{-1}(e)$, so $\p(e_1) = e$ and $\p(e_0) = e$, then a splitting of the map $e_0$ is given by $e_1 \to^{e_0} e_0 \to^{e_0} e_1$ and this is all in $\p^{-1}(e)$.

Now suppose $f\colon e_0 \to {\sf p}(e)$ is a map in $\split(\B)$, with $e$ a restriction idempotent in $\E$; we must show that there is a prone 
lifting to some map in $\split(\E)$.   As a mere map, $f\colon B \to {\sf p}(E)$ in $\B$ has a prone arrow over it 
$f'\colon E' \to E$ but also $e_0$ has a, prone, restriction idempotent $e_0^{*}$ on $E'$ over it (as ${\sf p}$ is an admissible latent fibration): we claim that 
$e_0^{*} f' e\colon e'_0 \to e$ is prone over $f\colon e_0 \to {\sf p}(e)$: it is certainly over $f$.   Suppose now that $g\colon e_1 \to e$ in $\E$ and 
${\sf p}(g) = h f$ is precise with $h\colon {\sf p}(e_1) \to e_0$  then there is a lifting of $h$, $\widetilde{h}$, in $\E$ to give a precise triangle 
$\widetilde{h} (e_0^{*} f') = g$.  The preciseness means $\rst{\widetilde{h}} = \rst{g} \leq e_1$ so it starts at the idempotent $e_1$ as desired. 
It also ends at $e_0^{*}$ as $h e_0= h$ so $\widetilde{h} e_0^{*} = \widetilde{h}$.  Finally we have $\widetilde{h}e_0^{*} f' e = \widetilde{h} f' e = g e = g$.
Thus $e_0^{*} f' e\colon e_0^{*} \to e$ is prone over $f\colon e_0 \to {\sf p}(e)$ and $\split({\sf p})\colon \split(\E) \to \split(\B)$ is a latent fibration.

Finally we must show that above each restriction idempotent  $e$ on ${\sf p}(e')$, in $\split(\B)$, which means $e \leq {\sf p}(e')$ in $\B$, there is a prone 
restriction idempotent on $e'$: it is of course $e^{*}e'$, which is prone by the argument above, and less than or equal to $e'$. 
\end{proof}

The property of being separated (thus, hyperconnected) is also preserved under this process:

\begin{corollary}
If ${\sf p}\colon \E \to \B$ is a latent hyperfibration then $\split({\sf p})\colon \split(\E) \to \split(\B)$ is a latent hyperfibration.
\end{corollary}

\begin{proof}
Partial isomorphisms in $\split(\E)$ are partial isomorphisms in $\E$ so are prone there and this easily means they are prone in $\split(\E)$.  The result thus follows from Proposition \ref{prop:separated_equivalences} and Proposition \ref{prop:hyperconnected_equivalence}.  
\end{proof}


\section{The Dual of a Latent Hyperfibration}\label{sec:dual}

Latent hyperfibrations have the  interesting property that one can form their fibred dual.  This process is, for example, important in  framing the semantics of reversible differential programming in which a fibrewise ability to take the dual is required.  For an ordinary fibration, one can construct its fibred dual by using a span construction on the total category of the fibration (for example, see \cite[pg. 898]{journal:benabou_dual}, \cite[Section 1.10.11]{book:Jacobs-Cat-Log}, and \cite{arxiv:simple-dual}).  In this section, we see how to generalize this construction to any latent hyperfibration.  We begin by investigating certain conditions under which we can form a restriction category of spans.  


\subsection{Hyper-opens and spans}

Let $\X$ be a restriction category and ${\cal H}$ and ${\cal Q}$ be two {\bf downclosed systems} of maps, that is classes of maps which are downclosed, closed to composition, and contain the partial isomorphisms, which  commute in the sense that for each $h\colon A \to B \in {\cal H}$ and $q\colon C \to B \in {\cal Q}$ there is a latent pullback 
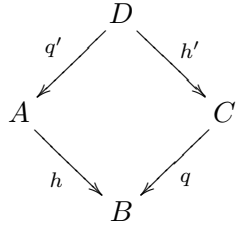
\[ \xymatrix{& D \ar[dr]^{h'} \ar[dl]_{q'} \\ A \ar[dr]_{h} & & C \ar[dl]^{q} \\ & B} \]
such that $h' \in {\cal H}$  and $q' \in {\cal Q}$.  Note that if there is one latent pullback with this property then {\em all\/} latent pullbacks for the cospan $(h,q)$  will have this 
property as the classes are certainly closed to mediation.

We shall be particularly interested in the situation in which the maps in ${\cal H}$ are {\bf hyper-open}
  in the sense that, for each $(h\colon X\to Y) \in {\cal H}$  there is an isomorphism 
$\exists_h$ which is inverse to the pullback $h^*$ in the sense that the following are inverse:
\[  \exists_h\colon \rst{h}/{\cal O}(X) \to \exists_h(1_X)/{\cal O}(Y); e \mapsto \exists_h(e) ~~~~~~h^{*}\colon \exists_h(1_X)/{\cal O}(Y) \to \rst{h}/{\cal O}(X)\colon e \mapsto \rst{he} \]

Note that any hyper-open map is automatically open (for a map to be open one only requires the existence of the adjoint, where hyper-open requires the adjoint to be an isomorphism; see \cite{journal:ranges1}) and so in particular any hyper-open map $h\colon X \to Y$ has an associated map $\widehat{h}\colon Y \to Y$ which satisfies the properties of a range combinator.  

Given such a commuting pair of classes of maps we can form the span category  ${\sf Span}_\X({\cal H},{\cal Q})$ with the following data:
  \begin{description}
    \item[Objects: ] Those of $\X$;
    \item[Maps: ] An arrow $A \to B$ consists of an equivalence class of spans 
      
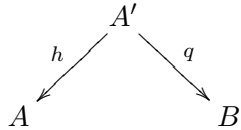
\[ \xymatrix{ & A' \ar[dl]_{h} \ar[dr]^{q} \\  A && B}   \]  
      such that $h \in {\cal H}$, $q \in {\cal Q}$ with $\rst{h} = \rst{q}$ under the equivalence relation that $(f,g) \sim (f',g')$ when there 
      is a mediating partial isomorphism $\alpha \in {\cal H}$ such that:
      
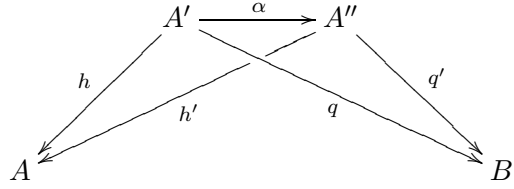
\[ \xymatrix@C=4em@R=4em{
          & A' \ar[dl]_{h} \ar[drr]_{q} \ar[r]^{\alpha} & A''  \ar[dll]^{h'}|<<<<<<<<<\hole  \ar[dr]^{q'} \\ A & & & B }   \]
    where  $\alpha h' = h$ and $\alpha q' = q$,  $\alpha^{-1} h = h'$ and $\alpha^{-1}q = q'$, and $\rst{\alpha^{-1}} = \rst{q'}$ and $\rst{\alpha} = \rst{q}$.
    \item[Identities: ] The identity span $A = A = A$;
    \item[Composition: ] By latent pullback.
  \end{description}
  
  ${\sf Span}_\X({\cal H},{\cal Q})$ is clearly a category but, in general, is not a restriction category.  However, if ${\cal H}$ is a hyper-open class of maps then 
  it is a restriction category with $\rst{(h,q)}= (\widehat{h},\widehat{h})$.

\begin{proposition}\label{prop:span_res_cat}
If ${\cal H}$ and ${\cal Q}$ are two commuting classes of maps with ${\cal H}$ a class of hyper-opens then 
${\sf Span}_\X({\cal H},{\cal Q})$ is a restriction category.
\end{proposition}
\begin{proof}
We will use $\circ$ to denote composition in ${\sf Span}_\X({\cal H},{\cal Q})$.  For [R.1], let 
	
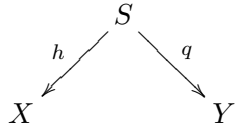
\[ \xymatrix{ & S \ar[dl]_{h} \ar[dr]^{q} \\  X && Y}   \] 
be a representative of an arrow in ${\sf Span}_\X({\cal H},{\cal Q})$.  Then since $\hat{h}$ is a restriction idempotent (and hence a partial isomorphism), by Lemma \ref{latent-pullbacks}.iii, the composite $\rs{(h,q)} \circ (h,q)$ is given by
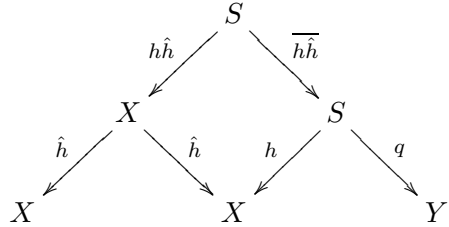
\[ \xymatrix{ & & S \ar[dl]_{h\hat{h}} \ar[dr]^{\rs{h\hat{h}}} & & \\ & X \ar[dl]_{\hat{h}} \ar[dr]^{\hat{h}} & & S \ar[dl]_{h} \ar[dr]^{q} & \\ X & & X & & Y } \]
But using the range properties of open maps, we have
	\[ (h\hat{h}\hat{h},\rs{h\hat{h}}q) = (h,\rs{h}q) = (h,\rs{q}q) = (h,q) \]
as required.  [R.2] and [R.3] follow similarly.

For [R.4], let $(h_2^*h_1,q_1^*q_2)$ denote the composite $(h_1,q_1) \circ (h_2,q_2)$:
\[ \xymatrix{ & & P \ar[dl]_{h_2^*} \ar[dr]^{q_1^*} & & \\ & S \ar[dl]_{h_1} \ar[dr]^{q_1} & & T \ar[dl]_{h_2} \ar[dr]^{q_2} & \\ X & & Y & & Z } \]
(where the diamond is a latent pullback). 

Then $\rs{(h_1,q_1) \circ (h_2,q_2)} = (\widehat{h_2^*h_1},\widehat{h_2^*h_1})$, so as for [R.1], 
	\[ \rs{(h_1,q_1) \circ (h_2,q_2)} \circ (h_1,q_1) = (h_1\widehat{h_2^*h_1}, \rs{h_1\widehat{h_2^*h_1}} q_1) = (\rs{h_1\widehat{h_2^*h_1}}h_1, \rs{h_1\widehat{h_2^*h_1}} q_1). \]
	
On the other hand, $(h_1,q_1) \circ \rs{(h_2,q_2)}$ is given by
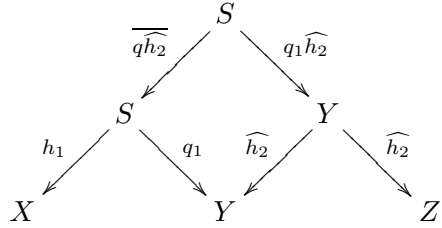
\[ \xymatrix{ & & S \ar[dl]_{\rs{q\widehat{h_2}}} \ar[dr]^{q_1 \widehat{h_2}} & & \\ & S \ar[dl]_{h_1} \ar[dr]^{q_1} & & Y \ar[dl]_{\widehat{h_2}} \ar[dr]^{\widehat{h_2}} & \\ X & & Y & & Z } \]
that is,
	\[ (h_1,q_1) \circ \rs{h_2,q_2} = (\rs{q_1\widehat{h_2}}h_1,q_1\widehat{h_2}) = (\rs{q_1\widehat{h_2}}h_1, \rs{q_1\widehat{h_2}}q_1) \]
Thus, comparing the two expressions, it suffices to prove that
	\[ \rs{h_1\widehat{h_2^*h_1}} = \rs{q_1\widehat{h_2}}. \]

Note that these are both restriction idempotents on $S$.  To prove they are equal, since $h_2^*$ is hyper-open, pulling back along $h_2^*$ is an isomorphism; thus, it suffices to instead prove that 
	\[ \rs{h_2^* \rs{h_1\widehat{h_2^*h_1}}} = \rs{h_2^*\rs{q_1\widehat{h_2}}}. \]
Indeed,
	\[ \rs{h_2^* \rs{h_1\widehat{h_2^*h_1}}} = \rs{h_2^*h_1\widehat{h_2^*h_1}} = \rs{h_2^*h_1} \]
while
   \[  \rs{h_2^*\rs{q_1\widehat{h_2}}} = \rs{h_2^* q_1 \widehat{h_2}} = \rs{q_1^*h_2 \widehat{h_2}} = \rs{q_1^*h_2} = \rs{h_2^*q_1} = \rs{h_2^*h_1} \]
since $\rs{q_1} = \rs{h_1}$.  

\end{proof}
  

\subsection{The fibrational dual}

Our main construction of this section is the following:

\begin{proposition}
Suppose $\p\colon \E \to \B$ is a latent hyperfibration.  Then there is a restriction category $\E^*$ whose objects are those of $\E$ and whose maps are equivalence classes of spans
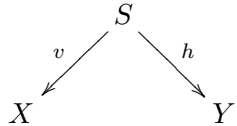
\[ \xymatrix{ & S \ar[dl]_{v} \ar[dr]^{h} & \\ X & & Y} \]
where $v$ is subvertical, $h$ is prone, and $\rs{v} = \rs{h}$; the equivalence is up to subvertical partial isomorphism and as in Proposition \ref{prop:span_res_cat}, the restriction of $[(v,c)]$ is $[(\hat{v},\hat{v})]$.  Moreover, there is a restriction functor $\p^*\colon \E^* \to \B$ which is defined on objects as $\p^*(A) = \p(A)$ and on arrows as $\p^*(v,h) = \p(h)$.
\end{proposition}
\begin{proof}
Using Lemma \ref{lemma:pullback_prone_subvertical}, the prones and subverticals in a latent hyperfibration form a commuting pair of classes of maps; moreover, the subverticals are hyper-open and the classes are downclosed systems of maps.   Thus Proposition \ref{prop:span_res_cat} tells us we can form a restriction category from the spans.  We will use $\circ$ to denote composition in $\E^*$.

We need to show that $\p^*$ is a well-defined restriction functor.  To start, first note that for a subvertical partial isomorphism $\alpha$, $\alpha \alpha^{(-1)} = \rs{\alpha}$ gives $\p(\alpha)\p(\alpha^{(-1)}) = \rs{\p(\alpha)} = \p(\alpha)$, so that $\p(\alpha) \leq \p(\alpha^{(-1)})$.   The reverse inequality follows similarly, and so $\p(\alpha) = \p(\alpha^{(-1)})$.  We will use this to check $\p^*$ is well-defined.  

Indeed, suppose that $(v_1,h_1)$ and $(v_2,h_2)$ represent the same map in $\E^*$, so that, in particular, there is a subvertical partial isomorphism $\alpha$ so that $h_1 = \alpha h_2$ and $\rs{\alpha^{(-1)}}h_2 = h_2$.  Then
	\[ \p(h_1) = \p(\alpha)\p(h_2) = \p(\alpha^{(-1)})\p(h_2) = \p(\rs{\alpha^{(-1)}})\p(h_2) = \p(\rs{\alpha^{(-1)}} h_2) = \p(h_2), \]
so $\p^*$ is well-defined.  

For preservation of composition, suppose that we form the composite of $(v_1,h_1)$ and $(v_2,h_2)$:
\[ \xymatrix{ & & S \ar[dl]_{v_2^*} \ar[dr]^{h_1^*} & & \\ & S \ar[dl]_{v_1} \ar[dr]^{h_1} & & T \ar[dl]_{v_2} \ar[dr]^{h_2} & \\ X & & Y & & Z } \]
As in the previous lemma, we can choose $h_1^*$ to be the prone arrow over $\p(h_1)\p(v_2)$ so $\p(h_1^*) = \p(h_1)\p(v_2)$.  Then we have
	\[ \p(h_1^*h_2) = \p(h_1^*)\p(h_2) = \p(h_1)\p(v_2)\p(h_2) = \p(h_1)\p(\rs{v_2})\p(h_2) = \p(h_1)\p(h_2) \]
since $\rs{v_2} = \rs{h_2}$.  Thus
	\[ \p^*[(v_1,h_1) \circ (v_2,h_2)] = \p^*[(v_2^*v_1,h_1^*h_2) = \p(h_1^*h_2) = \p(h_1)\p(h_2) = \p^*(v_1,h_1) \p^*(v_2,h_2), \]
so composition is preserved.

Finally, $\p^*$ is a restriction functor since $\p(\hat{h}) = \p(h)$ by definition of $\hat{h}$.  
\end{proof}

The following lemma is useful when working with maps in $\E^*$.  

\begin{lemma}\label{lemma:basic_dual_results} Suppose that $\p\colon \E \to \B$ is a hyperfibration.  
\begin{enumerate}[(i)]
	\item For any $[(v,c)]\colon X \to Y$ in $E^*$ and any $c'\colon Y \to Z$ prone in $\E$, 
		\[ [(v,c)] \circ [(\rs{c'},c)] = [(\rs{cc'}v,cc')]. \]
	\item For any $v\colon Y \to X$ subvertical in $\E$ and any $[(v',c)]\colon Y \to Z$ in $\E^*$,
		\[ [(v,\rs{v})] \circ [(v',c)] = [(v'v,\rs{v'v}c)]. \]
	\item Any $[(v,c)]\colon X \to Y$ in $\E^*$ can be factored as:
		\[ (v,c) = (v,\rs{v}) \circ (\rs{c},c). \]
\end{enumerate}
\end{lemma}
\begin{proof}
For (i), by Lemma \ref{latent-pullbacks}.iii, $(v,c) \circ (\rs{f},f)$ is given by
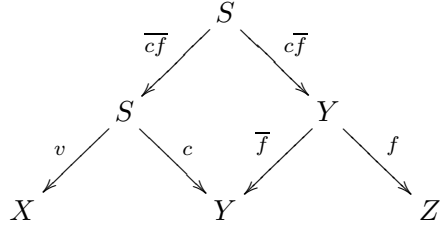
\[ \xymatrix{ & & S \ar[dl]_{\rs{cf}} \ar[dr]^{c\rs{f}} & & \\ & S \ar[dl]_{v} \ar[dr]^c & & Y \ar[dl]_{\rs{f}} \ar[dr]^{f} & \\ X & & Y & & Z } \]
Thus, the composite is $(v,c) \circ (\rs{f},f) = (\rs{cf}v,cf)$;  part (ii) follows similarly.  

For (iii), by (i), $(v,\rs{v}) \circ (\rs{c},c) = (\rs{\rs{v}c}v,\rs{v}c)$, which equals $(v,c)$ since $\rs{v} = \rs{c}$.
\end{proof}

\begin{lemma}\label{lemma:prone_e*}
If $\p\colon \E \to \B$ is a hyperfibration, and $f\colon X \to Y$ is prone in $\E$, then $(\rs{f},f)$ is prone in $\E^*$.
\end{lemma}
\begin{proof}
Suppose we are given 
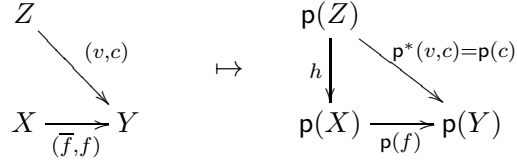
\[ \xymatrix{ Z  \ar[dr]^{(v,c)} & ~\ar@{}[drr]|{\textstyle\mapsto}&& \p(Z) \ar[dr]^{\p^*(v,c) = \p(c)} \ar[d]_{h}
\\ 
X \ar[r]_{(\rs{f},f)}  & Y &&   \p(X) \ar[r]_{\p(f)} & \p(Y)} \]
such that $\rs{h} = \rs{p(c)}$.  Then we claim that $(v,h')$ is the required unique fill-in, where $h'$ is the unique fill-in from the universal property of $f$ as a prone arrow in $E$:
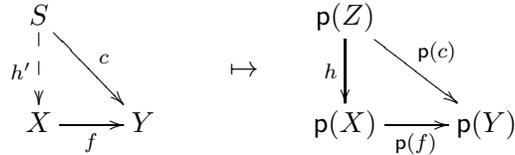
\[ \xymatrix{ S \ar@{-->}[d]_{h'} \ar[dr]^{c} & ~\ar@{}[drr]|{\textstyle\mapsto}&& \p(Z) \ar[dr]^{\p(c)} \ar[d]_{h} 
\\ 
X \ar[r]_{f}  & Y && \p(X) \ar[r]_{\p(f)} & \p(Y)} \]
First, $(v,h')$ is a well-defined arrow in $\E^*$: by Lemma \ref{lemma:left_factor_separated}, $h'$ is prone in $\E$, and by definition of $h'$ and $(v,c)$, $\rs{h'} = \rs{c} = \rs{v}$.  

Second, $(v,h')$ does make the triangle commute: using Lemma \ref{lemma:basic_dual_results}.i, 
	\[ (v,h') \circ (\rs{f},f) = (\rs{h'f}v,h'f) = (\rs{c}v,c) = (v,c) \]
since $\rs{c} = \rs{v}$.  

It remains to show that $(v,h')$ is unique with these properties.  Thus, suppose we have an arrow $(w\colon T \to Z, k\colon T \to X)$ in $\E^*$ such that 
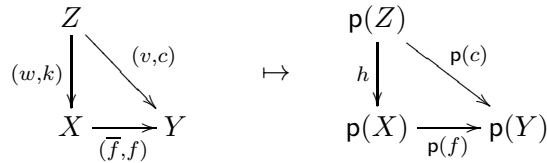
\[ \xymatrix{ Z  \ar[d]_{(w,k)}   \ar[dr]^{(v,c)} & ~\ar@{}[drr]|{\textstyle\mapsto}& & \p(Z) \ar[dr]^{\p(c)} \ar[d]_{h}
\\ 
X \ar[r]_{(\rs{f},f)}  & Y &&  \p(X) \ar[r]_{\p(f)} & \p(Y)} \]
Since the triangle commutes, using Lemma \ref{lemma:basic_dual_results}.i, $(\rs{kf}w,kf) = (v,c)$.  Thus, by definition of arrows in $\E^*$, there is some subvertical partial isomorphism $\alpha\colon S \to T$ such that 
	\[ \alpha kf = c, \ \alpha\rs{kf} w = v, \ \rs{\alpha} = \rs{c}, \mbox{ and } \rs{\alpha^{(-1)}} = \rs{kf}. \]
We claim that this same $\alpha$ witnesses that $(w,k) = (v,h')$ in $\E^*$; that is, we claim that 
	\[ \alpha k = h', \ \alpha w = v, \ \rs{\alpha} = \rs{h'}, \mbox{ and } \rs{\alpha^{(-1)}} = \rs{w}. \]

Indeed, for the first requirement, we have $\alpha kf = c = h'f$.  However, we also have 
	\[ \p(\alpha k) = \rs{\p(\alpha)}\p(k) = \rs{\p(c)}\p(k) = \rs{h}\p(k) = \rs{h}h = h = \p(h'), \]
so by proneness of $f$, $\alpha k = h'$.  
For the second requirement, 
	\[ v = \alpha\rs{kf}w = \rs{\alpha kf}\alpha w = \rs{c}\alpha w = \rs{\alpha} \alpha w = w. \]
For the third requirement, 
	\[ \rs{\alpha} = \rs{c} = \rs{h'}. \]
For the fourth requirement, we use hyperconnectedness:
	\[ \rs{\p(\alpha^{(-1)})} = \rs{\p(k)\p(f)} = \rs{h \p(f)} = \rs{h} = \rs{\p(k)} = \rs{\p(w)}, \]
with the third equality since $h\p(f) = c$ is precise.  Thus, since $\p$ is a hyperconnection, $\rs{\alpha^{(-1)}} = \rs{w}$.

Thus $\alpha$ witnesses  $(w,k) = (v,h')$ in $\E^*$, and so $(\rs{f},f)$ is prone, as required.
\end{proof}

\begin{theorem}\label{thm:dual_hyperfibration}
If $\p\colon \E \to \B$ is a hyperfibration then so is $\p^*\colon \E^* \to \B$.  
\end{theorem}
\begin{proof}
Given an $f\colon A \to B$ in $\B$ and a $Y$ over $\B$ let $f'\colon X \to Y$ be its $\p$-prone lift.  Then the previous lemma tells us that $(\rs{f'},f')$ is $\p^*$-prone, and so $\p^*$ is a latent fibration.  Moreover, it is immediately admissible as this lift is a restriction idempotent if $f$ is.  Finally, it is easy to check that $\p^*$ is a hyperconnection since $\p$ is, so $\p^*$ is a hyperfibration.  
\end{proof}

Our goal in the rest of this section is to further understand this latent fibration.  We begin by characterizing the subvertical maps:

\begin{lemma}\label{lemma:verticals_in_e*} Let $\p\colon \E \to \B$ be a hyperfibration.  A map $[(v,c)]\colon X \to Y$ is subvertical in $\E^*$ if and only if $c$ is a partial isomorphism, if and only if $[(v,c)]$ has a unique representative of the form
		\[ \xymatrix{ & Y \ar[dl]_{v'} \ar[dr]^{\rs{v'}} & \\ X & & Y } \]
\end{lemma}
\begin{proof}
If $(v,c)$ is subvertical, then $c$ is prone and subvertical, so by Proposition \ref{prop:hyper_consequences}.ii, $c$ is a partial isomorphism.  Then define $v' := c^{(-1)}v$.  Then the pair $(v',\rs{v'})$ represents the same map as $(v,c)$ via the subvertical partial isomorphism $c$; note that since $\rs{v} = \rs{c}$,
		\[ \rs{v'} = \rs{c^{(-1)}\rs{v}} = \rs{c^{(-1)}c} = \rs{c^{(-1)}}. \]
	Moreover, such a representative is unique.  If we have any other equivalent span 
		\[ \xymatrix{ & Y \ar[dl]_{w} \ar[dr]^{\rs{w}} & \\ X & & Y } \]
	 then there is some subvertical partial isomorphism $\alpha\colon Y \to Y$ such that (in particular) $\alpha\rs{v'} = \rs{w}$, $\rs{\alpha} = \rs{v'}$, and $\alpha v = w$.  But then $\alpha \leq \alpha \rs{v'} = \rs{w}$, so $\alpha$ is a restriction idempotent, and so we have
	 	\[ w = \alpha v' = \rs{\alpha} v' = \rs{v'} v' = v' \]
     If we have a map $[(v,c)]$ with such a representative, then it is clearly subvertical.

\end{proof}

Before we look at prone maps, it will be helpful to understand certain partial isomorphisms in $\E^*$.
\begin{lemma}\label{lemma:partial_isos_e*}
Let $\p\colon \E \to \B$ be a hyperfibration.  If $(v,\rs{v})\colon X \to S$ represents a partial isomorphism in $\E^*$, then $v$ is a partial isomorphism in $\E$.
\end{lemma}
\begin{proof}
Since $(v,\rs{v})$ is a subvertical partial isomorphism, its partial inverse is also subvertical (and over the same map, $\p(\rs{v})$).  Then by the previous lemma, its partial inverse can be taken to be of the form
	\[ \xymatrix{ & X \ar[dl]_{v'} \ar[dr]^{\rs{v'}} & \\ S & & X} \]
We want to show that $v'$ is the partial inverse of $v$.  Since $(v',\rs{v'})$ represents a partial inverse of $(v,\rs{v})$, we have that their composite is equivalent to $\rs{(v,\rs{v})} = (\hat{v},\hat{v})$; by Lemma \ref{lemma:basic_dual_results}.ii, we have that the composite is represented by 	
	\[ (v'v,\rs{v'v}\rs{v'}) = (v'v,\rs{v'v}). \]
Thus, there is a subvertical partial isomorphism $\alpha\colon X \to X$ so that in particular
	\[ \alpha \rs{v'v} = \hat{v}, \alpha v'v = \hat{v}, \rs{\alpha} = \rs{\hat{v}} = \hat{v}. \]
Then $\alpha \leq \alpha \rs{v'v} = \hat{v}$, so $\alpha$ is a restriction idempotent, and so $\alpha = \rs{\alpha} = \hat{v}$.  Morever, as noted above $\p(\rs{v'}) = \p(v) = \p(\hat{v})$, so since $\p$ is a hyperconnection, $\hat{v} = \rs{v'}$.  So then
	\[ \rs{v'} = \hat{v} = \alpha v'v = \rs{v'}v'v = v'v. \]
The other equality $(vv' = \rs{v})$ follows similarly, and so $v$ is indeed a partial isomorphism.   
\end{proof}

We now characterize prone maps in $\E^*$:

\begin{lemma}\label{lemma:prones_in_e*} Let $\p\colon \E \to \B$ be a hyperfibration.  A map $[(v,c)]\colon X \to Y$ is prone in $\E^*$ if and only if $v$ is a partial isomorphism, if and only if $[(v,c)]$ has a unique representative of the form
		\[ \xymatrix{ & X \ar[dl]_{\rs{c'}} \ar[dr]^{c'} & \\ X & & Y } \]
\end{lemma}
\begin{proof}
By Lemma \ref{lemma:basic_dual_results}.iii, 
	\[ [(v,c)] = [(v,\rs{v})] \circ [(\rs{c},c)]. \]
However, by Lemma \ref{lemma:prone_e*}, $[(\rs{c},c)]$ is prone, by assumption $[(v,c)]$ is prone, and this factorization is precise since $[(v,c)]$ and  $[(v,\rs{v})]$ both have restriction $[(\hat{v},\hat{v})]$.  Thus by Lemma \ref{lemma:left_factor_separated}, $[(v,\rs{v})]$ is also prone.  Thus by Proposition \ref{prop:hyper_consequences}.i, $[(v,\rs{v})]$ is a partial isomorphism in $\E^*$, and so by Lemma \ref{lemma:partial_isos_e*}, $v$ is a partial isomorphism.

Given that $v$ is a partial isomorphism, define $c'\colon X \to Y$ by $c' := v^{(-1)}c$.  Then it is straightforward to check that $v\colon S \to X$ witnesses that $(\rs{c'},c')$ represents the same map as $[(v,c)]$.  That this representative is unique is similar to the proof of a similar statement in Lemma \ref{lemma:verticals_in_e*}.

The result that such maps are prone in $\E^*$ was Lemma \ref{lemma:prone_e*}. 
\end{proof}

By Proposition \ref{prop:factorization}, any map $f\colon X \to Y$ in the total category of a latent fibration $\p\colon \E \to \B$ has a subvertical/prone factorization, and this factorization is unique up to vertical partial isomorphism. We will need a slight strengthening of this result in the case when $\p$ is a latent hyperfibration:
\begin{lemma}\label{lemma:factorization_in_hyper}
Suppose $\p\colon \E \to \B$ is a latent hyperfibration.  Then for any $f\colon X \to Y$ in $\E$, there is a factorization 
	\[ X \to^{v} S \to^{c} Y \]
where $\rs{c} = \hat{v}$.  Such a factorization is unique up to subvertical partial isomorphism.
\end{lemma}
\begin{proof}
Given such an $f$, let $c\colon S \to Y$ be a prone lift of $\p(f)$, and let $v\colon X \to S$ be the induced map over $\rs{\p(f)}$:
\[ \xymatrix{ X  \ar@{-->}[d]_{v}   \ar[dr]^{f} & ~ \ar@{}[drr]|{\textstyle\mapsto}&&\p(X) \ar[dr]^{\p(f)} \ar[d]_{\rs{p(f)}} 
\\ 
S \ar[r]_{c}  & Y && \p(X) \ar[r]_{\p(f)} & \p(Y)} \]
Then we have
	\[ \p(\rs{c}) = \rs{\p(c)} = \rs{p(f)} = \p(\rs{v}) = \p(\hat{v}) \]
(with the last by definition of $\hat{v}$) and so since $\p$ is a hyperconnection, $\rs{c} = \hat{v}$.  Uniqueness of such a factorization follows from the uniqueness property of prone arrows.  
\end{proof}

With the above results in hand, we can now prove a result that can be used to prove that $(\E^*)^*$ is isomorphic to $\E$, but is useful in its own right.  
\begin{proposition}\label{prop:dual_functor_correspondence}
Suppose that $\p\colon \X \to \B$ and $\q\colon \Y \to \B$ are latent hyperfibrations.  Then there is a bijection between morphisms of latent fibrations (i.e., restriction functors which preserve the projection and prone arrows) 
	\[ \{F\colon \X \to \Y^*\} \cong \{G\colon \X^* \to \Y\}. \]
\end{proposition}
\begin{proof}
We will sketch the construction and leave most of the details to the reader.  Given a morphism of latent fibrations $F\colon \X \to \Y^*$, define $\hat{F}\colon \X^* \to Y$ as follows.  On objects, $\hat{F}$ is defined as $F$.  On a representative arrow
	\[ \xymatrix{ & S \ar[dl]_v \ar[dr]^c & \\ X & & Y} \]
since $F$ is a restriction functor, $F(v)$ is a subvertical in $\Y^*$, and so by Lemma \ref{lemma:verticals_in_e*} corresponds to a unique subvertical arrow which we denote as $F_v(v)\colon FX \to FS$.  Similarly since $F$ preserves prone arrows, $F(c)$ is prone in $\Y^*$, and so by Lemma \ref{lemma:prones_in_e*}, corresponds to a unique prone arrow $F_c(c)\colon FS \to FY$.  We then define $\hat{F}([(v,c)])$ to be the composite
	\[ FX \to^{F_v(v)} FS \to^{F_c(c)} FY. \]
This is well-defined since if $(v,c)$ is equivalent to some other $(v',c')$, then there is a partial isomorphism taking $v$ to $v'$ and $c$ to $c'$; thus $F(v)$ and $F(v')$ go to the same map under $F$, and so $F_v(v) = F_v(v')$ and similarly $F_c(c) = F_c(c')$.  

Conversely, suppose we have a morphism of latent fibrations $G\colon \X^* \to \Y$.  We define $\hat{G}\colon \X \to \Y^*$ on objects as $G$.  For an arrow $f\colon X \to Y$ in $\X$, by Lemma \ref{lemma:factorization_in_hyper} there is a factorization
	\[ X \to^{v} S \to^{c} Y \]
such that $\hat{v} = \rs{c}$ (and this factorization is unique up to subvertical partial isomorphism).  Then 
	\[ \xymatrix{& X \ar[dl]_v \ar[dr]^{\rs{v}} & \\ S & & X \ar@{}[u]_{~~~~} } \xymatrix{ & S \ar[dl]_{\rs{c}} \ar[dr]^c & \\ S & & Y} \]
are respectively subvertical and prone in $\X^*$.  Thus we define $\hat{G}(f)$ to be the equivalence class of the span
	\[ \xymatrix{& GS \ar[dl]_{G([(v,\rs{v})])} \ar[dr]^{G([(\rs{c},c)])} & \\ GX & & GY } \]
Note that this has the required restriction property since
	\[ \rs{[(v,\rs{v})]} = \hat{v} = \rs{c} = \hat{\rs{c}} = \rs{[(\rs{c},c)]}. \]
Also note that $\hat{G}$ is independent of the choice of factorization since such factorizations are unique up to subvertical partial isomorphism.
\end{proof}

\begin{corollary}
If $\p\colon \E \to \B$ is a latent hyperfibration, then $\E^{**}$ is isomorphic to $\E$. 
\end{corollary}
\begin{proof}
Using Proposition \ref{prop:dual_functor_correspondence}, from the identity functor $\E^* \to^{1} \E^*$ we get a functor $\E \to^{\eta} \E^{**}$ and a functor $\E^{**} \to^{\varepsilon} \E$; using the previous results it is straightforward to check that these are inverse to one another.
\end{proof}


\section{Conclusion}

Aside from developing the applications of latent fibrations to, for example, restriction differential categories and partial lenses, there is much theoretical work to be done to bring the theory of latent fibrations to a state of maturity.  We certainly hope to return to this task in future papers and we also hope that others will join us in this development.  While there are some aspects of this theory for which we already have a partial understanding there are other aspects which are completely open.   For example, traditionally fibrations are used to model logical features such as existential and universal quantification and as a semantics of type theories, however, at this stage, we have little idea of what a type theory 
corresponding to a latent fibration might look like.

Turning to aspects for which we do have some understanding we have shown above that every latent fibration has associated re-indexing (or substitution)  restriction \emph{semi\/}functors.  Nonetheless, we have not here provided a full equivalence of such a collection of such re-indexing semifunctors to a latent fibration.  In other words, we have not provided the analogue of the normal Grothendieck construction for latent fibrations.  However, at this stage, we do have an understanding of this construction which we hope to publish separately.  Notably the construction requires more structure than for the  normal case but, as in the normal case, has the potential of providing a significantly different and important perspective.  The additional structure is actually a ``categorification'' of the structure used to define restriction presheaves  introduced in \cite{journal:lin-garner-cocomp-rcat}.  

A further, rather separate issue, concerns the behaviour of factorization systems over latent fibrations.  It is well-known that given a normal fibration, $\p\colon \E \to \B$, any orthogonal factorization system in $\B$ can be lifted to one in $\E$.  A similar result can be proven for latent fibrations.  However, some care is needed at the outset of this project, for defining what is meant by a factorization system for a restriction categories has some, perhaps, unexpected features.  This also we hope to return to in future work.

At this stage in the development of the theory of latent fibrations, what we know still seems dwarfed by what we still do not know.  Perhaps the major open issues concern the understanding of how logical features and, indeed, restrictional features (joins, meets, discreteness, latent limits, classifiers, etc.) interact with the structure of latent fibrations and how these, in turn, translate into type theoretic features.

\bibliographystyle{plain}
\bibliography{references}

\appendix

\section{The original definition of prone}
\label{Appendix-A}

The notion of a latent fibration was first introduced in \cite{msc:nester-calgary} in order to describe the structures encountered while exploring realizability at the level of restriction categories.   In this paper we have used a more convenient notion of ``prone'' which is a considerable simplification of the original definition.   The purpose of this appendix is to show that the original definition of prone from \cite{msc:nester-calgary} and the one used in this paper are indeed equivalent.

Here is the original definition of prone:

\begin{definition} \label{olddefn}
Let ${\sf p} \colon \E \to \B$ be a restriction functor, with $\E$ and $\B$ restriction categories.
An arrow $f \colon X' \to X$ in $\E$ is {\bf ${\sf p}$-prone} in case whenever we have $g \colon Y \to X$ in $\E$ and $h \colon \p(Y) \to {\sf p}(X')$ in $\X$ such that $h{\sf p}(f) \geq {\sf p}(g)$,
\[
\text{in $\E$}\hspace{10pt}
\xymatrix{
Y \ar[rd]^g \ar@{.>}[d]_{\exists\widetilde{h}} & \\
X' \ar[r]_f \ar@{}[ur]|(0.33){\geq} & X
}
\hspace{30pt}
\text{in $\B$}\hspace{10pt}
\xymatrix{
{\sf p}(Y) \ar[rd]^{{\sf p}(g)} \ar[d]_h & \\
{\sf p}(X') \ar[r]_{{\sf p}(f)} \ar@{}[ur]|(0.33){\geq} & {\sf p}(X)
}
\]
 there is a {\bf lifting} $\widetilde{h} \colon Y \to X'$ in $\X$ such that
\begin{enumerate}[(a)]
\item $\widetilde{h}$ is a {\bf candidate lifting}:  that is, $\widetilde{h}f \geq g$ and ${\sf p}(\widetilde{h}) \leq h$.
\item $\widetilde{h}$ is the {\bf smallest} candidate lifting: that is, for any other candidate lifting $k \colon Y \to X'$ in $\X$, $\widetilde{h} \leq k$.
\end{enumerate}
\end{definition}

Our objective is to show that this notion of a prone map is equivalent to Definition \ref{newdefn}.i  which has been used in this paper.  A first step in this direction is to show that the lifting of $h$ is unique having $\widetilde{h}f \geq g$ minimal is precisely to ask for the apparently stronger requirement that $\widetilde{h}f = g$ and $\rst{\widetilde{h}}= \rst{g}$, in other words that the lifted triangle is precise:

\begin{lemma} (see \cite{msc:nester-calgary} Lemma 6.2) \label{ChadsLemma}
For a restriction functor ${\sf p}\colon \E \to \B$, if $f$ is a  ${\sf p}$-prone arrow (in the sense of Definition \ref{olddefn}) and $h{\sf p}(f) \geq {\sf p}(g)$ then $\widetilde{h}$ is the lifting of $h$
\[
\text{in $\E$}\hspace{10pt}
\xymatrix{A \ar[rd]^g \ar[d]_{\widetilde{h}} & \\
B \ar[r]_f \ar@{}[ur]|(0.33){\geq} & C
}
\hspace{30pt}
\text{in $\B$}\hspace{10pt}
\xymatrix{A \ar[rd]^{{\sf p}(g)} \ar[d]_h & \\
B \ar[r]_{{\sf p}(f)} \ar@{}[ur]|(0.33){\geq} & C
}
\]
if and only if $\widetilde{h}$ is the unique map such that $\rst{\widetilde{h}} = \rst{g}$ and $\widetilde{h}f = g$ and ${\sf p}(\widetilde{h}) \leq h$.
\end{lemma}

\begin{proof} 
\begin{description}
\item[($\Rightarrow$)]
Note that for the lifting $\widetilde{h}$ we have $\rst{\widetilde{h}} \leq \rst{\widetilde{h} f} \leq \rst{g}$ and yet $\rst{g}\widetilde{h} \leq \widetilde{h}$ is certainly 
a candidate lifting so, as $\widetilde{h}$ is the smallest such, $\rst{g}\widetilde{h} = \widetilde{h}$, so $\rst{g} \leq \rst{\widetilde{h}}$ and, thus, $\rst{\widetilde{h}} = \rst{g}$.
It then follows that $\widetilde{h}f \leq g$ as $\rst{\widetilde{h}f}g \leq \rst{\widetilde{h}}g = \rst{g} g = g$ and so $\widetilde{h} f = g$.
\item[($\Leftarrow$)]
Suppose that $\widetilde{h}$ is the unique map such that $\rst{\widetilde{h}} = \rst{g}$, $\widetilde{h}f=g$, and ${\sf p}(\widetilde{h}) \leq h$, then certainly  it is a candidate lifting.  Suppose also that $v$ is some other candidate lifting, so that $vf \geq g$ and ${\sf p}(v) \leq w$,  then $\rst{g}v$ has $\rst{\rst{g}v} \geq \rst{\rst{g}vf} \geq \rst{\rst{g} g}  = \rst{g} \geq \rst{\rst{g}v}$ so $\rst{\rst{g}v} = \rst{g}$, as $vf \geq g$ so $\rst{g}vf = g$, and ${\sf p}(\rst{g}v) = \rst{{\sf p}(g)}{\sf p}(v) \leq h$.  So $\rst{g}v$ satisfies the specification of $\widetilde{h}$ and so $\rst{g}v = \widetilde{h}$ and, consequently, $v \geq \widetilde{h}$ showing the minimality of $\widetilde{h}$.
\end{description}
\end{proof}

Of course once one realizes the lifted triangle is always going to be precise then one can reduce the lax triangle in the base to a precise triangle with $h' =  \rst{\p(g)} h$ (by Lemma \ref{Jaws}.v).  This means:

\begin{proposition}
Definition \ref{olddefn} of  being prone is equivalent to Definition \ref{newdefn}. 

\medskip
More precisely, if ${\sf p}\colon \E \to \B$ is a restriction functor, $f\colon B \to C$ is ${\sf p}$-prone -- in the sense of Definition \ref{olddefn} -- in $\E$ if and only if every 
precise triangle $h\p(f) = {\sf p}(g)$ with left factor $h$ in $\B$ has a unique $\widetilde{h}$ with ${\sf p}(\widetilde{h})=h$ making $\widetilde{h}f = g$ a precise triangle in $\E$:
\[ \xymatrix{A \ar@{..>}[d]_{\widetilde{h}} \ar[dr]^g & \ar@{}[drr]|{\textstyle\stackrel{\p}{\longmapsto}}&&&{\sf p}(A) \ar[d]_{{\sf p}(\widetilde{h}) = h} \ar[dr]^{{\sf p}(g)}
\\ 
B \ar[r]_f & C&&& {\sf p}(B) \ar[r]_{{\sf p}(f)} & {\sf p}(C)} \]
\end{proposition}

\begin{proof}~
\begin{description}
\item{($\Rightarrow$)} If $f$ is prone in the sense of Definition \ref{olddefn}, and $h {\sf p}(f) = {\sf p}(g)$ is a precise triangle with left factor $h$ then, by Lemma 
\ref{ChadsLemma}, $\widetilde{h}f =g$ is a precise triangle with ${\sf p}(\widetilde{h}) \leq h$.  However, 
$\rst{{\sf p}(\widetilde{k})} = {\sf p}(\rst{\widetilde{h})} = {\sf p}(\rst{g}) = \rst{{\sf p}(g)} = \rst{h}$ so ${\sf p}(\widetilde{h}) = h$.
\item{($\Leftarrow$)} Let $h \p(f) \geq \p(g)$ be a lax triangle with lax left factor $h$, then, by Lemma \ref{Jaws}.v, $(\rst{\p(g)}h)\p(f) = \p(g)$ is a precise triangle and so there is, by assumption, a unique left factor making a precise triangle $(\widetilde{\rst{{\sf p}(g)}h})f = g$ with ${\sf p}(\widetilde{\rst{{\sf p}(g)}h}) = \rst{{\sf p}(g)}h$.  
This $\widetilde{\rst{{\sf p}(g)}h}$ we claim satisfies all the requirements of being a lifting of $h$.  It is a candidate lifting as 
$\widetilde{\rst{\p(g)}h}f=g\geq g$ and $\p(\widetilde{\rst{\p(g)}h}) = \rst{\p(g)}h \leq h$.  

To show that $\widetilde{\rst{\p(g)}h}$ is the {\em least\/} candidate lifting, consider another candidate lifting, $k$, with $kf \geq g$ and $\p(k) \leq h$ then $kf \geq g$ then $(\rst{g}k)f = g$ is precise as is $(\rst{\p(g)}\p(k))\p(f)= \p(g)$.  Now $\p(\rst{g}k) = \p(\rst{g}\p(k) \leq \p(\rst{g})h$ as $\p(k) \leq h$, however, this inequality is an equality as $\rst{\p(\rst{g}k)} = \rst{\p(g)} = \rst{\rst{\p(g)}h}$ so that by the uniqueness of lifting (under Definition \ref{newdefn})    we have $\rst{g}k = \widetilde{\rst{\p(g)}h}$ so that  $\widetilde{\rst{\p(g)}h} = \rst{g}k  \leq k$ as required.
\end{description}
\end{proof}

\section{Latent fibrations as fibrations in the 2-category \texorpdfstring{${\sf SRest}$}{SRest}}
\label{Appendix-B}

In this appendix we show that, in ${\sf SRest}$, the 2-category of restriction categories, restriction semifunctors, and transformations, fibrations in the 2-categorical sense are precisely latent fibrations:

\begin{theorem}
Latent fibrations are precisely fibrations in the 2-category ${\sf SRest}$.
\end{theorem}

Recall that Street \cite{street_fibration} provided a general definition of a fibration in a 2-category.  The following is an expanded version of Street's definition from \cite{Loregian_2019}:

\begin{definition} A 1-cell $p\colon E \to B$ in a 2-category is a \textbf{fibration} if every 2-cell
	\[ \xymatrix{X \ar[r]^{e} \ar[dr]_b^{\Rightarrow \beta} & E \ar[d]^{p} \\ & B } \]
has a $p$-Cartesian lift $\alpha\colon e' \Rightarrow e$ so that $p \alpha = \beta$.  A 2-cell
	\[ \xymatrix{X \ar@/^0.5pc/[rr]^{e'}_{\Downarrow \alpha } \ar@/_0.5pc/[rr]_{e} & & E } \]
is \textbf{$p$-Cartesian} if for all $F\colon Y \to X$ and all 2-cells
\[ \xymatrix{Y \ar[rr]^{e''} \ar@/_0.6pc/@{{}{ }{}}[rr]_{\Downarrow \xi} \ar[dr]_F & & E \\ &  X \ar[ur]_e & } ~~~~~~~ 
   \xymatrix{Y \ar[rr]^{e''} \ar[d]_F \ar@{}[drr]|{\Downarrow \gamma} & & E \ar[d]^{p} \\ X \ar[r]_{e'} & E \ar[r]_p & B} \]
such that $p \xi = p \alpha F \cdot \gamma$, there is a unique 2-cell $\zeta\colon e'' \Rightarrow e'F$ such that $\xi = \alpha F \cdot \zeta$ and $p \zeta = \gamma$.
\end{definition}

We shall prove the theorem using two propositions corresponding to the two implications required:

\begin{proposition}
If $\p$ is a fibration in the 2-category ${\sf SRest}$, then $\p\colon \E \to \B$ is a latent fibration.
\end{proposition}
\begin{proof}
Let $b \to^{f} \p(e)$ be an arrow in $\B$ such that $f\p(1_e) = f$.  Define:
\begin{itemize}
	\item $\X := \{\bullet\}$ (the terminal category)
	\item $e(\bullet) := e$, $e(1_{\bullet}) := 1_e$
	\item $b(\bullet) := b$, $b(1_{\bullet}) := \rs{f}$
	\item $\beta_{\bullet} := b \to^{f} \p(e)$
\end{itemize}
Note that naturality of a 2-cell which comes from a terminal category is not automatic as semi-functors need not preserve identities!  However, $\beta$ is natural in this case as
	\[ b(\idb)\beta_{\bu} = \rs{f}f = f \mbox{ \ while \ } \beta_{\bu} \p(e(\idb)) = f \p(1_e) = f\]
by assumption on $f$.  Moreover, $\beta$ is a transformation of semifunctors since 
	\[ \rs{\beta_{\bu}} = \rs{f} = b(\idb). \]
	
So by the 2-fibration property, there exists a 2-cell $\alpha\colon e' \Rightarrow e$ such that $\p(\alpha) = \beta$ which is $\p$-Cartesian (in the 2-categorical sense).  Explicitly, $\alpha$ consists of a single arrow $\alpha(\idb) = f'\colon e' \to e$ with $\rs{f'} = e'(\idb)$, $\p(e') = b$, and $\p(f') = f$.

Now suppose we have 
	\[ \xymatrix{e'' \ar[dr]^{g} &\ar@{}[drr]|{\textstyle\mapsto}&&  \p(e'') \ar[d]_h \ar[dr]^{\p(g)}&
\\ 
e' \ar[r]_{f'} & e && b \ar[r]_{f} & p(e)} \]
with $\rs{h} = \rs{\p(g)}$.  From this, define:
\begin{itemize}
	\item $\Y := \X = {\bullet}$ and $F$ the identity functor,
	\item $e''(\bu) := e'', e''(\idb) := \rs{g}$,
	\item $\xi\colon e'' \to e$ by $\xi_{\bu} := g$,
	\item $\gamma\colon \p(e'') \to \p(e')$ by $\gamma_{\bu} := h$.
\end{itemize}
$\xi$ is natural as $\rs{g}g = g1$, and is a semifunctor transformation since $\rs{g}  = e''(\idb)$.  $\gamma$ is natural as the square translates to requiring that
	\[ \rs{\p(g)}h = h\rs{f} \]
which holds by preciseness of the triangle $hf = \p(g)$, and $\gamma$ is a semifunctor transformation since
	\[ \rs{\gamma_{\bu}} = \rs{h} = \rs{\p(g)} = \p(\rs{g}) = \p(e''(\idb)). \]

So by $\p$-Cartesian-ness of $\alpha$, there is a unique 2-cell $\zeta\colon e'' \to eF$ so that $\xi = \alpha F \cdot \zeta$ and $\p \zeta = \gamma$.  That is, $\zeta$ consists of a single arrow $\zeta_{\bu}\colon e'' \to e'$ which fills in the triangle above and sits over $h$.  Moreover, $\alpha$ being a semifunctor transformation means that $\rs{\zeta_{\bu}} = e''(1_{\bu}) = \rs{g}$, so it \emph{precisely} fills in the top triangle.

\end{proof}

But we can also go back:

\begin{proposition}
Suppose $\p\colon \E \to \B$ is a latent fibration.  Then $\p$ is a fibration in the 2-category ${\sf SRest}$.
\end{proposition}
\begin{proof}
Let 
	\[ \xymatrix{\X \ar[r]^{e} \ar[dr]_b^{\Rightarrow \beta} & \E \ar[d]^{\p} \\ & \B } \]
be a 2-cell.  In particular, we have for each $X \in \X$ an arrow 
	\[ b(x) \to^{\beta_X} \p(e(x)) \mbox{ in $\B$ such that } \rs{\beta_x} = b(1_X). \]
Moreover, naturality of $\beta_X$ with respect to $1_X$ gives
	\[ b(1_X)\beta_X = \beta_X\p(e(1_X)), \]
Then we have
	\[ \beta_X = \rs{\beta_X}\beta_X = b(1_X)\beta_X = \beta_X\p(e(1_X)) \leq \beta_X \p(1_{e(X)}). \]
But since $\p(1_{e(X)})$ is a restriction idempotent, we always have $\beta_X \p(1_{e(X)}) \leq \beta_X$, so actually  	
	\[ \beta_X \p(e(1_X)) = \beta_X = \beta_X \p(1_{e(X)}). \]
	
Thus we have the required property, and so $b(x) \to^{\beta_X} \p(e(x))$ has a prone lift
	\[ e'(X) \to^{\beta_X^*} e(X). \]
Define $\alpha_X := \beta_X^* e(1_X)$.  Note that this is over $\beta_X$ as
	\[ \p(\alpha_X) = \p(\beta_X^* e(1_X)) = \beta_X \p(e(1_X)) = \beta_X \]
as in the calculation above.  

We need to define $e'\colon \X \to \E$ as a functor, so we need to give its action on arrows.  Given $f\colon X' \to X$, define $e'(f)$ as the unique fill-in
\[ \xymatrix{ e'(X') \ar[r]^{\beta^*_{X'}} \ar@{-->}[d]_{e'(f)} & e(X') \ar[d]^{e(f)} \ar@{}[drr]|{\textstyle\mapsto} && b(X') \ar[r]^{\beta_{X'}} \ar[d]_{b(f)} & \p(e(X')) \ar[d]^{\p(e(f))}
\\ 
e'(X) \ar[r]_{\beta^*_{X}} & e(X) && b(X) \ar[r]_{\beta_X} & \p(e(X)) } \]
Note that the square on the right (in $\B$) commutes since $\beta$ is natural, and is precise since
	\[ \rs{\beta_{X'}\p(e(f))} = \rs{b(f)\beta_X} = \rs{b(f)\rs{\beta_X}} = \rs{b(f)b(1_X)} = \rs{b(f)}. \]
Thus we do get an $e'(f)$ as above, and by definition it has the precise property that $\rs{e'(f)} = \rs{\beta^*_{X'}e(f)}$.  

We need to show that $e'$ is a semifunctor and $\alpha$ is a semifunctor transformation.  For $e'$ preserving restriction, note that $e'(\rs{f})$ is defined as the unique fill-in
\[ \xymatrix{ e'(X') \ar[r]^{\beta^*_{X'}} \ar@{-->}[d]_{e'(\rs{f})} & e(X') \ar[d]^{e(\rs{f})} \ar@{}[drr]|{\textstyle\mapsto}&& b(X') \ar[r]^{\beta_{X'}} \ar[d]_{b(\rs{f})} & \p(e(X')) \ar[d]^{\p(e(\rs{f}))}
\\ 
e'(X') \ar[r]_{\beta^*_{X'}} & e(X') && b(X') \ar[r]_{\beta_{X'}} & \p(e(X')) } \]
   
It suffices to show $\rs{e'(f)}$ satisfies the same fill-in property as $e'(\rs{f})$.  Indeed,
	\[ \beta^*_{X'} e(\rs{f}) = \beta^*_{X'} \rs{e(f)} = \rs{\beta^*_{X'} e(f)} \beta^*_{X'} = \rs{\beta^*_{X'} e(f) e(1_X)} \beta^*_{X'} = \]
	\[ \rs{e'(f)\beta^*_X e(1_X)} \beta^*_{X'} = \rs{e'(f)e'(1_X)} \beta^*_{X'} = \rs{e'(f)}\beta^*_{X'}, \]
where we have that by preciseness of the fill-in, $\rs{e'(1_X)} = \rs{\beta_X^*e(1_X)}$; moreover, this fill-in is itself precise as
	\[ \rs{e'(f)} = \rs{\beta^*_{X'} e(f)} = \rs{\beta^*_{X'}e(\rs{f})}, \]
and this fill-in is over $b(\rs{f})$ as 
	\[ \p(\rs{e'(f)}) = \rs{\p(e'(f))} = \rs{b(f)} = b(\rs{f}). \]
Thus by uniqueness of fill-in, $e'(\rs{f}) = \rs{e'(f)}$.

For preservation of composites, given $X'' \to^{g} X' \to^{f} X$, $e'(gf)$ is defined as the unique fill-in
\[ \xymatrix{ e'(X'') \ar[r]^{\beta^*_{X''}} \ar@{-->}[d]_{e'(gf)} & e(X'') \ar[d]^{e(gf)}\ar@{}[drr]|{\textstyle\mapsto} &&b(X'') \ar[r]^{\beta_{X''}} \ar[d]_{b(gf)} & \p(e(X'')) \ar[d]^{\p(e(gf))} 
\\ 
e'(X) \ar[r]_{\beta^*_{X}} & e(X) && b(X) \ar[r]_{\beta_X} & \p(e(X)) } \]
So it suffices to show $e'(g)e'(f)$ satisfies the same properties as $e'(gf)$.  Indeed, it makes the square commute as
	\[ e'(g)e'(f) \alpha_X = e'(g) \beta^*_{X'} e(f) = \beta^*_{X''} e(g)e(f) = \beta^*_{X''} e(gf), \]
it is over $b(gf)$ as
	\[ \p(e'(g)e'(f) = \p(e'(g)) \p(e'(f)) = b(g)b(f) = b(gf), \]
and it is precise since
	\[ \rs{e'(g)e'(f)} = \rs{e'(g)\rs{e'(f)}} = \rs{e'(g)\beta^*_{X'}e(f)} = \rs{\beta^*_{X''} e(g) e(f)} = \rs{\beta^*_{X''}e(gf)}. \]
Thus by uniqueness of fill-in, $e'(gf) = e'(g)e'(f)$.  

For the restriction requirement on the component $\alpha_X$, consider
	\[ \rs{\alpha_X} = \rs{\beta_X^*e(1_X)} = \rs{e'(1_X)} \]
by precise-ness of the fill-in, $e'(1_X)$.  (Note that this is why we defined $\alpha_X$ in this way; $\beta_X^*$ itself would not in general satisfy this condition).  For naturality of $\alpha_X$, we have
	\[ e'(f)\alpha_X = e'(f) \beta^*_X e(1_X) = \beta^*_{X'} e(f)e(1_X) = \beta^*_{X'} e(1_X') e(f)  = \alpha_{X'} e(f) \]
as required.

Now we need to show this $\alpha$ is $\p$-Cartesian (in the 2-categorical sense).  For this, we are given $F\colon \Y \to \X$ and 2-cells
\[ \xymatrix{\Y \ar[rr]^{e''} \ar@/_0.6pc/@{{}{ }{}}[rr]_{\Downarrow \xi} \ar[dr]_F & & \E \\ &  \X \ar[ur]_e & } ~~~~~~~ 
   \xymatrix{\Y \ar[rr]^{e''} \ar[d]_F \ar@{}[drr]|{\Downarrow \gamma} & & E \ar[d]^{\p} \\ X \ar[r]_{e'} & E \ar[r]_{\p} & \B} \]
such that $\p\xi = \gamma \p \alpha F$.  Note that since these are semifunctor transformations, $\rs{\xi_Y} = e''(1_Y)$ and $\rs{\gamma_Y} = \p(e''(1_Y))$.  

Then we define $\zeta_Y\colon e''(Y) \Rightarrow e'(FY)$ as the unique fill-in
\[ \xymatrix{e''(Y) \ar@{-->}[d]_{\zeta_Y} \ar[dr]^{\xi_Y} & \ar@{}[drr]|{\textstyle\mapsto}&& \p(e''(Y)) \ar[d]_{\gamma_Y} \ar[dr]^{\p(\xi_Y)} &
\\ 
e'(FY) \ar[r]_{\beta^*_{FY}} & e(FY)  \ar@{}[u]&&\p(e'(FY)) \ar[r]_{\beta_{FY}} & \p(e(FY))} \]
Note that the base triangle commutes since by assumption $\p(\xi_Y) = \gamma_Y \p(\alpha_{FY})$ and we proved above that $\p(\alpha_{FY}) = \beta_{FY}$. Moreover, the base triangle is precise since
	\[ \rs{\gamma_Y} = \p(e''(1_Y)) = \p(\rs{\xi_Y}) = \rs{\p(\xi_Y)}. \]

Thus, we do have such a map $\zeta_Y$; by definition it has $\rs{\zeta_Y} = \rs{\xi_Y}$.  Thus, it has the required restriction property to be a semifunctor transformation, as $\rs{\xi_Y} = e''(1_Y)$.   However, we still need it to be natural.

For naturality, given any $g\colon Y' \to Y$ in $\Y$, we need
	\[ \xymatrix{e''(Y') \ar[r]^{\zeta_{Y'}} \ar[d]_{e''(g)} & e'(FY') \ar[d]^{e'(Fg)} \\ e''(Y) \ar[r]_{\zeta_Y} & e'(FY)} \]
To prove this, we claim that both compositions satisfy the properties to be the unique fill-in
	\[ \xymatrix{ e''(Y') \ar[r]^{\xi_{Y'}} \ar@{-->}[d] & e(FY') \ar[d]^{e(Fg)} \\ e'(FY) \ar[r]^{\beta^*_{FY}} & e(FY) } \]
Indeed, putting $e''(g) \zeta_Y$ into the square we get
	\[ e''(g) \zeta_Y \beta^*_{FY} = e''(g) \xi_Y = \xi_{Y'} e(Fg) \]
while putting $\zeta_{Y'}e'(Fg)$ in we get
	\[ \zeta_{Y'}e'(Fg)\beta^*_{FY} = \zeta_{Y'} \beta^*_{FY'} e(Fg) = \xi_{Y'} e(Fg) \]
Applying $\p$ to both gives the same result since
	\[ \p(e''(g)\zeta_Y) = \p(e''(g)) \gamma_Y = \gamma_{Y'}\p(e'(Fg)) \]
while
	\[ \p(\zeta_{Y'} e'(Fg)) = \gamma_{Y'} \p (e'(Fg)). \]
Finally, the triangle is precise in one case since
	\[ \rs{e''(g)\zeta_Y} = \rs{e''(g)\rs{\zeta_Y}} = \rs{e''(g)\xi_Y} = \rs{\xi_{Y'} e(Fg)} \]
and in the other case since 
	\[ \rs{\zeta_{Y'} e'(Fg)} = \rs{\zeta_{Y'} \rs{e'(Fg)}} = \rs{\zeta_{Y'} \alpha_{FY'} e(Fg)} = \rs{\xi_{Y'} e(Fg)}. \]
Thus, the required 2-cell $\zeta$ exists.  

By definition, it satisfies the required property that $\p(\zeta) = \gamma$.  For the required property $\zeta \alpha_F = \xi$, consider
	\[ \zeta_Y \alpha_{FY} = \zeta_Y \beta^*_{FY} e(1_{FY}) = \xi_Y e(1_{FY}). \]
However, as with $\beta$, using naturality of $\xi$ with respect to $1_Y$, we have
	\[ \xi_Y = \rs{\xi_Y} = e''(1_Y)\xi_Y = \xi_Y e(F(1_Y)) \leq \xi_Y e(1_{FY}) \]
and we always have the opposite inequality since $e(1_{FY})$ is a restriction idempotent.  Thus $\xi_Y e(1_{FY}) = \xi_Y$, and so $\zeta \alpha_F = \xi$, as required.  

Finally, for uniqueness, suppose we have $\zeta'\colon e'' \to e'(F)$ (so $\rs{\zeta'_Y} = e''(1_Y)$) such that $\p(\zeta') = \gamma$ and $\zeta' \alpha_F = \xi$.  Then expanding this last equality, we have
\[ \xi_Y  = \zeta'_Y \beta^*_{FY} e(1_{FY}), \]
and so since $e(1_{FY})$ is a restriction idempotent,
\begin{eqnarray*}
\xi_Y & = & \rs{\zeta'_Y \beta^*_{FY} e(1_{FY})} \zeta'_Y \beta^*_{FY}  \\
& = & \rs{\xi_Y} \zeta'_Y \beta^*_{FY} \mbox{ (by above)} \\
& = & \zeta'_Y \beta^*_{FY} \mbox{ (since $\rs{\xi_Y} = e''(1_Y) = \rs{\zeta'_Y}$)} 
\end{eqnarray*}

Thus we have
	\[ \xymatrix{ e''(Y) \ar[dr]^{\xi_Y} \ar[d]_{\zeta'_Y} & \\ e'(FY) \ar[r]_{\beta^*_{FY}} & e(FY)} \]
commutes, $\zeta'_Y$ is over $\gamma_Y$, and the triangle is precise since $\rs{\zeta'_Y} = e''(1_Y) = \rs{\xi_Y}$.  Thus by uniqueness of fill-ins for prone arrows, $\zeta' = \zeta$ as required.  
\end{proof}

\end{document}